\documentclass[11]{article}

\usepackage{amsmath,amsfonts,latexsym,amsthm,amsbsy,amssymb}
\usepackage[numbers,sort&compress]{natbib}
\usepackage{graphicx}
\usepackage{color} 
\usepackage{caption}

\usepackage{algorithmic,algorithm,eqparbox,array}

\theoremstyle{remark}
\newtheorem{remark}{Remark} 
\newtheorem{lemma}{Lemma}
\selectfont
\title{Spatially dispersionless, unconditionally stable
  FC-AD solvers for variable-coefficient PDEs}

\author{O.~P. Bruno\footnote{Applied and Computational Mathematics,
    California Institute of Technology, Pasadena, CA 91125, USA --
    {\tt bruno@acm.caltech.edu}} and A.~Prieto\footnote{Departamento
    de Matem\'aticas, Univ. da Coru\~na, 15071
    A Coru~na, Spain.}}

\newcommand{\Nfourier}{F}  

\newcommand{\Ndeg}{m}

\newcommand{\Nd}{N_{d}}

\newcommand{\Ndelta}{N_{\Delta}} 

\newcommand{\Nodes}{N}

\newcommand{\Niter}{N_{\mathrm{iter}}}

\newcommand{\Nover}{N_{\mathrm{over}}}

\newcommand{\tildea}{a}

\newcommand{\tildeb}{c}

\begin{document}

\date{}
\maketitle
\begin{abstract}
  We present fast, spatially dispersionless and unconditionally stable
  high-order solvers for Partial Differential Equations (PDEs) with
  {\em variable coefficients} in general smooth domains. Our solvers,
  which are based on (i)~A certain ``Fourier continuation'' (FC)
  method for the resolution of the Gibbs phenomenon, together with
  (ii)~A new, preconditioned, FC-based solver for two-point boundary
  value problems (BVP) for variable-coefficient Ordinary Differential
  Equations, and (iii)~An Alternating Direction strategy, generalize
  significantly a class of FC-based solvers introduced recently for
  constant-coefficient PDEs. The present algorithms, which are
  applicable, with high-order accuracy, to variable-coefficient
  elliptic, parabolic and hyperbolic PDEs in general domains with
  smooth boundaries, are unconditionally stable, do not suffer from
  spatial numerical dispersion, and they run at FFT speeds. The accuracy,
  efficiency and overall capabilities of our methods are demonstrated
  by means of applications to challenging problems of diffusion and
  wave propagation in heterogeneous media.
\end{abstract}

\noindent {\it Keywords:} High-order methods, 
Alternating Direction Implicit schemes, 
numerical dispersion,
variable coefficient problems.

\maketitle

\vspace{-6pt}

\section{Introduction\label{intro}}

We present fast, spatially dispersionless and unconditionally stable
high-order solvers for Partial Differential Equations (PDEs) with {\em
  variable coefficients} in general smooth domains. Our algorithms,
which generalize significantly a class of
solvers~\cite{bruno10,lyon10} introduced recently for
constant-coefficient PDEs, are based on (i)~A certain ``Fourier
continuation'' (FC) method~\cite{bruno10} for the resolution of the
Gibbs phenomenon, together with (ii)~A new, preconditioned, FC-based
solver for two-point boundary value problems (BVP) for
variable-coefficient Ordinary Differential Equations (ODE), and
(iii)~The Alternating Direction Implicit (ADI)
methodology~\cite{peaceman55,douglas56}. One of the main enabling
elements in our overall FC-AD algorithm (Fourier-Continuation
Alternating-Directions) is the new solver~(ii) for two-point boundary
value problems with variable coefficients. Relying on preconditioners
that result from inversion of oversampled finite-difference matrices
together with a new methodology for enforcement of boundary conditions
and the iterative linear algebra solver GMRES, this algorithm produces
rapidly the solutions required for FC-AD time-stepping in the
variable-coefficient context. (A non-oversampled finite-difference
preconditioner related to but different from the one used here was
introduced in~\cite{Sun96} in the context of orthogonal collocation
methods for equations with constant coefficients.) The resulting PDE
solvers, which in practice are found to be unconditionally stable, do
not suffer from spatial numerical dispersion and they run at a
computational cost that grows as $\mathcal{O}(\Nodes\log_{2} \Nodes)$
with the size $\Nodes$ of the computational grid. A variety of
examples presented in this paper demonstrate the accuracy, speed and
overall capabilities of the proposed methodology.

The variable-coefficient FC-AD algorithms introduced in this paper
enjoy all the good qualities associated with the constant coefficient
solvers presented in~\cite{bruno10,lyon10}: the performance of the new
solvers compare favourably, in terms of accuracy and speed, with those
associated with previous approaches; a detailed discussion in these
regards can be found in the introductory sections
of~\cite{bruno10,lyon10}. In particular, in this paper we demonstrate
the high-order, essentially dispersionless character of the new FC-AD
solvers by means of solutions to parabolic and hyperbolic problems.
For example, the results presented in
Section~\ref{sec:dispersionless}, which include FC-AD fixed-accuracy
solutions at fixed numbers of points-per-wavelength for problems of
sizes ranging from one to one-hundred wavelengths in size, demonstrate
the spatial dispersionlessness of the FC-AD algorithm for hyperbolic
equations. Further, as shown in Table~\ref{tab:cputime}, for example,
the proposed FC-AD algorithm can evolve a solution characterized by
one-million spatial unknowns in a computing time of approximately
$1.5$ seconds per time step in a single-core run.

The remainder of this paper is organized as follows: after the
introduction in Section~\ref{pdes_consid} of the variable-coefficient
PDEs we consider, Section~\ref{sec:time-disc} details the ADI
approximations we employ. Section~\ref{sec:fc-gram} presents the FC
method, including, in Section~\ref{scaled_fc_gram}, a special version
of the FC algorithm we need to tackle variable-coefficient
differential equations.  Sections~\ref{sec:ODE-theory}
and~\ref{ode_impl} then describe our new FC-based solver for two-point
BVP, which include 1) Iterative FC-based solvers for periodic ordinary
differential equations in a certain ``continued'' periodic context,
and 2)~Techniques that enable enforcement of boundary conditions in
presence of either sharp or diffuse boundary layers.  The combined
FC/ADI scheme is described in Section~\ref{sec:full-disc}. The overall
properties of the resulting FC-AD PDE solver, finally, are
demonstrated in Section~\ref{sec:numerical} through a variety of
numerical results.

\section{Preliminaries}
\label{sec:prelim}
\subsection{Variable coefficients PDEs\label{pdes_consid}}
This paper presents FC-based solvers for linear equations containing
time-independent but spatially variable coefficients. While the
methods we present are applicable to any partial differential equation
for which an ADI splitting is available, for definiteness we focus on
the basic variable-coefficient parabolic and hyperbolic problems
\begin{equation}
\label{eq:heat}
\left\{
\begin{array}{ll}
\alpha\partial_{t}u-\mbox{div}(\beta\mbox{grad}u)=f&
\mbox{ in }\Omega\times(0,T),\\
u=g&\mbox{ on }\partial\Omega\times(0,T),\\
u=u_{0}&\mbox{ in }\Omega\times\{0\},
\end{array}\right.
\end{equation}
and
\begin{equation}
\label{eq:wave}
\left\{
\begin{array}{ll}
\alpha\partial_{t}^{2}u-\mbox{div}(\beta\mbox{grad}u)=f&
\mbox{ in }\Omega\times(0,T),\\
u=g&\mbox{ on }\partial\Omega\times(0,T),\\
u=u_{0}&\mbox{ in }\Omega\times\{0\},\\
\partial_{t}u=u_{1}&\mbox{ in }\Omega\times\{0\}
\end{array}
\right.
\end{equation}
in a bounded open set $\Omega$ with smooth boundary $\partial\Omega$
and within the time interval $(0,T)$.  Here $\alpha$, $\beta$, $u_0$, $u_1$, $f$
and $g$ are suficiently regular functions defined in $\Omega$;
additionally the coefficient functions $\alpha$ and $\beta$ are
assumed to satisfy the coercivity conditions
$$
\begin{array}{ll}
\alpha_{0}\le \alpha\le\alpha_{1}&\qquad \mbox{ in } \Omega,\\
\beta_{0}\le \beta\le\beta_{1}&\qquad \mbox{ in } \Omega
\end{array}
$$
for some positive constants $\alpha_0$, $\alpha_1$, $\beta_0$ and
$\beta_1$, while the initial and source functions are required to
verify the relevant compatibility conditions in
$\partial\Omega\times\{0\}$, namely, $u_0=g$ for the parabolic problem
\eqref{eq:heat} and $u_0=g$, $u_{1}=\partial_{t}g$ for the hyperbolic
problem \eqref{eq:wave}.


\subsection{Alternating Direction schemes}
\label{sec:time-disc}
This section presents the ADI splitting schemes we use for the
solution of the PDE problems~\eqref{eq:heat} and~\eqref{eq:wave}.
These splitting schemes result as adequate generalizations of the
Peaceman-Rachford scheme~\cite{peaceman55} for diffusion problems and
the non-centered scheme~\cite{lyon10} for the wave equation to the
present variable-coefficient context.  In what follows $\Delta t>0$
denotes the time step used in the computational time interval $[0,T]$;
it is assumed that $n_{\mathrm{max}}\Delta t=T$ for a certain positive
integer $n_{\mathrm{max}}$. Letting $t_{n}=n\Delta t$ for integer and
even fractional values of $n$ (e.g., $t_{n+\frac14} =
(n+\frac14)\Delta t$), we have, in particular,
$t_{n_{\mathrm{max}}}=T$.

\subsubsection{Diffusion equation}\label{sec:diffusion}
Given the initial values $u_{0}$, the right hand-sides
$f^{n+\frac14}(x,y)=f(x,y,t_{n+\frac14})$ and
$f^{n+\frac34}(x,y)=f(x,y,t_{n+\frac34})$, and the boundary values
$g^{n+1}(x,y)=g(x,y,t_{n+1})$, the exact solution
$\phi^{n}(x,y)=u(x,y,t_{n})$ for $n=1,\ldots,n_{\mathrm{max}}$ 
of the diffusion problem~(\ref{eq:heat}) satisfies the Crank-Nicolson
relation~\cite{peaceman55,bruno10}
\begin{equation}\label{eq:heat_crank}
\left\{
\begin{array}{ll}
  \alpha\displaystyle\frac{\phi^{n+1}-\phi^{n}}{\Delta t}
  -\mbox{div}\left(\beta\mbox{grad}
  \displaystyle\frac{\phi^{n+1}+\phi^{n}}{2}\right)=
  \displaystyle\frac{f^{n+\frac14}+f^{n+\frac34}}{2}
  +\mathcal{O}(\Delta t^2)\qquad&\mbox{ in }\Omega,\\
  \phi^{0}=u_{0}\qquad&\mbox{ in }\Omega,\\[0.2cm]
  \phi^{n+1}=g^{n+1}\qquad&\mbox{ in }\partial\Omega.
\end{array}
\right.
\end{equation}
Following~\cite{peaceman55}, an ADI scheme can be obtained from the
Crank-Nicolson iteration: denoting by $u^{n+1}$ the corresponding
approximation of $\phi^{n+1}$ (for integer values $n=1, 2,\dots$)
and using the intermediate quantity $u^{n+\frac12}$, the ADI scheme is
embodied in the equations
\begin{equation}\label{eq:heat-adi-1}
u^{n+\frac{1}{2}}
-\frac{\Delta t}{2\alpha}\partial_x(\beta\partial_x u^{n+\frac{1}{2}})
=u^{n}+\frac{\Delta t}{2\alpha}\partial_y(\beta\partial_y u^{n})
+\frac{\Delta t}{2\alpha}f^{n+\frac{1}{4}}\qquad\mbox{ in }\Omega,
\end{equation}
(with boundary condition $u^{n+\frac12}=g^{n+\frac12}$ on
$\partial\Omega$) and
\begin{equation}\label{eq:heat-adi-2}
u^{n+1}-\frac{\Delta t}{2\alpha}\partial_y(\beta\partial_yu^{n+1})
=u^{n+\frac{1}{2}}
+\frac{\Delta t}{2\alpha}\partial_x(\beta\partial_x u^{n+\frac{1}{2}})
+\frac{\Delta t}{2\alpha}f^{n+\frac{3}{4}}\qquad\mbox{ in }\Omega,
\end{equation}
(with boundary condition $u^{n+1}=g^{n+1}$ on $\partial\Omega$);
cf.~\cite{bruno10}.

An algorithm based on the ADI iteration
(\ref{eq:heat-adi-1})-(\ref{eq:heat-adi-2}) can be conveniently
obtained as a sequence of four operations involving two additional
auxiliary quantities $w^{n}$ and $w^{n+\frac12}$:
\begin{itemize}
\item[(D1)] Initialize $u^{0}$ and $w^{0}$ as
\begin{equation*}
\left\{
\begin{array}{ll}
u^{0}=u_{0}\qquad&\mbox{ in }\Omega,\\
w^{0}=\displaystyle\left(1
+\Delta t\frac{\partial_y\beta}{2\alpha}\partial_{y} 
+\Delta t\frac{\beta}{2\alpha}\partial_{y}^{2}\right)u^{0}
\qquad&\mbox{ in }\Omega.
\end{array}\right.
\end{equation*}
\end{itemize}
\noindent 
Then, for $n=0,1,\ldots,n_{\mathrm{max}}$,
\begin{itemize} 
\item[(D2)] Obtain $u^{n+\frac12}$ by solving the boundary value
  problem
\begin{equation}\label{D2}
\left\{
\begin{array}{ll}
\displaystyle\left(1
-\Delta t\frac{\partial_x\beta}{2\alpha}\partial_{x} 
-\Delta t\frac{\beta}{2\alpha}\partial_{x}^{2}\right)u^{n+\frac12}
=w^{n}+\frac{\Delta t}{2\alpha}f^{n+\frac{1}{4}}&
\qquad\mbox{ in }\Omega,\\[0.2cm]
u^{n+\frac12}=g^{n+\frac12}&
\qquad\mbox{ on }\Omega.
\end{array}\right.
\end{equation}
\noindent 
\item[(D3)] Update $w^{n+\frac12}$ according to
\begin{equation*}
w^{n+\frac12}=
  2u^{n+\frac12}-w^{n}-\frac{\Delta t}{2\alpha}f^{n+\frac{1}{4}}
\qquad\mbox{ in }\Omega,
\end{equation*}
and, finally,
\item[(D4)] Obtain $u^{n+1}$ by solving the boundary value problem
\begin{equation}\label{D4}
\left\{
\begin{array}{ll}
\displaystyle\left(1
-\Delta t\frac{\partial_y\beta}{2\alpha}\partial_{y} 
-\Delta t\frac{\beta}{2\alpha}\partial_{y}^{2}\right)u^{n+1}
=w^{n+\frac12}+\frac{\Delta t}{2\alpha}f^{n+\frac{3}{4}}&
\qquad\mbox{ in }\Omega ,\\
u^{n+1}=g^{n+1}&
\qquad\mbox{ on }\partial\Omega.
\end{array}\right.
\end{equation}
\end{itemize}
The ADI scheme thus requires solution of the one-dimensional boundary
value value problems (D2) and (D4) (see Section~\ref{sec:ODE-theory})
and solution updates (D1) and (D3). Each one of these operations
involves differential operators with respect to a single spatial
variable---as it behooves an ADI discretization~\cite{marchuk90}.  
\begin{remark}
  An implementation of this algorithm for the case in which the
  coefficients $\alpha$ and $\beta$ are constant was put forth
  in~\cite{bruno10}. In that reference it was noted that, for
  $n\in\mathbb{N}$, $u^{n+1}$ is a globally second order accurate
  approximation of the exact solution of the diffusion equation
  (\ref{eq:heat}): the error at any fixed time step $t=t^*$ is a
  quantity of order ${\mathcal O}(\Delta t)^2$. This is in spite of
  the approximation $u^{n+\frac12}=g^{n+\frac12}$, which is necessary
  to enable applicability to complex domains, and which induces a
  second-order local truncation error in step (D2); see~\cite[Remark
  3.2]{bruno10} for details. Numerical experiments we present in this
  paper reveal once again a second-order global error in the solutions
  resulting from the scheme above. As shown in
  reference~\cite{bruno10} and Section~\ref{diff_numer}, further,
  solutions of higher order of temporal accuracy can be extracted from
  the solutions produced by the ADI scheme above by means of the
  Richardson extrapolation methodology.
\end{remark}

\subsubsection{Wave equation}
\label{wave:time-discrete}
To derive our Alternating Direction scheme for the linear wave
problem~(\ref{eq:wave}), in turn, we follow~\cite{lyon10} and note
that the exact solution $\phi^{n}(x,y) = u(x,y,t_n)$ satisfies the
discrete relation
\begin{equation}\label{eq:ADI_2nd}
\left\{
\begin{array}{ll}
\alpha\displaystyle\frac{\phi^{n+1}-2\phi^{n}+\phi^{n-1}}{\Delta t^2}
 -\mbox{div}(\beta\mbox{grad}\phi^{n+1})
=f^{n+\frac12}+\mathcal{O}(\Delta t)
\qquad&\mbox{ in }\Omega,\\
\phi^{0}=u_{0}
\qquad&\mbox{ in }\Omega,\\[0.2cm]
\phi^{1}=u_{0}+\Delta t u_{1}+\mathcal{O}(\Delta t)
\qquad&\mbox{ in }\Omega,\\[0.2cm]
\phi^{n+1}=g^{n+1}\qquad&\mbox{ in }\partial\Omega,
\end{array}
\right.
\end{equation}
where, once again, $f^{n+\frac12}(x)=f(x,y,t_{n+\frac12})$ and
$g^{n+1}(x)=g(x,y,t_{n+1})$.  Following~\cite{lyon10} we split the
stiffness term and thus obtain the ADI time-stepping scheme
\begin{itemize}
\item[(W1)] Initialize $u^{0}$ and $u^{1}$ as
\begin{equation*}
\left\{
\begin{array}{ll}
u^{0}=u_{0}\qquad&\mbox{ in }\Omega,\\
u^{1}=u_{0}+\Delta t u_{1}\qquad&\mbox{ in }\Omega.
\end{array}\right.
\end{equation*}
\end{itemize}
Then, for $n=0,1,\ldots,n_{\mathrm{max}}$,
\begin{itemize} 
\item[(W2)] Obtain $w^{n+\frac12}$ by solving 
the boundary value problem,
\begin{equation}\label{W2}
\left\{
\begin{array}{ll}
\displaystyle\left(1
-\Delta t^2\frac{\partial_x\beta}{\alpha}\partial_{x} 
-\Delta t^2\frac{\beta}{\alpha}\partial_{x}^2\right)w^{n+\frac12}
=2u^{n}-u^{n-1}+\frac{\Delta t^2}{\alpha}f^{n+\frac{1}{2}}&
\qquad\mbox{ in }\Omega,\\[0.2cm]
w^{n+\frac12}=g^{n+1}&
\qquad\mbox{ on }\partial\Omega
\end{array}\right.
\end{equation}
\item[(W3)] Obtain $u^{n+1}$ as the solution of the boundary value
  problem
\begin{equation*}
\left\{
\begin{array}{ll}
\displaystyle\left(1
-\Delta t^2\frac{\partial_y\beta}{\alpha}\partial_{y} 
-\Delta t^2\frac{\beta}{\alpha}\partial_{y}^2\right)u^{n+1}
=w^{n+\frac12}&\qquad\mbox{ in }\Omega,\\[0.2cm]
u^{n+1}=g^{n+1}&\qquad\mbox{ on }\partial\Omega,
\end{array}\right.
\end{equation*}
\end{itemize}
\begin{remark}
  Note that the expressions~(D2)-(D4), and (W2)-(W3) of the
  alternating-direction ODEs for the heat and wave equations
  incorporate a division by the lowest-order variable-coefficient
  $\alpha$ in equations~\eqref{eq:heat_crank} and~\eqref{eq:ADI_2nd}.
  This is an essential detail of our algorithm: if such a division by
  $\alpha$ is not incorporated in the algorithm, the iterative GMRES
  solution of these ODEs (which is presented in
  Section~\ref{sec:solver-implementation}) would require large numbers
  of iterations. In the divided form, in contrast, the matrix of the
  linear-algebra problem is close to the identity for small $\Delta
  t$, and small numbers of GMRES iterations suffice to yield highly
  accurate ODE solutions. In fact, we have found in practice that the
  divided ODE forms give rise to small numbers of iterations even for
  large values of $\Delta t$.
\end{remark}
\begin{remark}
  It is easy to check that the scheme (W1)-(W3) is first order
  consistent. We refer to~\cite[Section 5]{lyon10} for a discussion of
  the global order of accuracy of the algorithm; in practice, and in
  agreement with that reference, we find the algorithm produces
  solutions with global first order accuracy. As shown in
  reference~\cite{bruno10} and Section~\ref{wave_numer}, further,
  solutions of higher order of temporal accuracy can be extracted from
  the solutions produced by the ADI scheme above by means of the
  Richardson extrapolation methodology.
\end{remark}

Problems~(D2),~(D4),~(W2) and~(W3) amount to two-point BVP of the form
\begin{align}
& \displaystyle u-pu'-qu''=f
\qquad \mbox{ in }(a,b),\label{eq:pb-orig}\\
& u(a)=d_{a}\quad ,\quad u(b)=d_{b},\label{eq:pb-orig_2}
\end{align}
where the right-hand side $f$ and the variable coefficients $p$ and
$q$ are bounded smooth functions defined in the interval $[a,b]$, with
$q>\eta>0$ for some constant $\eta$.  More precisely, the ODE
coefficients and the right-hand sides of the problems~(D2) and~(D4),
(W2) and (W3) are given by
\begin{equation}
\label{eq:coefODE_heat_wave}
\left\{
\begin{array}{ll}
  \!\!p(x)=\mathcal{P}^{H}(x,y),\ 
  q(x)=\mathcal{Q}(x,y),\ 
  f(x)=\mathcal{F}(x,y,t)
  &\mbox{ in problems (D2) and (W2) (}y,t\mbox{ fixed), and}\\
  \!\!p(y)=\mathcal{P}^{V}(x,y),\ 
  q(y)=\mathcal{Q}(x,y),\ 
  f(y)=\mathcal{F}(x,y,t)
  &\mbox{ in problems (D4) and (W3) (}x,t\mbox{ fixed),}
\end{array}
\right.
\end{equation}
where for~(D2) and~(D4)
\begin{equation}
\label{eq:coefODE_heat_def}
\mathcal{P}^{H}(x,y)=
\Delta t\frac{\partial_{x}\beta(x,y)}{2\alpha(x,y)},\ 
\mathcal{P}^{V}(x,y)=
\Delta t\frac{\partial_{y}\beta(x,y)}{2\alpha(x,y)},\ 
\mathcal{Q}(x,y)=
\Delta t\frac{\beta(x,y)}{2\alpha(x,y)},\ 
\mathcal{F}(x,y,t)=
\Delta t\frac{f(x,y,t)}{2\alpha(x,y)},
\end{equation}
while for~(W2) and~(W3)
\begin{equation}
\label{eq:coefODE_wave_def}
\mathcal{P}^{H}(x,y)=
\Delta t^2\frac{\partial_{x}\beta(x,y)}{\alpha(x,y)},\ 
\mathcal{P}^{V}(x,y)=
\Delta t^2\frac{\partial_{y}\beta(x,y)}{\alpha(x,y)},\ 
\mathcal{Q}(x,y)=
\Delta t^2\frac{\beta(x,y)}{\alpha(x,y)},\ 
\mathcal{F}(x,y,t)=
\Delta t^2\frac{f(x,y,t)}{\alpha(x,y)}.
\end{equation}

To solve the two-point BVP~\eqref{eq:pb-orig}-\eqref{eq:pb-orig_2} 
numerically we utilize a discrete method 
based on three main elements: 1) The Fourier Continuation method (see
Section~\ref{sec:fc-gram}) to produce accurate Fourier-series
approximations of the ODE source terms and variable coefficients; 2) A
specialized Fourier collocation method for solution of ODE boundary
value problems (see Section~\ref{sec:fourier-discrete})---which, in
view of item~1), can be applied to non-periodic boundary-value
problems without the accuracy degradation associated with the Gibbs
phenomenon; and~3) A new strategy for the enforcement the boundary
conditions in the context arising from items~1) and~2) (see
Sections~\ref{sec:equiv-for} and \ref{asympt_exp}).


\section{Fourier Continuation and scaled FC(Gram)}
\label{sec:fc-gram}
\subsection{Discrete Fourier analysis background}
Let $\mathcal{C}^{k}_{\mathrm{per}}(\tildea,\tildeb)$ denote the space
of $k$-times continuously differentiable periodic functions of period
$\tildeb-\tildea$. The discrete Fourier series of
$v\in\mathcal{C}^{k}_{\mathrm{per}}(\tildea,\tildeb)$ is given by
$$
\mathcal{J}v(x)=\sum_{n\in\mathcal{T}(\Nfourier)} \hat{v}_{n}
\mathrm{e}^{i\frac{2\pi n(x-\tildea)}{\tildeb-\tildea}} \in
\mathsf{B}(\tildea,\tildeb).
$$
Here $\mathcal{T}(\Nfourier)=
\{n\in\mathbb{N}:-\Nfourier/2+1\le n\le \Nfourier/2\}$ for
$\Nfourier$ even and $\mathcal{T}(\Nfourier)=
\{n\in\mathbb{N}:(\Nfourier-1)/2\le n\le -(\Nfourier-1)/2\}$
for $\Nfourier$ odd, 
\begin{equation}
  \mathsf{B}(\tildea,\tildeb)=
  \left\{g:(\tildea,\tildeb)\to\mathbb{R} \mbox{ such that } 
    g(x) = \sum_{n\in\mathcal{T}(\Nfourier)} g_{n}
\mathrm{e}^{i\frac{2\pi n (x-\tildea)}{\tildeb-\tildea}}\right\}
\end{equation}
denotes the $\Nfourier$-dimensional space of trigonometric
polynomials, and the amplitudes $\hat{v}_{n}$ are given by the
trapezoidal-rule expression
$$
\hat{v}_{n}=\frac{1}{\Nfourier}
\sum_{j=1}^{\Nfourier}v(x_{j})
\mathrm{e}^{-i\frac{2\pi n (x_{j}-\tildea)}{\tildeb-\tildea}},
\qquad n\in\mathcal{T}(\Nfourier).
$$
Note that, for conciseness, the ``degree'' $\Nfourier$ is not
explicitly displayed in the notation
$\mathsf{B}(\tildea,\tildeb)$.

We point out that, as is well-known~\cite{hesthaven07,boyd01},
\begin{enumerate}
\item An element of the set $\mathsf{B}(\tildea,\tildeb)$ is
  determined uniquely by its values at the equispaced grid
  $\tildea=x_{1}<\ldots<x_{\Nfourier}=\tildeb-h$, where
\begin{equation}\label{eq:fourier-grid-points}
  x_{j}=\tildea+(j-1)h,\qquad j=1,\ldots,\Nfourier,\quad 
\quad h=(\tildeb-\tildea)/\Nfourier, \mbox{ and}, 
\end{equation}
\item The discrete Fourier operator $\mathcal{J}$ is an interpolation
  operator from $\mathcal{C}^{k}_{\mathrm{per}}(\tildea,\tildeb)$
  into $\mathsf{B}(\tildea,\tildeb)$, that is,
  $\mathcal{J}v(x_j)=v(x_j)$ for all $j=1,\ldots,\Nfourier$.
\end{enumerate}

\subsection{Fourier Continuation and FC(Gram)\label{fc_gram}}
The Fourier Continuation method~\cite{bruno10,albin11} is an algorithm
which, acting on a set of values $f(x_{1}),\ldots,f(x_{\Nodes})$ of a
function $f:[a,b]\to\mathbb{R}$ at $\Nodes$ equidistant points $x_{j}\in
[a,b]$, produces a trigonometric polynomial
$f^{c}\in\mathsf{B}(\tildea,\tildeb)$, of a given degree
$\Nfourier$ and on a given interval
$[\tildea,\tildeb]\supset[a,b]$ with $b\le c$, which approximates the 
function $f$ closely on the interval $[a,b]$. Note that the endpoints 
$a$ and $b$ are not required to belong to the Fourier grid $\{x_j\}_{j=1}^{\Nodes}$. 
Clearly, for $[a,b]=[\tildea,\tildeb]$, the trigonometric interpolation
operator $\mathcal{J}$ leads to a naive Fourier Continuation procedure
which gives rise to the well known Gibbs ringing near $x=a$ and
$x=b$---unless $f$ is a smooth periodic function of period $(b-a)$. 
The selection of a larger expansion interval $[\tildea,\tildeb]$
allows for a smooth transition between the values $f(x_{j})$ near
$j=\Nodes$ to the values $f(x_{j})$ near $j=1$, and thus enables highly
accurate Fourier approximation in the interval $[a,b]$, avoiding the
undesirable Gibbs ringing effect and related accuracy deterioration.

A number of algorithms has been introduced which provide accurate
Fourier continuations for smooth non-periodic functions (including,
for example, the FC(SVD) method~\cite{bruno10}, which relies on
Singular Value Decompositions and the related
algorithms~\cite{bruno07,boyd02,bruno03}). Use of the spectrally
accurate Fourier continuations as a component of efficient PDE
solvers, however, must rely on correspondingly efficient Fourier
Continuation algorithms. For that purpose, an accelerated FC method,
the FC(Gram) procedure, was introduced
recently~\cite{bruno10,lyon10,albin11}.  In this section, we present a
brief description of FC(Gram) method; full details can be found
in~\cite{bruno10}. A new ``scaled'' version of the FC(Gram) approach,
which is necessary for the applications presented in this paper, is
introduced in Section~\ref{scaled_fc_gram}.

For the present description of the FC(Gram) method, without loss of
generality we assume $[a,b] = [0,1]$, as in
references~\cite{bruno10,lyon10}.  The FC(Gram) algorithm is based on
the $\Ndelta$ left-most and $\Ndelta$ right-most grid values $f(x_j)$
of $f$ within the subintervals $[0,\Delta]$ and $[1-\Delta,1]$, with
$\Delta=h\Ndelta$. Fixed the number of extension points $\Nd$
(belonging to $[1,1+d]$ with $d=h(\Nd-1)$) where $f^{c}(x_{j})$ has to
be computed, the continuation problem is reduced to seek a smooth
periodic function $f^c$ in $[\tildea,\tildeb]=[0,1+d]$.  For adequate
reference note that our quantities $\Nodes$, $N_{\Delta}$ and $\Nd$
are denoted by $n$, $n_\Delta$ and $n_d$ in the
contribution~\cite{bruno10}.
 
To compute $f^{c}$, the left-most grid values of $f$ in $[0,\Delta]$
are mapped in $[1+d,1+d+\Delta]$. Then, they are projected on certain
Gram bases of orthogonal polynomials---with orthogonality dictated by
the natural discrete scalar product defined by the grid
points. Finally, using the orthogonal polynomial expansion, a periodic
matching function $f_{\mathrm{match}}$, in the interval
$[1-\Delta,1+2d+\Delta]$ is computed. This function consists of a sum
of contributions of each polynomial projection (up to degree $\Ndeg$)
and blends smoothly the values at the left and right subintervals,
$[1-\Delta,1]$ and $[1+d,1+d+\Delta]$. The grid values of
$f_{\mathrm{match}}$ in $[1,1+d]$ give the required grid values of
$f^{c}$. This procedure can be efficiently implemented by precomputing
the sets of matching functions $\left\{f^r_\mathrm{even}
\right\}_{r=0}^{\Ndeg}$ and
$\left\{f^r_\mathrm{odd}\right\}_{r=0}^{\Ndeg}$ associated to the even
and odd pairs of the polynomial basis in the interval $[0,1]$ and
then, each FC(Gram) continuation in an arbitrary interval can be
obtained by using an affine transformation (see \cite[Section
2.3]{bruno10} for details).

It is clear that the accuracy of the Fourier continuation operator
depends on the number $\Nodes$ of Fourier grid points contained in
$[0,1]$, the maximum degree of polynomials $\Ndeg$ used in the
polynomial projections, the number $\Ndelta$ of right and left grid
points, and also on the number $\Nd$ of extension points outside the
interval $[0,1]$ (see item~\ref{pt_two} in Section~\ref{sec:numerical}
for an indication on the actual values of these parameters used in our
numerical examples).

\subsection{Scaled FC(Gram) algorithm\label{scaled_fc_gram}}
The FC-based PDE solvers presented in this paper rely on the FC
methodology as a tool for solving two-point BVPs for ordinary
differential equations with variable coefficients. As discussed in
Section~\ref{sec:accuracy-solver}, to obtain an accurate solution of
the relevant two-point BVPs for a given PDE problem in the present
variable-coefficient context, our method requires, unlike previous
FC-based approaches, use of numbers $N_g$ of extension points and
associated extension intervals of length $g$ that vary (linearly) with
$\Nodes$ (see Section~\ref{param_selc}).  Use of such
$\Nodes$-dependent extension intervals in the procedure described in
Section~\ref{fc_gram} entails evaluation of correspondingly
$\Nodes$-dependent sets of matching functions
$\left\{f^r_\mathrm{even}\right\}_{r=0}^{\Ndeg}$ and
$\left\{f^r_\mathrm{odd}\right\}_{r=0}^{\Ndeg}$---a procedure that is
inelegant and computationally expensive.

To overcome this difficulty, a slightly modified ``scaled'' version of
the FC(Gram) algorithm is introduced in this section. This approach is
based on use of a set of precomputed matching functions
$f_{\mathrm{match}}$, as described in Section~\ref{fc_gram}, in a
fixed interval $[1,1+d]$ and for a certain fixed value $\Nd$.  To
obtain inexpensively a matching function at a new number $N_{g}$ of
extension points in a new interval $[1,1+g]$ (where it is assumed that
$g/N_{g} = d/\Nd$), the ``scaled'' procedure uses as a matching
function a composition of the form 
\begin{equation}\label{scaled_f_match}
f_{\mathrm{match}}\circ \xi,
\end{equation}  
where $\xi : [1-\Delta,1+g+\Delta]\to [1-\Delta,1+d+\Delta]$ is a
smooth diffeomorphism.  To preserve the point values of
$f_{\mathrm{match}}$ in $[1-\Delta,1]$ and $[1+d,1+d+\Delta]$, the
scaling function $\xi$ is assumed to be map linearly the intervals
$[1-\Delta,1]$ and $[1+g,1+g+\Delta]$ onto $[1-\Delta,1]$ and
$[1+d,1+d+\Delta]$, respectively.  Throughout this paper we use the
scaling function
\begin{equation}\label{xi_def}
\xi(s)=1-\Delta+(d+2\Delta)
\mathrm{frac}\left(\frac{s+(d-g)\phi((s-1)/g)-1+\Delta}{d+2\Delta}\right)
\qquad s\in[1,1+g],
\end{equation}
where, letting $\lfloor x \rfloor$ denote the largest integer less than 
or equal to $x$, we have set $\mathrm{frac}(x) = x-\lfloor x\rfloor$, 
and the auxiliary function $\phi$ is given by
\begin{equation*}
\phi(s)=
\begin{cases}
1 & \mbox{ if }s>1,\\
\displaystyle\left(1+\exp\left(\frac{1}{s}-\frac{1}{1-s}\right)\right)^{-1} 
& \mbox{ if }s\in[0,1],\\
0 & \mbox{ if }s<0.\\
\end{cases}
\end{equation*}

In what follows, the scaled FC(Gram) continuation of a function $f$
resulting from this procedure will be denoted by either of the
following symbols
\begin{equation}\label{cont_oper}
  \tilde f  = E(f) = E_g(f).
\end{equation}
The first two of these notations can of course be used only when the
value of $g$ is either inconsequential or implicit in the context.
The scaled FC(Gram) algorithm produces Fourier continuations for
varying values of $g$ and $N_{g}$ at a computational
cost that is essentially the same as that required by the original
FC(Gram) procedure for a fixed number $\Nd$ of extension points.

\begin{remark}\label{FCGram_cost}
  Once the even and odd continuation of the Gram polynomials are
  computed and stored in memory, the FC(Gram) and scaled FC(Gram)
  continuations of any function only involve the computation of
  $\Ndeg$ inner products of size $\Ndelta$, which requires $2\Ndeg$
  sums and products, and the evaluation of a trigonometric
  interpolant, which is performed by means of an FFT of size
  $\Nodes+N_{g}-1$. Since $\Ndeg$ and $\Ndelta$ are independent of $\Nodes$,
  both the scaled and un-scaled versions of the FC(Gram) algorithm run
  at a computational cost of $\mathcal{O}((\Nodes+N_{g}-1)\log(\Nodes+N_{g}-1))$
  operations.
\end{remark}

\section{FC BVP solver~I: particular solution and boundary conditions}
\label{sec:ODE-theory}
This section presents an FC-based method for solution of
variable-coefficient BVPs of the
form~\eqref{eq:pb-orig}--\eqref{eq:pb-orig_2}.  As explained in what
follows, this approach relies on the scaled FC(Gram) method to produce
a periodic extension of the non-periodic
problem~\eqref{eq:pb-orig}--\eqref{eq:pb-orig_2} to an interval
$(\tildea,\tildeb)\supset (a,b)$, with boundary values at the
endpoints of the original interval $(a,b)$. As detailed in
Section~\ref{sec:fourier-discrete}, a particular solution for
equation~\eqref{eq:pb-orig} can easily be produced on the basis of the
FC(Gram) method. The enforcement of the boundary conditions, which
requires some consideration, is presented in
Sections~\ref{sec:equiv-for} and~\ref{asympt_exp}. In the first one of
these sections a direct numerical approach is presented for the
evaluation of solutions of the homogeneous
problem~\eqref{eq:pb-orig}--\eqref{eq:pb-orig_2} ($f=0$) with non-zero
boundary values, which can be used to correct the boundary
values of a particular solution. In Section~\ref{asympt_exp} a
complementary approach is introduced, which can effectively treat
challenging boundary layers and stiff equations that arise as small
time steps and correspondingly small coefficients $q(x)$ are used
(cf. equations~\eqref{D2},~\eqref{D4} and~\eqref{W2}).

\subsection{FC-based particular solution}
\label{sec:fourier-discrete}
To obtain a periodic embedding to the interval
$(\tildea,\tildeb)\supset (a,b)$ of the
BVP~\eqref{eq:pb-orig}--\eqref{eq:pb-orig_2} we utilize Fourier
continuations of the functions $f$, $p$ and $q$.  To preserve the
ellipticity of the problem, however, we must ensure that the extension
used for the coefficient $q$ takes on strictly positive values. To
produce such strictly positive extension of a positive function $q$ we
utilize a infinitely differentiable diffeomorphism
$\eta:\mathbb{R}\to(C_1,C_2)$ such as, for instance,
$\eta(s)=C_{1}+(C_{2}-C_{1})(1+\arctan(s))/2$. Using such a
diffeomorphism we define certain ``limited extensions'' of a positive
function $q$ by means of the expression
\begin{equation}\label{limited_ext}
  \tilde{q}^{\ell}=\eta\circ E(\eta^{-1}\circ q),
\end{equation}
where $E$ denotes the Fourier continuation procedure defined
in~\eqref{cont_oper}. The expression~\eqref{limited_ext} ensures that
the values of the periodic extensions $\tilde{q}^{\ell}$ are bounded
by above and below by the positive values $C_{1}$ and $C_{2}$,
respectively.  Hence, the ``periodic embedding'' of the ODE
problem~\eqref{eq:pb-orig}-\eqref{eq:pb-orig_2} is given by
\begin{equation}\label{per_ext}
u-\tilde{p}u'-\tilde{q}^{\ell}u''=\tilde{f}\quad\mbox{in}\quad
(\tildea,\tildeb),
\end{equation}
where, in accordance with equation~\eqref{cont_oper}, $\tilde{p}$ and
$\tilde{f}$ denote the scaled FC(Gram) continuation of the functions
$p$ and $f$.

To produce approximate periodic solutions of period
$\tildeb-\tildea $ of the ODE problem \eqref{eq:pb-orig}, our
algorithms relies on Fourier collocation: a numerical solution of
the form
\begin{equation*}
u_{\tilde{f}}(x)=\sum_{n\in\mathcal{T}(\Nodes+\Nd-1)} 
\hat{u}_{n}\mathrm{e}^{i\frac{2\pi n 
(x-\tildea)}{\tildeb-\tildea}}\in\mathsf{B}(\tildea,\tildeb)
\end{equation*}
is sought that satisfies~\eqref{eq:pb-orig} at each one of the grid
points $x_{j}$, $j=1,\ldots,\Nodes+\Nd-1$, that is
\begin{equation}
\label{eq:ode-lin-sys}
u_{\tilde{f}}(x_{j})
-\tilde{p}(x_{j})\frac{du_{\tilde{f}}}{dx}(x_{j})
-\tilde{q}^{\ell}(x_{j})\frac{d^{2}u_{\tilde{f}}}{dx^2}(x_{j})
=\tilde{f}(x_{j})
\qquad\mbox{ for } j=1,\ldots,\Nodes+\Nd-1.
\end{equation}
(if problem \eqref{eq:pb-orig}-\eqref{eq:pb-orig_2} has a unique
solution, this system of equations is not singular, see
e.g.~\cite{boyd01}.)  Since the solution $u_{\tilde{f}}$ is determined
uniquely by its values $u_{\tilde{f}}(x_{j})$ for
$j=1,\ldots,\Nodes+\Nd-1$ the $\mathbb{R}^{\Nodes+\Nd-1}$-vector of
point values of $u_{\tilde{f}}(x_{j})$, $j=1,\ldots,\Nodes+\Nd-1$ are
used as the unknowns of the problem. This linear system of equations
is solved by means of the iterative solver GMRES; see
Section~\ref{gmres} for details.
\begin{remark}
  Note that, since the coefficients $\tilde{p}$ and $\tilde{q}^{\ell}$
  are variable, the matrix associated to the linear system
  (\ref{eq:ode-lin-sys}) is not sparse. To avoid the expense
  associated with a direct solution of this linear system we utilize a
  preconditioned iterative solver, as described in
  Section~\ref{ode_impl}.
\end{remark}
Clearly, the particular solution described in this section does not
generally satisfy the boundary conditions imposed in
\eqref{eq:pb-orig_2}. The necessary corrections are described in the
following section.

\subsection{Boundary conditions~I: exterior sources}
\label{sec:equiv-for}

To enforce the necessary boundary conditions in~(\ref{eq:pb-orig_2})
we consider two auxiliary boundary value problems involving
ODEs of the form~(\ref{eq:pb-orig}) but with certain adequately-chosen
right-hand sides $g_{a}$, $g_{b}\in\mathsf{B}(\tildea,\tildeb)$.  The
right-hand sides $g_{a}$ and $g_{b}$ are taken to vanish at all
discretization points in the original interval $(a,b)$ but not to
vanish in $(b,c)$---that is to say, the selected right-hand sides
correspond to sources supported in the exterior of the physical domain
$[a,b]$, see e.g. Figure~\ref{fig:FC-gram-scaled} right. Clearly, any
linear combination of the form $\tilde{f}+\lambda_{a}g_{a}+
\lambda_{b}g_{b}$ with $\lambda_{a},\lambda_{b}\in\mathbb{R}$
coincides with $\tilde{f}$ in the part of the spatial grid contained
in $(a,b)$.  Calling $u_{a}$ and $u_{b}\in\mathsf{B}(\tildea,\tildeb)$
the approximate FC solutions (as described in
Section~\ref{sec:fourier-discrete}) corresponding to $g_{a}$ and
$g_{b}$, it is clear that any linear combination of the form
$\lambda_{a}u_{a}+\lambda_{b}u_{b}$ satisfies the ODE
(\ref{eq:pb-orig}) with null right-hand side at each one of the
collocation points $x_{j}\in(a,b)$. It follows that, denoting by
$(\lambda_{0,a},\lambda_{0,b})$ and $(\lambda_{1,a},\lambda_{1,b})$
the solutions of the $2\times2$ linear systems
\begin{equation}
\label{eq:linear-sys-bc-dis-1}
\begin{pmatrix}
u_{a}(a) & u_{b}(a)\\
u_{a}(b) & u_{b}(b)\\
\end{pmatrix}
\begin{pmatrix}
\lambda_{0,a}\\
\lambda_{0,b}\\
\end{pmatrix}
=
\begin{pmatrix}
1\\
0\\
\end{pmatrix}\qquad \mbox{and}\qquad
\begin{pmatrix}
u_{a}(a) & u_{b}(a)\\
u_{a}(b) & u_{b}(b)\\
\end{pmatrix}
\begin{pmatrix}
\lambda_{1,a}\\
\lambda_{1,b}\\
\end{pmatrix}
=
\begin{pmatrix}
0\\
1\\
\end{pmatrix},
\end{equation}
(as shown below, the system matrix is invertible for sufficiently fine
discretizations provided the functions $g_{a}$ and $g_{b}$ are chosen
appropriately), the functions
$w_{a}=\lambda_{0,a}u_{a}+\lambda_{0,b}u_{b}$ and
$w_{b}=\lambda_{1,a}u_{a}+\lambda_{1,b}u_{b}$ are approximate FC
solutions of~\eqref{eq:pb-orig} with null right-hand side, 
which satisfy the Dirichlet boundary
conditions $w_{a}(a)=1$, $w_{a}(b)=0$, $w_{b}(a)=0$ and $w_{b}(b)=1$.
Thus, the Fourier series
\begin{equation}\label{u_nodes}
u_{\Nodes}=u_{\tilde{f}}+(d_{a}-u_{\tilde{f}}(a))w_{a}+(d_{b}-u_{\tilde{f}}(b))w_{b}
\end{equation}
is an approximate FC solution of ODE problem~\eqref{eq:pb-orig} that
satisfies exactly the prescribed Dirichlet boundary
conditions~\eqref{eq:pb-orig_2}.

The needed invertibility of the system matrix in
equation~(\ref{eq:linear-sys-bc-dis-1}) for sufficiently fine
discretizations indeed holds provided the functions $g_{a}$ and
$g_{b}$ satisfy the conditions
\begin{equation}\label{eq:assump-lemma-ga-gb}
\int_{a}^{b}g_{a}(x)\,\mathrm{d}x>0,\qquad \int_{a}^{b}g_{b}(x)\,\mathrm{d}x=0.
\end{equation}
To establish this fact we let $U_a$ and $U_b$ be the {\em exact}
$(c-a)$-periodic solutions of equation~\eqref{eq:pb-orig} with
respective right-hand sides $g_{a}$ and $g_{b}$, and, using
Lemma~\ref{lem:ode} presented in Appendix~\ref{Appen}, we show at
first that the ``exact-solution system matrix'', that is, the matrix
that results as $u_a$ and $u_b$ are substituted by $U_a$ and $U_b$ in
equation~\eqref{eq:linear-sys-bc-dis-1}, is invertible.  Indeed, if
the exact-solution system matrix is singular then the $U_{a}$ and
$U_{b}$ satisfies $U_{a}(a)+\mu U_{b}(a)=0$ and $U_{a}(b)+\mu
U_b(b)=0$ for a certain $\mu\in\mathbb{R}$. The function $V =
U_{a}+\mu U_{b}$ is then a solution of the two-point 
BVP~(\ref{eq:pb-orig})--(\ref{eq:pb-orig_2}) with $f=0$, $d_a=0$ and
$d_b=0$, and thus, it vanishes identically since, in view of
Lemma~\ref{lem:ode} this problem admits a unique solution. It then
follows by the smoothness and periodicity of $V$ that $V(b)=V'(b)=0$
and $V(c)=V'(c)=0$.  But this is a contradiction, since
Lemma~\ref{lem:ode} tells us that such solutions do not exist under
the hypothesis~\eqref{eq:assump-lemma-ga-gb}. Hence, the
exact-solution system matrix is invertible. The invertibility for the
approximate-solution system matrix~\eqref{eq:linear-sys-bc-dis-1} for
sufficiently fine grids then follows from the convergence of the FC
solutions $u_a$ and $u_b$ to $U_{a}$ and $U_{b}$ as the grid-size
tends to zero (see Remark~\ref{convergence}).
\begin{remark}\label{convergence}
  Stability and convergence proofs, which are beyond the scope of this
  paper, are left for future work; here we merely note that
  convergence of the FC solutions to exact solutions is amply
  demonstrated by the numerical results presented in
  Sections~\ref{ode_impl} and~\ref{sec:full-disc}.
\end{remark}

\begin{remark}
  For the numerical experiments presented in this work we have used
  the auxiliary right-hand sides
\begin{align*}
g_{a}&=\phi\circ \psi,\\
g_{b}&=\frac{\psi g_{a}}
{\displaystyle\max_{x\in(b,\tildeb)}
\left|\psi(x)g_{a}(x)\right|},
\end{align*}
where $\psi(x)=2x-(b+\tildeb)/(\tildeb-b)$ and
$\phi(x)=\exp(1-1/(1-x^2))$; see
e.g. Figure~\ref{fig:FC-gram-scaled}. As required by
Lemma~\ref{lem:ode} in Appendix~\ref{Appen}, the functions $g_{a}$ and
$g_{b}$ satisfy the assumption~\eqref{eq:assump-lemma-ga-gb}, and
their supports, which are contained in the interval $(b,\tildeb)$, do
not intersect the interval $(a,b)$.
\end{remark}

\begin{remark}\label{bound_lay}
  Note that the solution $u$ and its numerical approximation
  $u_\Nodes$ generally possess boundary
  layers~\cite[Ch. 7]{bender_orzag} at the interval endpoints $a$ and
  $b$ in cases in which the coefficient function $\tilde{q}^{\ell}$ in
  equation~\eqref{eq:ode-lin-sys} is small---which, in view of
  equations~\eqref{D2},~\eqref{D4} and~\eqref{W2}, it certainly is for ODE
  problems arising from the FC PDE solver for small values of $\Delta
  t$. The presence of such boundary layers can be appreciated by
  consideration of equation~\eqref{u_nodes}: the functions $w_a$ and
  $w_b$ provide (numerical approximations of) the boundary-layer
  contributions.
\end{remark}

\begin{remark}
  The numerical boundary-layer FC solutions $w_{a}$ and $w_{b}$
  mentioned in Remark~\ref{bound_lay} may be affected by Gibbs-like
  ringing errors near the endpoints unless the underlying grid
  adequately resolves the exact boundary layers.  To ensure the
  detection of a non-resolved boundary layer (so that a procedure can
  be used to guarantee that the discretization errors are not
  inherited by $w_{a}$ and $w_{b}$ in stiff problems), the grid size
  $h$ is compared to the quantities
  $\sqrt{\tilde{q}^{\ell}(a)}$ and
  $\sqrt{\tilde{q}^{\ell}(b)}$.  If the grid-size
  $h$ is of the order of or smaller than these quantities, the
  boundary layer is well-resolved by the numerical approximations
  $w_{a}$ and $w_{b}$ and no further action is needed.  Otherwise, if
  $h$ is larger than $\sqrt{\tilde{q}^{\ell}(a)}$
  and $\sqrt{\tilde{q}^{\ell}(b)}$, then $w_{a}$
  and $w_{b}$ do not adequately resolve the boundary layers, and an
  alternative treatment is necessary
  (Remark~\ref{rem:detect_boundary_layer} presents details on the
  actual thresholds used in our numerical examples). Such an
  alternative method, based on use of asymptotic expansions, is
  presented in Section~\ref{asympt_exp}.
\end{remark}

\subsection{Boundary conditions~II: High-order asymptotic-matching
  expansions\label{asympt_exp}}

As noted in the introductory paragraph of
Section~\ref{sec:ODE-theory}, stiff ODEs and challenging boundary
layers arise in the solutions of the boundary value
problem~\eqref{eq:pb-orig}--(\ref{eq:pb-orig_2}) as small time steps
are used. In this section we present an algorithm that can resolve
such boundary layers, without resort to unduly fine spatial meshes, on
the basis of the method of matched asymptotic
expansions~\cite{bender_orzag}.  For our development of this
asymptotic procedure, we introduce a small positive parameter
$\varepsilon$ (that is taken to equal $\Delta t$ in
equation~\eqref{W2} and $\sqrt{\Delta t/2}$ in equations~\eqref{D2} 
and~\eqref{D4}), we
let $p = \varepsilon^2 p_0$ and $q= \varepsilon^2 q_0$, and we
re-express equation~\eqref{eq:pb-orig} in the form
\begin{align}
& (Lu)(x)=r(x)u(x)-\varepsilon^2
\displaystyle\frac{d}{dx}\left(s(x)\frac{du}{dx}(x)\right)=f(x)\quad,
\quad x\in(a,b),\label{eq:pb-orig-div-form}\\
& u(a)=d_{a}\quad ,\quad u(b)=d_{b},\label{eq:pb-orig_2-div-form}
\end{align}
where $f$ and the variable coefficients $s$ and $r$ are bounded smooth
functions defined in the interval $[a,b]$ (which satisfy the conditions
$s,r>\eta>0$ for some constant $\eta$) given by
$$
p_0(x)=\frac{1}{r(x)}\frac{ds}{dx}(x),\qquad q_0(x)=\frac{s(x)}{r(x)}.
$$
Further, using the change of variables
\begin{equation}\label{ch_of_vars}
  y(x)=\int_{a}^{x}\frac{1}{s(\tau)}\,\mathrm{d}\tau,\qquad x\in(a,b),
\end{equation}
equation~\eqref{eq:pb-orig-div-form} can be re-cast in the simpler
form
\begin{equation}\label{constant_coeff}
\hat{L}\hat{u}(y)=\hat{r}(y)\hat{u}(y)-\varepsilon^2
\displaystyle\frac{d^2\hat{u}}{dy^2}(y)=0, \qquad y\in(0,\hat{b}),
\end{equation}
where  $\hat{r}(y(x))=r(x)s(x)$, $\hat{b}=y(b)$ and $\hat{u}(y(x))=u(x)$

In what follows we produce, on the basis of the method of matched
asymptotic expansions~\cite{bender_orzag}, approximate solutions
$w_{a,\varepsilon}$ and $w_{b,\varepsilon}$ of the homogeneous ($f=0$)
version of equation~\eqref{eq:pb-orig-div-form}, satisfying
$w_{a,\varepsilon}(a)=1$, $w_{a,\varepsilon}(b)=0$,
$w_{b,\varepsilon}(a)=0$, $w_{b,\varepsilon}(b)=1$; like the functions
$w_a$ and $w_b$ introduced in the previous section,
$w_{a,\varepsilon}$ and $w_{b,\varepsilon}$ can be used to correct the
boundary values of a periodic FC solution $u$ of
\eqref{eq:pb-orig-div-form}. Unlike the functions introduced in the
previous section, however, the asymptotic-expansion solutions
$w_{a,\varepsilon}$ and $w_{b,\varepsilon}$ produce accurate solutions
in the small $\varepsilon$ regime without requiring use of fine
meshes, and they can be evaluated with a fixed computational cost for
arbitrarily small values of $\varepsilon$.
\begin{remark}\label{eps_delta_t}
  As mentioned above in this section, the small parameter
  $\varepsilon$ included in the variable coefficients $p$ and $q$
  corresponds to the quantities $\Delta t$ and $\sqrt{\Delta t/2}$ in
  the time discretization of the wave and heat equations,
  respectively.  Hence, the uniform convergence of the ODE solutions
  as $h\to 0$ with respect to $\varepsilon$ guarantees the spatial
  convergence of the FC-AD scheme independently of the value of
  $\Delta t$ used in the time-marching schemes.
\end{remark}

We lay down our procedure for evaluation of $w_{a}(x)$ (left boundary
layer); the corresponding method for $w_{b}(x)$ (right boundary layer)
is entirely analogous. Using the change of variables $y = y(x)$ we
define a new unknown $\hat{w}= \hat{w}(y)$ by
$\hat{w}(y(x))=w_{a}(x)$; clearly $\hat{w}(y)$ satisfies
equation~\eqref{constant_coeff} together with the boundary conditions
$$
\hat{w}(0)=1,\qquad \hat{w}(\hat{b})=0.
$$

To obtain the approximate solution $\hat{w}(y)$ for small values of
$\varepsilon$ we use outer and inner solutions in the corresponding
matched asymptotic expansion~\cite{bender_orzag} associated with the
left boundary layer.  Setting $\varepsilon=0$ in the ODE operator
$\hat{L}$ and using the null boundary condition at $y=b$ we see that
the lowest order term in the outer asymptotic expansion vanishes. In
fact, as it follows from the discussion below, the outer solution
vanishes to all orders:
$$
\hat{w}_{\mathrm{out}}(y)=0,\qquad y\in(0,\hat{b}).
$$
In the inner region, in turn, we use the scaled variable
$Y=y/\varepsilon$ and we express the inner solution
$\hat{w}_{\mathrm{in}}(Y) \sim \hat{w}(\varepsilon Y) $ and the ODE
coefficient $\hat{r}$ as asymptotic expansions
\begin{equation}\label{add_exp}
\hat{w}_{\mathrm{in}}(Y)=\sum_{j=0}^{k}\varepsilon^{j}\hat{w}_{j}(Y),\qquad
\hat{r}(\varepsilon Y)=\sum_{j=0}^{k}\varepsilon^{j}\hat{r}_{j}Y^{j}.
\end{equation}
In particular, the Taylor expansion of the ODE coefficient
$\hat{r}(\varepsilon Y)$ induces a corresponding expansion for the ODE
operator $\hat{L}$, namely
$\hat{L}=\sum_{j=0}^{k}\varepsilon^{j}\hat{L}_{j}$, where
$$
\hat{L}_{0}=-\frac{d^2}{dY^2}+\hat{r}_{0},\qquad
\hat{L}_{j}=\hat{r}_{j}Y^{j},\ j=1,\ldots,k.
$$
It follows that the functions $\hat{w}_{j}$ are solutions of the
following ODE problems:

\begin{equation}\label{eq:edo_L0}
\left\{
\begin{array}{ll}
\hat{L}_{0}\hat{w}_{0}(Y)=0,&\qquad Y\in(0,\infty),\\[0.2cm]
\hat{w}_{0}(0)=1,&\\[0.2cm]
\displaystyle\lim_{Y\to\infty}\hat{w}_{0}(Y)=0&
\end{array}
\right.
\end{equation}
for $j=0$, and
\begin{equation}\label{eq:edo_Lj}
\left\{
\begin{array}{ll}
  \hat{L}_{j}\hat{w}_{j}(Y)=
  \displaystyle\sum_{m=1}^{j}\hat{L}_{m}\hat{w}_{j-m}(Y)
  =\displaystyle\sum_{m=1}^{j}\hat{r}_{m}Y^{m}\hat{w}_{j-m}(Y),
  &\qquad Y\in(0,\infty),\\
  \hat{w}_{j}(0)=0,&\\[0.2cm]
  \displaystyle\lim_{Y\to\infty}\hat{w}_{j}(Y)=0&
\end{array}
\right.
\end{equation}
for $j=1,\ldots,k$. It is easy to check that
$\hat{w}_{0}(Y)=\exp(-\sqrt{\hat{r}_{0}}Y)$, that, similarly,
$\hat{w}_{j}$ can be expressed in closed form in terms of adequate
combinations of exponentials and polynomials 
(see Algorithm~\ref{algor:asym_expand}), and
finally, that, as claimed above, all terms in the outer asymptotic
expansion vanish.

In sum, we have
\begin{equation}\label{eq:asymp_expand}
w_{a,\varepsilon}(x)=\sum_{j=0}^{k}\varepsilon^{j}\hat{w}_{j}\left(\frac{y(x)}
{\varepsilon}\right),
\end{equation}
where the functions $\hat{w}_{j}$ satisfy~\eqref{eq:edo_L0}
and~\eqref{eq:edo_Lj}; a corresponding asymptotic boundary correction
function $w_{b,\varepsilon}$ can be obtained similarly. The maximum errors
resulting from these approximations satisfy
$||w_{a}-w_{a,\varepsilon}||=\mathcal{O}(\varepsilon^{k+1})$ and
$||w_{b}-w_{b,\varepsilon}||=\mathcal{O}(\varepsilon^{k+1})$ for small
$\varepsilon$. Using these asymptotic boundary correction functions we
then obtain an approximate FC solution of ODE
problem~\eqref{eq:pb-orig}, which satisfies exactly the prescribed
Dirichlet boundary conditions~\eqref{eq:pb-orig_2}, by means of the
expression
\begin{equation}\label{u_nodes_asymp}
  u_{\Nodes}=u_{\tilde{f}}+(d_{a}-u_{\tilde{f}}(a))w_{a,\varepsilon}
  +(d_{b}-u_{\tilde{f}}(b))w_{b,\varepsilon}.
\end{equation}
\begin{remark}\label{asymp_order_bc}
  Except for the accuracy tests presented in
  Section~\ref{sec:accuracy-solver}, all numerical examples presented
  in this paper use asymptotic expansions of order $k=3$ for
  evaluation of $w_{a,\varepsilon}$ and $w_{b,\varepsilon}$ whenever
  asymptotic expansions are needed for evaluation of boundary
  corrections.  The selection of the value $k=3$ is based, precisely,
  on a series of numerical tests presented in
  Section~\ref{sec:accuracy-solver}.
\end{remark}

We conclude this section with a description of the algorithm we use
for evaluation of the terms in the asymptotic
series~\eqref{eq:asymp_expand}. At first a Taylor series centered at
$y=0$ for the ODE coefficient $\hat{r}$ is obtained, either from a
closed form expression or by numerical differentiation of an
associated Fourier continuation series: the derivatives needed to
evaluate
\begin{equation}\label{eq:taylor_deriv}
  \hat{r}_{j} = \left.\frac{d^{j}\hat{r}}{dy^{j}}(0)=\left(s(x)\frac{d}{dx}\right)^{j}(r
    (x)s(x))
  \right|_{x=a}\mbox{ for }j=0,\ldots,k.
\end{equation}
can be accurately produced by differentiation of limited extensions
provided by the scaled FC(Gram) procedure, that is, by replacing $r$
and $s$ by $\tilde{r}^\ell$ and $\tilde{s}^\ell$, respectively (see
Subsection~\ref{sec:fourier-discrete}). Hence, only the grid values of
the ODE coefficients are necessary to approximate the Taylor
coefficients. This numerical differentiation procedure was used for
all of the examples requiring use of asymptotic expansions in
Sections~\ref{sec:accuracy-solver} and~\ref{sec:numerical}.

Additionally, it is also necessary to provide an accurate numerical
approximation of the primitive~\eqref{ch_of_vars} that defines the
change of variables $y(x)$.  Our algorithm produces this primitive by
relying on the Fourier continuation $\tilde{s}^\ell$. Using this
function we obtain the Fourier series
\begin{equation*}
  \frac{1}{\tilde{s}^\ell}(x)=\sum_{n\in\mathcal{T}(\Nodes+\Nd-1)} 
  \sigma_{n}\mathrm{e}^{i\frac{2\pi n 
      (x-\tildea)}{\tildeb-\tildea}}\in\mathsf{B}(\tildea,\tildeb)
\end{equation*}
(by means of a Fast Fourier transform (FFT)), from which integration
is straightforward:
\begin{equation}\label{eq:int_fft_y}
  y_{\Nodes}(x)=\int_{a}^{x}\left(\frac{1}{\tilde{s}^\ell(\tau)}
    -\sigma_0\right)
  \,\mathrm{d}\tau+(x-a)\sigma_0,\ x\in(a,b).
\end{equation}
In summary, the pseudocode in Algorithm~\ref{algor:asym_expand} outlines the procedure we
use for evaluation of the grid values of the asymptotic expansion
$w_{a,\varepsilon}$ of order $k$, from two given vectors of grid
values of the variable coefficients $r$ and $s$.

\begin{algorithm}
  \caption{Evaluation of the $k$-th order asymptotic expansion
    $w_{a,\varepsilon}$ for the ODE
    problem~\eqref{eq:pb-orig-div-form}--\eqref{eq:pb-orig_2-div-form}
    with variable coefficients $r$ and $s$.}
\label{algor:asym_expand}
\begin{algorithmic}
  \REQUIRE the grid values of the funcions $r,s$ in $(a,b)$ \STATE
  $\hat{r}_{0},\ldots,\hat{r}_{k}\leftarrow$\texttt{taylor\_coef}($k,r,s$)
  \COMMENT{compute the Taylor coefficients in
    equation~\eqref{eq:taylor_deriv}} \STATE
  $\hat{w}_{0},\ldots,\hat{w}_{k}\leftarrow$\texttt{ode\_solve}($\hat{r}_{0},\ldots,\hat{r}_{k}$)
  \COMMENT{compute analytically the ODE solutions \eqref{eq:edo_L0}
    and \eqref{eq:edo_Lj}} \STATE
  $y_{\Nodes}\leftarrow$\texttt{int\_fft}($1/\tilde{s}^\ell$)
  \COMMENT{integrate $1/\tilde{s}^\ell$ using a zero-mean Fourier
    series~\eqref{eq:int_fft_y}} \STATE
  $\hat{w}_{\mathrm{in}}\leftarrow$\texttt{expand}($\hat{w}_{0},\ldots,\hat{w}_{k}$)
  \COMMENT{evaluate the sums~\eqref{add_exp}} \STATE
  $w_{a,\varepsilon}\leftarrow$\texttt{compose}($\hat{w}_{\mathrm{in}},y_{\Nodes}/\varepsilon$)
  \COMMENT{compose the inner expansion with
    $Y=y_{\Nodes}/\varepsilon$} \ENSURE the grid values of
  $w_{a,\varepsilon}$
\end{algorithmic}
\end{algorithm}

\subsection{Discrete ODE operators}\label{sec:discrete-operators}
Here we introduce two operators $\tilde{A}_{\Nodes}$ and
$\tilde{B}_{\Nodes}$ associated with the ODE solvers described in the
previous sections; use of these operators facilitates the introduction
of the overall PDE solvers in Section~\ref{sec:full-disc}.  Briefly,
the first one of these is the ``ODE solution operator'' (described in
Sections~\ref{sec:fourier-discrete} through~\ref{asympt_exp}) which,
given coefficients, right-hand sides and boundary values, produces
approximate grid values $\boldsymbol{u} =
(u_{\Nodes}(x_1),\dots,u_{\Nodes}(x_\Nodes))^t$ of the solution of the
two-point BVP~\eqref{eq:pb-orig}--\eqref{eq:pb-orig_2}. The operator
$\tilde{B}_{\Nodes}$, amounts, simply, to evaluation of the ODE
differential operator, for given ODE coefficients, for a given discrete
function $\boldsymbol{w}$.

More precisely, given the equispaced grid points
$x_{1},\ldots,x_{\Nodes}$ in the domain $(a,b)$ of the two-point
boundary value problem \eqref{eq:pb-orig}--\eqref{eq:pb-orig_2} and
the grid values
$\boldsymbol{p}=(p_{1},\ldots,p_{\Nodes})^{t}\in\mathbb{R}^{\Nodes}$
and
$\boldsymbol{q}=(q_{1},\ldots,q_{\Nodes})^{t}\in\mathbb{R}^{\Nodes}$,
of the ODE variable coefficients $p$ and $q$, the linear operator
(matrix)
$\tilde{A}_{\Nodes}=\tilde{A}_{\Nodes}(\boldsymbol{p},\boldsymbol{q})$
maps the grid values
$\boldsymbol{f}=(f_{1},\ldots,f_{\Nodes})^{t}\in\mathbb{R}^{\Nodes}$
of the right-hand side $f$ and the Dirichlet data
$\boldsymbol{d}=(d_{a},d_{b})^t\in\mathbb{R}^2$ into the approximate
point values
$\boldsymbol{u}=(u_{\Nodes}(x_{1}),\ldots,u_{\Nodes}(x_{\Nodes}))^{t}$
of the ODE solution as defined in Sections~\ref{sec:fourier-discrete}
through~\ref{asympt_exp}:
\begin{equation}\label{eq:def_A_N}
  (\boldsymbol{f},\boldsymbol{p},\boldsymbol{q},\boldsymbol{d})
  \in\mathbb{R}^{3\Nodes+2}\mapsto 
  \boldsymbol{u}=
  \tilde{A}_{\Nodes}(\boldsymbol{p},\boldsymbol{q})
\begin{pmatrix}
\boldsymbol{f}\\
\boldsymbol{d}
\end{pmatrix}
\in\mathbb{R}^{\Nodes}.
\end{equation}
Analogously, the matrix
$\tilde{B}_{\Nodes}=\tilde{B}_{\Nodes}(\boldsymbol{p},\boldsymbol{q})$
maps the grid values
$\boldsymbol{w}=(w_{1},\ldots,w_{\Nodes})^{t}\in\mathbb{R}^{\Nodes}$
into the grid point values
$\boldsymbol{g}=(g_{1},\ldots,g_{\Nodes})^t$ resulting from evaluation
of the periodic ODE operator on the left-hand side of~\eqref{per_ext}
on $\boldsymbol{w}$, that is to say,
\begin{equation}\label{eq:def_B_N}
  (\boldsymbol{w},\boldsymbol{p},\boldsymbol{q})
  \in\mathbb{R}^{3\Nodes}\mapsto \boldsymbol{g}=
  \tilde{B}_{\Nodes}(\boldsymbol{p},\boldsymbol{q})\; \boldsymbol{w}
  \in\mathbb{R}^{\Nodes},
\end{equation}
where, letting $w\in\mathsf{B}(\tildea,\tildeb)$ denote the scaled
FC(Gram) continuation obtained from the grid values $\boldsymbol{w}$, we have set
$$
g_{j}=w(x_{j})
-\tilde{p}(x_{j})\frac{dw}{dx}(x_{j})
-\tilde{q}^{\ell}(x_{j})\frac{d^{2}w}{dx^2}(x_{j})
\qquad\mbox{ for }j=1,\ldots,\Nodes .
$$

\begin{remark}\label{end_points}
  Note that, 1) As pointed out in Section~\ref{fc_gram}, the endpoints
  $a$ and $b$ are not required to belong to the Fourier-continuation
  grid $\{x_j\}_{j=1}^{\Nodes}$ (and, in fact, they are chosen not to belong to the
  spatial grid $x_{1},\ldots,x_{\Nodes}$ used in the context of the
  ODE operators considered in this section), and, 2) The ODE
  evaluation inherent in the operator $\tilde{B}_{\Nodes}$ does not
  involve the Dirichlet boundary data $d_{a}$ and $d_{b}$, since
  neither $a$ nor $b$ belong to the spatial grid used.
\end{remark}
\begin{remark}\label{AN_BN_dep_on_ab}
  Note that the operators $\tilde{A}_{\Nodes}$ and
  $\tilde{B}_{\Nodes}$ depend on the interval endpoints $a$ and $b$,
  although the notations~\eqref{eq:def_A_N}-\eqref{eq:def_B_N} do not
  make this dependence explicit. This fact is relevant in the context
  of Section~\ref{sec:full-disc}, where the endpoints $a$ and $b$
  generally change as varying Cartesian lines and corresponding ODE
  domains are used as part of an alternating-direction PDE solution
  procedure.
\end{remark}

\subsection{Boundary projections\label{bound_proj}}
In order to ensure stability and convergence, the PDE solvers
described in Section~\ref{sec:full-disc} utilize slightly modified
versions of the ODE operators $\tilde{A}_{\Nodes}$ and
$\tilde{B}_{\Nodes}$ introduced in the previous section; the resulting
modified operators, which incorporate prescriptions put forth
in~\cite{bruno10} and~\cite{lyon10}, are denoted by $A_{\Nodes}$ and
$B_{\Nodes}$. Here we sketch the projection and correction procedures
that produce $A_{\Nodes}$ and $B_{\Nodes}$ from $\tilde{A}_{\Nodes}$
and $\tilde{B}_{\Nodes}$; full details in these regards can be found
in reference~\cite{bruno10}. 

Letting $\boldsymbol{v}$ denote
$\boldsymbol{v}=\tilde{A}_{\Nodes}(\boldsymbol{p},\boldsymbol{q})(\boldsymbol{f},\boldsymbol{d})^t$
(resp.
$\boldsymbol{v}=\tilde{B}_{\Nodes}(\boldsymbol{p},\boldsymbol{q})\boldsymbol{w}$),
the vector
$A_{\Nodes}(\boldsymbol{p},\boldsymbol{q})(\boldsymbol{f},\boldsymbol{d})^t$
(resp. $B_{\Nodes}(\boldsymbol{p},\boldsymbol{q})\boldsymbol{w}$) is
defined as the result of an application to the vector $\boldsymbol{v}$
of boundary projections and corrections resulting from the following
procedure:
\begin{itemize}
\item[(a)] \label{item_a} With reference to Section~\ref{fc_gram},
  construct a vector $\boldsymbol{v}^{p}$ of length $\Nodes$ by
  replacing the $\Ndelta$ left-most and $\Ndelta$ right-most values of
  $\boldsymbol{v} = (v_1,\dots,v_\Nodes)^t$ by the corresponding
  values of certain polynomial approximants, the ``open Gram
  polynomial projections'', that result from application of the
  FC(Gram) procedure to the corresponding values
  $(v_1,\dots,v_{\Ndelta})^t$ and $(v_{\Ndelta-\Nodes
    +1},\dots,v_{\Nodes})^t$, respectively. These corrections are
  called ``open'' as they do not involve the two endpoints of the
  interval $[a,b]$; see~\cite{bruno10}.
\item[(b)] Analogously, construct a vector $\boldsymbol{v}^{b}$ of
  length $\Nodes$ by replacing the $\Ndelta$ left-most and $\Ndelta$
  right-most values of $\boldsymbol{v} = (v_1,\dots,v_\Nodes)^t$ by
  the corresponding values of certain polynomial approximants, the
  ``closed Gram polynomial projections'', that result from application
  of the FC(Gram) procedure to the $\Ndelta +1$-tuples
  $(d_a,v_1,\dots,v_{\Ndelta})^t$ and $(v_{\Ndelta-\Nodes
    +1},\dots,v_{\Nodes},d_b)^t$, respectively, where $d_a$ and $d_b$
  are user-provided values (which, in the context of the PDE solver of
  Section~\ref{sec:full-disc}, are given be the PDE boundary values on
  the corresponding horizontal or vertical ADI lines). These
  corrections are called ``closed'' as they involve the two endpoints
  of the interval $[a,b]$; see~\cite{bruno10}.
\item[(c)] Using a second-order centered finite difference scheme
  obtain a discrete solution $\boldsymbol{\nu} =
  (\nu_1,\dots,\nu_{\Nodes})^t$ of the two-point
  BVP~\eqref{eq:pb-orig}-\eqref{eq:pb-orig_2} with right-hand
  side $f-E(f)$ (see equation~\ref{cont_oper}) and construct the open
  projection $\boldsymbol{\nu}^{p}$ of $\boldsymbol{\nu}$, using the
  prescription in item~(a) above.
\item[(d)] The vector
  $A_{\Nodes}(\boldsymbol{p},\boldsymbol{q})(\boldsymbol{f},\boldsymbol{d})^t$
  (resp. $B_{\Nodes}(\boldsymbol{p},\boldsymbol{q})\boldsymbol{w}$)
  is now produced as the linear combination
  $(1-\chi)\boldsymbol{v}^{p}+\chi \boldsymbol{v}^{b} +
  \boldsymbol{\nu} -\boldsymbol{\nu}^{p}$, where
\begin{equation}\label{chi}
\chi=\min\left\{1,\frac{\Nd-2}{h^2}\displaystyle
\max_{x\in(\tildea,\tildeb)}\tilde{q}^{\ell}(x)\right\}.
\end{equation}
(The definition~\eqref{chi} is a direct generalization of a formula
put forth in~\cite{bruno10} for the corresponding constant
coefficient two-point BVP.)
\end{itemize}

\section{FC BVP solver~II: implementation and numerical results\label{ode_impl}}
This section consists of four subsections, Sections~\ref{gmres}
through~\ref{ODE_cost}, which discuss, in turn, 1)~A method of
solution of the linear systems introduced in
Sections~\ref{sec:fourier-discrete} and~\ref{sec:equiv-for} via
preconditioned GMRES; 2) Selection of parameters needed for the scaled
FC(Gram) algorithm described in Section~\ref{scaled_fc_gram}; 3)~The
accuracy inherent in the resulting scaled FC BVP solver; and, 4)~The
computational cost of the ODE solution algorithm with parameters as in
point~3).

\subsection{Periodic ODE solution via preconditioned
  GMRES\label{gmres}}
\label{sec:solver-implementation}
The discretization~(\ref{eq:ode-lin-sys}) of the periodic variable
coefficient ODE~\eqref{per_ext} gives rise to the {\em full} linear
system of equations
\begin{equation}\label{lin_sys_B}
  \tilde{B}_{\Nodes}(\boldsymbol{p},\boldsymbol{q})\boldsymbol{w}
  =\boldsymbol{g}
\end{equation}
associated with the matrix
$\tilde{B}_{\Nodes}(\boldsymbol{p},\boldsymbol{q})$ defined in
equation~\eqref{eq:def_B_N}.  The solution of such linear systems can
be obtained efficiently by relying on the iterative linear algebra
solver GMRES~\cite{walker88}, on the basis of 1)~FFT-based evaluation
of forward maps given by the matrix
$\tilde{B}_{\Nodes}(\boldsymbol{p},\boldsymbol{q})$, and 2)~Use of a
finite-difference based preconditioner---as described in what follows.

The pseudocode in Algorithm~\ref{algor:ode_eval} outlines our
algorithm for evaluation of the operator
$\tilde{B}_{\Nodes}(\boldsymbol{p},\boldsymbol{q})\boldsymbol{w}$ for
a given vector $\boldsymbol{w}$ of grid values. Assuming that the
point values $\boldsymbol{p}$ and $\boldsymbol{q}$ of the variable
coefficients have been obtained a priori, the ODE evaluation involves
an application of the scaled FC(Gram) algorithm to $\boldsymbol{w}$
followed by computation of the point values of its first and second
derivatives and multiplication by the variable coefficients. In order
to mitigate aliasing effects in evaluation of products, oversampling
by a factor of two (obtained from an adequately zero-padded FFTs) is
used in the precomputed coefficients $\boldsymbol{p}$ and
$\boldsymbol{q}$ as well as the function values and derivatives
arising from the vector $\boldsymbol{w}$.
\begin{algorithm}
  \caption{Evaluation of the function
    $\boldsymbol{g}=\tilde{B}_{\Nodes}(\boldsymbol{p},\boldsymbol{q})\boldsymbol{w}$:
    application of the ODE operator for a vector $\boldsymbol{w}$ of
    function grid values}
\label{algor:ode_eval}
\begin{algorithmic}
  \REQUIRE the function grid values $\boldsymbol{w}$ \STATE
  $\tilde{\boldsymbol{w}}\leftarrow$\texttt{scaled\_fcgram}($\boldsymbol{w}$)
  \COMMENT{apply the scaled FC(Gram) procedure} \STATE
  $\hat{\tilde{\boldsymbol{w}}}\leftarrow$\texttt{fft}($\tilde{\boldsymbol{w}}$)
  \COMMENT{apply the Fourier transform} \STATE
  $\hat{\boldsymbol{g}}\leftarrow$\texttt{oversample}($\hat{\tilde{\boldsymbol{w}}}$)
  \COMMENT{oversample the Fourier series (by zero-padding) by a factor
    of two} \STATE $(\hat{\boldsymbol{g}}',\hat{\boldsymbol{g}}'')
  \leftarrow$\texttt{fourier\_diff}($\hat{\boldsymbol{g}}$)
  \COMMENT{compute Fourier space derivatives} \STATE
  $(\boldsymbol{g}',\boldsymbol{g}'')
  \leftarrow$\texttt{ifft}($\hat{\boldsymbol{g}}',\hat{\boldsymbol{g}}''$)
  \COMMENT{evaluate derivatives in physical space} \STATE
  $\boldsymbol{g}\leftarrow
  \boldsymbol{g}-\boldsymbol{p}\boldsymbol{g}'-\boldsymbol{q}\boldsymbol{g}''$
  \COMMENT{evaluate the ODE operator} \ENSURE the grid values of
  $\boldsymbol{g}$
\end{algorithmic}
\end{algorithm}

As mentioned above, our algorithm obtains the solutions of the linear
systems~\eqref{lin_sys_B} by means of a {\em preconditioned} GMRES
solver. (A direct application of the un-preconditioned GMRES algorithm
to these linear systems requires extremely large numbers of iterations
to meet even modest accuracy tolerances: in the case of the two-point
BVP considered in Figure~\ref{fig:GMRES-iter}, for example, the
numbers of iterations required to obtain a residual of $10^{-10}$ for
$N=100$, $1000$ and $3000$ are 145, 1279 and 3799, respectively.)  The
preconditioners used, which are mentioned in point 2) above, are given
by inverses of the matrices corresponding to (possibly oversampled)
second-order finite difference (FD) approximations of the periodic ODE
problem~\eqref{per_ext} (cf. reference~\cite{Sun96} and associated
comments in Section~\ref{intro}).  Since the ODE variable coefficients
and the right-hand side in \eqref{eq:pb-orig} are accurately
represented by their scaled FC(Gram) continuations, a zero-padding
procedure can be used to evaluate the ODE coefficients and right-hand
side on a spatial grid that is finer, by a certain factor $\Nover$,
than the one used in the original Fourier collocation
discretization~(\ref{eq:ode-lin-sys}), thus easily leading to an
oversampled preconditioner.  With limited impact on the computing cost
of the algorithm (see Section~\ref{ODE_cost}), a fine grid FD
discretization can improve significantly, by a factor of $1/\Nover^2$,
the accuracy of the preconditioner---and thus, as demonstrated in
Figure~\ref{fig:GMRES-iter}, its preconditioning capability.

\begin{figure}[!ht] 
\begin{center} 
$\begin{array}{cc}
\includegraphics[width=0.475\textwidth,angle=0]
{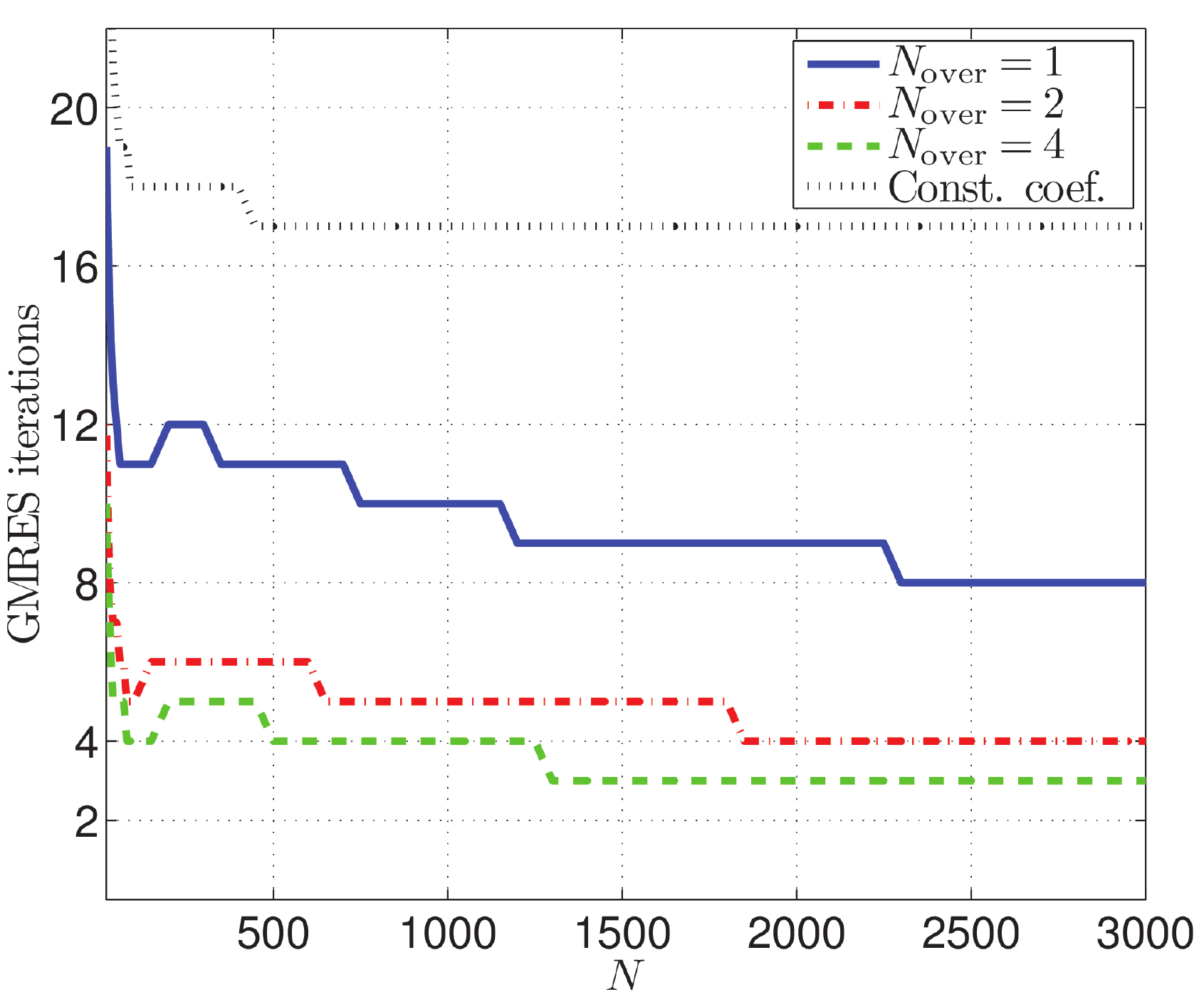} & 
\includegraphics[width=0.475\textwidth,angle=0]
{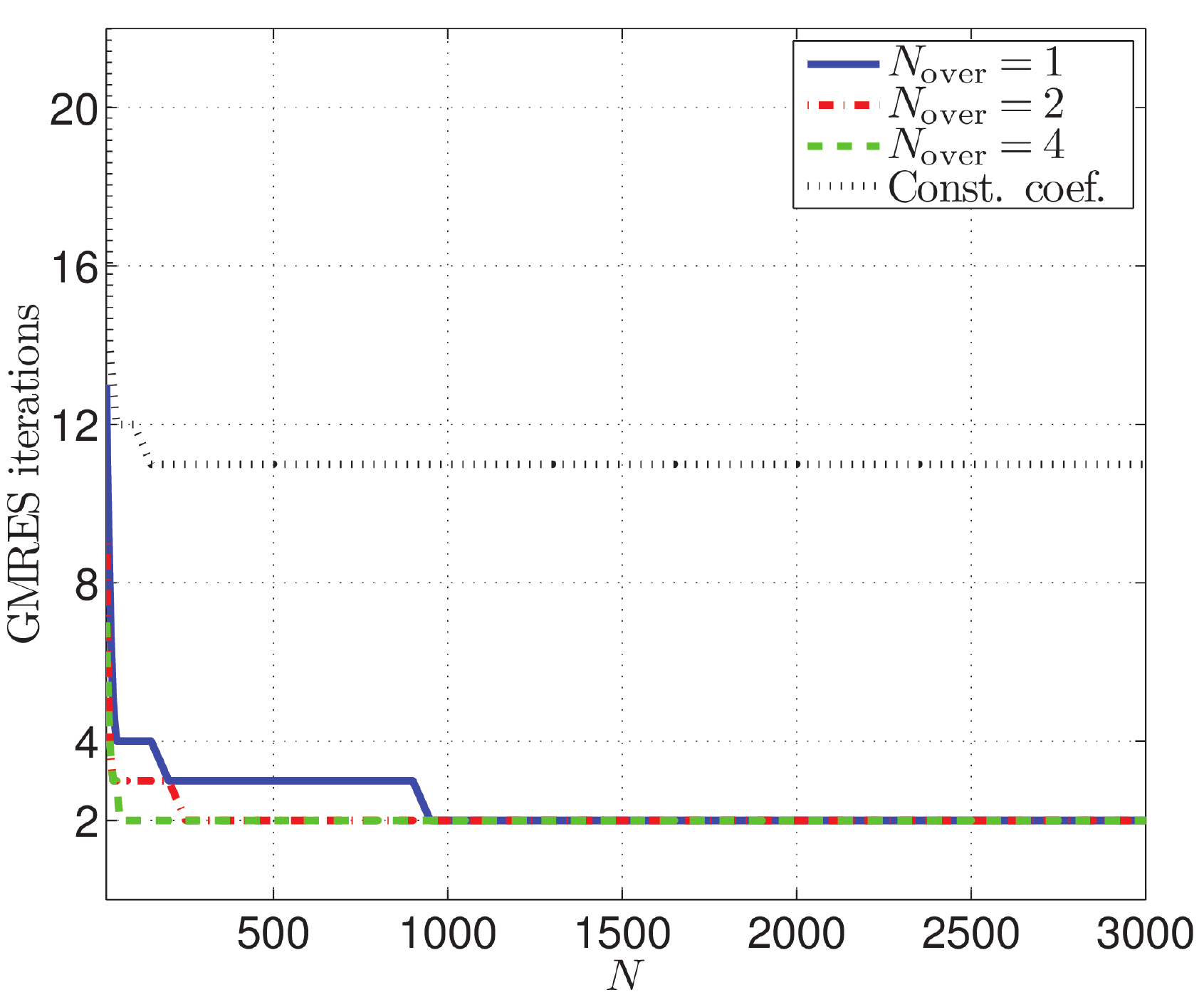}
\end{array}$ 
\caption{Number of GMRES iterations required for the solution of the
  linear system~(\ref{eq:ode-lin-sys}) as a function of the number
  $\Nodes$ of grid points in the one-dimensional mesh, using the GMRES
  residual tolerances $tol_{\mathrm{GMRES}}=10^{-15}$ (left) and
  $10^{-10}$ (right), for several finite-difference preconditioning
  oversampling numbers $\Nover$ and a for a simple preconditioner
  based on a constant-coefficient approximation.}
\label{fig:GMRES-iter}
\end{center} 
\end{figure}

Figure~\ref{fig:GMRES-iter} displays the number of iterations required
by the GMRES-based preconditioned iterative ODE solver described above
in this section to reach a residual tolerance $tol_{\mathrm{GMRES}}$,
for various values of the preconditioner oversampling parameter
$\Nover$. For these examples a two-point boundary value problem
\eqref{eq:pb-orig}--\eqref{eq:pb-orig_2} was considered in the
interval $(a,b)=(0,1)$ with coefficients given by $p(x)=24x/(1+4x^2)$
and $q(x)=(1+8x^3)/(1+4x^2)$, and with right-hand side and the
Dirichlet boundary values such that the function $u(x)=\cos(x^2+2)$ is
an exact solution, and the FC parameters $\Nd=\lfloor
26(1+(\Nodes-21)/100)\rfloor$ (see Section~\ref{param_selc}),
$\Ndelta=10$ and $\Ndeg=5$ were used.

These images demonstrate the GMRES convergence character of the
preconditioned FC-based periodic ODE solver: the iteration numbers
remain bounded independently of the grid size, and they decrease as
more accurate preconditioners are used. For comparison purposes,
results obtained through preconditioning by means of the constant
coefficient ODE solver~\cite{bruno10} (with constant coefficients
equal to the average of the given variable coefficients) are shown,
the corresponding results are marked as ``Const. coef.'' in
Figure~\ref{fig:GMRES-iter}. Clearly, in both plots, the
constant-coefficient preconditioner is less efficient than the
oversampled second-order FD preconditioners.

\subsection{Parameter selection for the scaled FC(Gram) algorithm of
  Section~\ref{scaled_fc_gram}\label{param_selc}}
The exterior-source procedure for enforcement of boundary conditions
introduced in Section~\ref{sec:equiv-for}, requires evaluation of two
right-hand sides, $g_{a},\,g_{b}\in\mathsf{B}(\tildea,\tildeb)$, with
support outside the interval $(a,b)$ (see
Section~\ref{sec:equiv-for}).  Plots in
Figure~\ref{fig:FC-gram-scaled} show our selections for the functions
$g_{a}$ and $g_{b}$ with $\Nd=26$ (left) and $\Nd=260$ (right), taking
in both cases an $\Nodes=1000$ point discretization of the interval
$(a,b)$.
\begin{figure}[!hb] 
\begin{center} 
$\begin{array}{cc}
\includegraphics[width=0.467\textwidth,angle=0]{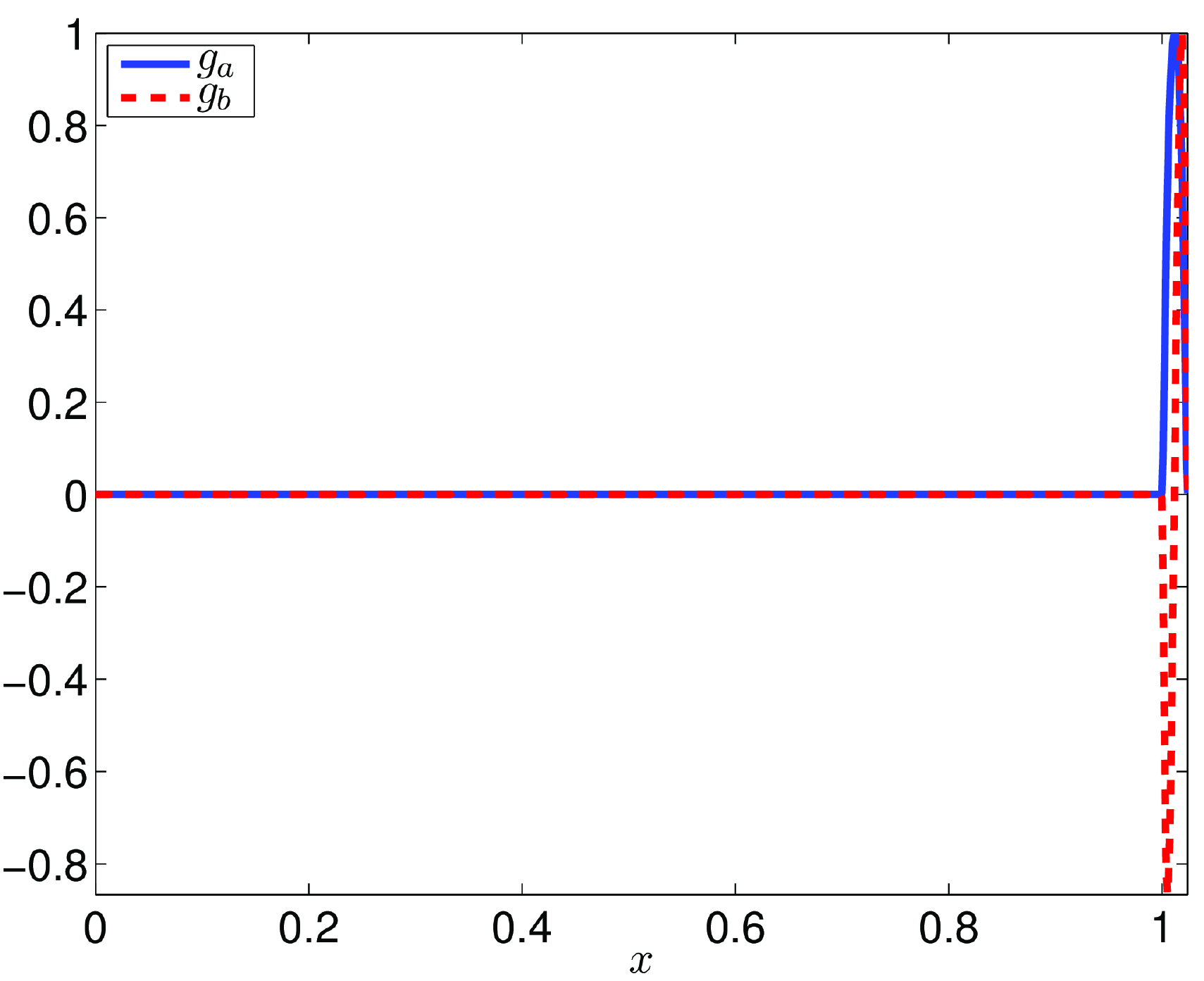}  
\includegraphics[width=0.493\textwidth,angle=0]{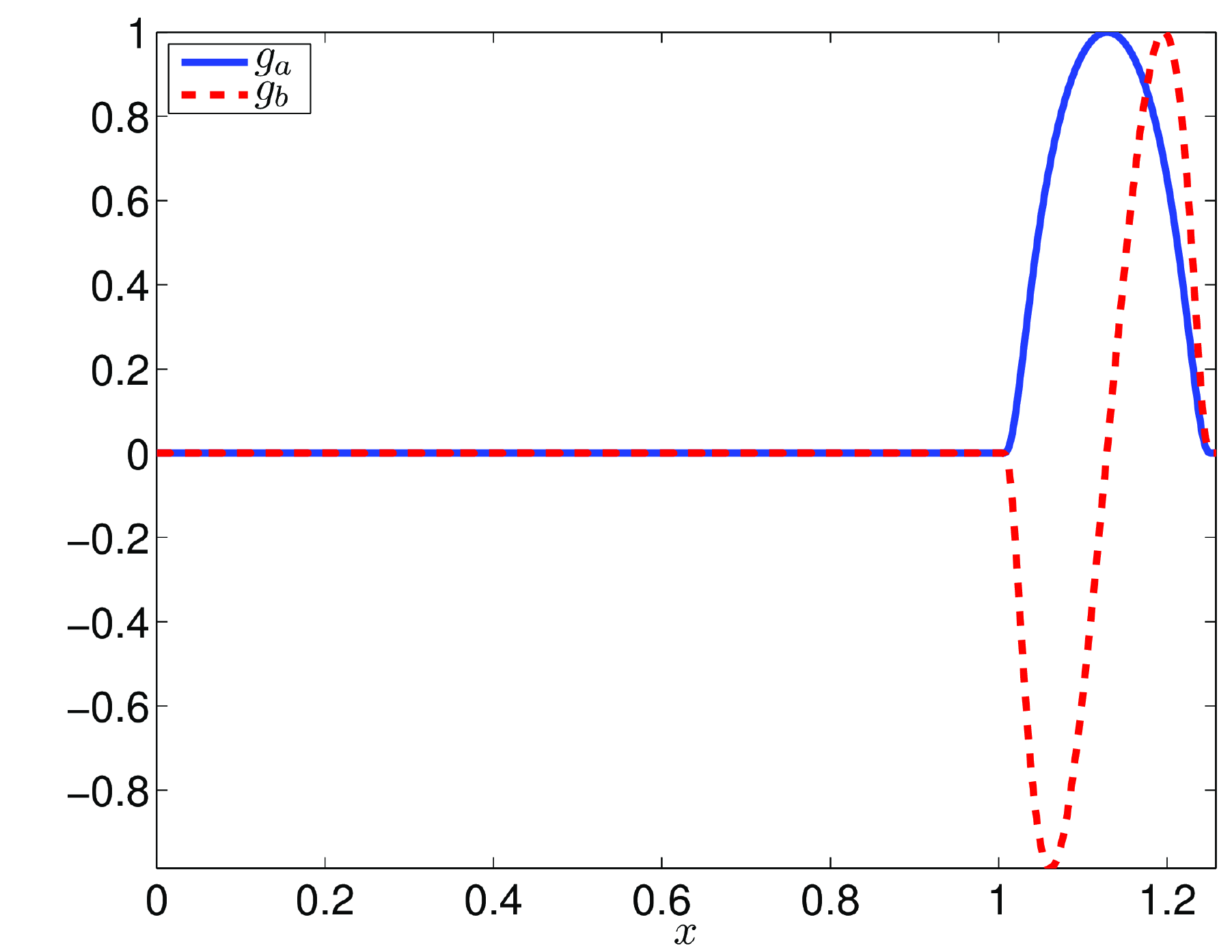} 
\end{array}$ 
\caption{Auxiliary right-hand sides $g_{a}$, $g_{b}$ used in the
exterior-source procedure for $\Nd=26$ (left) and $\Nd=260$ (right).}
\label{fig:FC-gram-scaled}
\end{center} 
\end{figure}
As demonstrated on the left portion of
Figure~\ref{fig:FC-gram-scaled}, if the number $\Nd$ of grid points
outside the interval $(a,b)$ is fixed, then the size of the support of
the functions $g_{a}$ and $g_{b}$ tends to zero and, thus, although
the maximum norm of these functions equal one for all
values of $\Nd$, both functions tend to zero in the
root-mean-square norm as $\Nodes\to\infty$.  Thus, for fixed $\Nd$,
the solutions $u_a$ and $u_b$ introduced in
Section~\ref{sec:equiv-for}, and thus the matrices of the
systems~\eqref{eq:linear-sys-bc-dis-1}, tend to zero as
$\Nodes\to\infty$.  Clearly, then, the
exterior-source method put forth in Section~\ref{sec:equiv-for}
requires that $\Nd\to \infty$ at least linearly with $\Nodes$ as
$\Nodes\to\infty$.

\begin{figure}[!hb] 
\begin{center} 
$\begin{array}{lr}
\includegraphics[width=0.475\textwidth,angle=0]{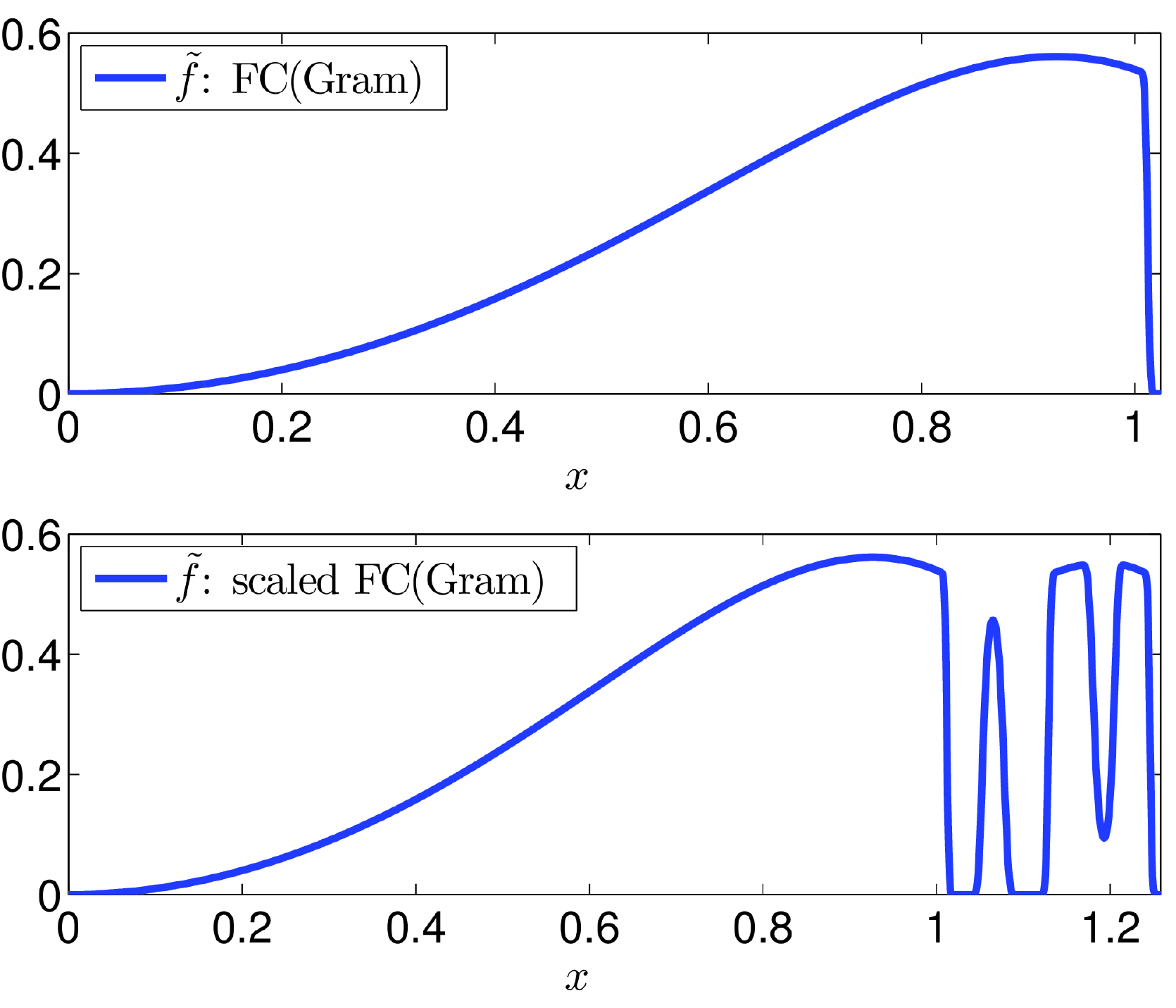} & 
\includegraphics[width=0.48\textwidth,angle=0]{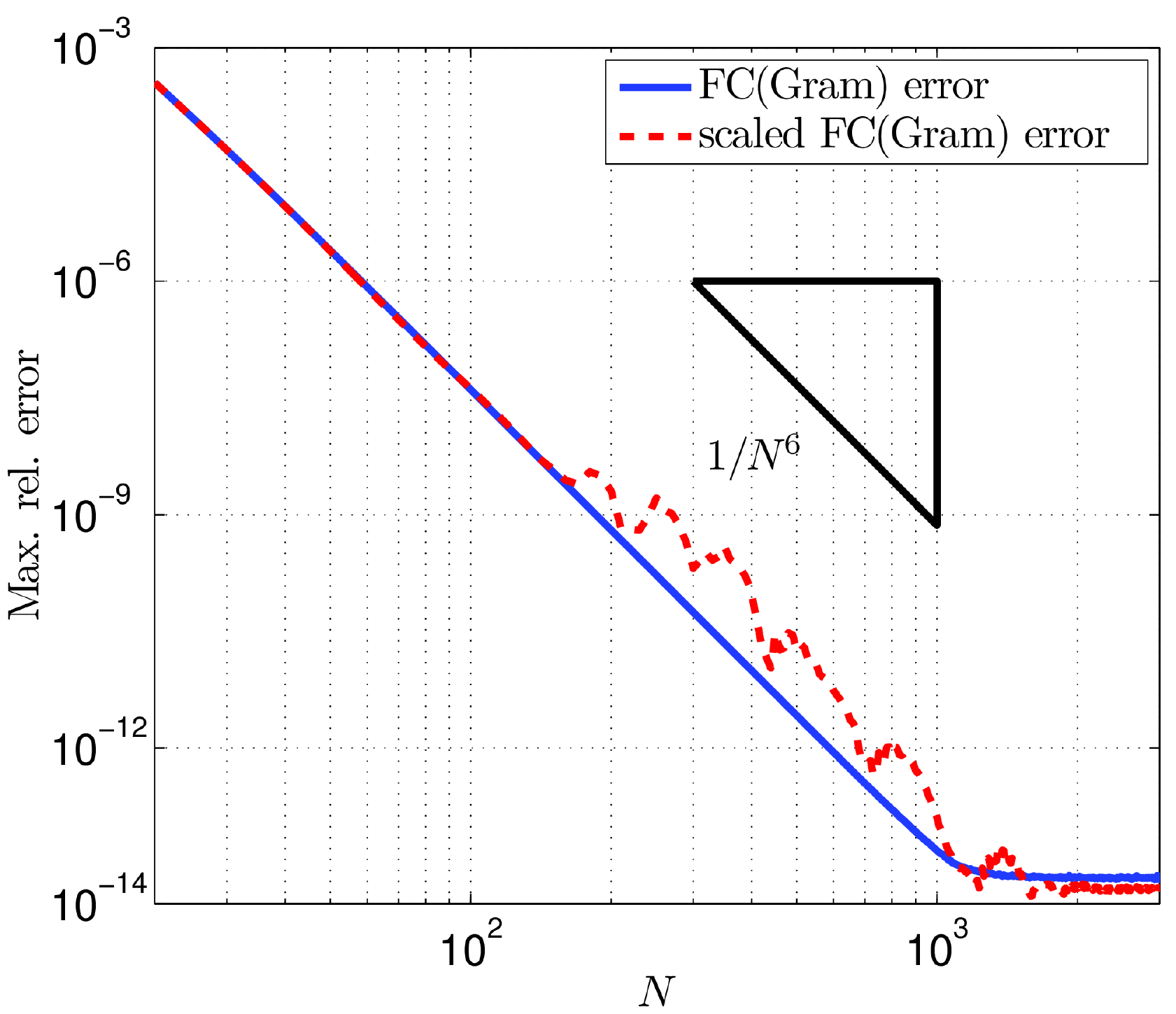}  
\end{array}$ 
\caption{Fourier continuation of the function $f(x)=x^2\cos(x^2)$ in
  the interval $(0,1)$ as produced by the original FC(Gram) method
  with $\Nd=26$ (upper-left), its scaled version using $\Nd=260$
  (lower-left), and the corresponding maximum relative errors
  (relative to the maximum value of the original function) as a
  function of $\Nodes$ (right). As demonstrated in
  Figure~\ref{fig:ODE-error}, the scaled algorithm is an essential
  element of the two-point boundary value solver introduced in
  Section~\ref{sec:ODE-theory}.}
\label{fig:bc-comparison-scaled}
\end{center} 
\end{figure}

The scaled FC(Gram) method introduced in Section~\ref{scaled_fc_gram}
was designed precisely to allow for evaluation of Fourier continuation
functions within a context that includes use of variable values of the
parameter $\Nd$, while still allowing for re-use of the basic matching
functions $\left\{f^r_\mathrm{even} \right\}_{r=0}^{\Ndeg}$ and
$\left\{f^r_\mathrm{odd}\right\}_{r=0}^{\Ndeg}$ defined in
Section~\ref{fc_gram}. The upper and lower left portions of
Figure~\ref{fig:bc-comparison-scaled} display the un-scaled and scaled
continuations of the function $f(x) = x^2\cos(x^2)$ using
\begin{equation}\label{floor}
\Nd=\lfloor 26(1+(\Nodes-21)/100)\rfloor,
\end{equation} 
where $\lfloor s \rfloor$ denote the largest integer less than or
equal to $s$.  The right portion of
Figure~\ref{fig:bc-comparison-scaled} demonstrates the fact that,
provided values of $\Nd$ proportional to $\Nodes$ are used (as is done
here, according to equation~\eqref{floor}), the relative error of the
scaled FC continuations (see Remark~\ref{error_eval}) decays at the
same rate as its non-scaled counterpart as $\Nodes$ tend to
infinity.  Although the scaled version is somewhat less accurate than
the original FC(Gram) continuations proposed in~\cite{bruno10}, the
scaling to the interval $[1,1+g]$ implicit in equation~\eqref{floor}
(see Section~\ref{scaled_fc_gram}) can be implemented through use of a
small set of precomputed scaled matching functions.
\begin{remark}\label{error_eval}
  Throughout the present Section~\ref{ode_impl}, the relative error of
  scaled and unscaled FC(Gram) continuations of given functions has
  been evaluated as the difference between the corresponding FC
  approximation and the corresponding exact function values over an
  equispaced grid ten times finer than the grid used in the FC
  procedure. (We have checked that, for the functions under
  consideration, use of higher grid refinement rates leads to
  essentially unchanged error estimates.) Clearly it is necessary to
  use such oversampling, as, although not exactly interpolatory, the
  FC procedure does arise from a least-squares approximation of
  function values (at $\Ndelta$ collocation points next to each one of
  the interval endpoints), and, thus, evaluation of errors solely at
  collocation points tend to produce error under-estimates. It is
  reasonable to expect (and we have verified in practice) that this
  difficulty does not arise when FC solutions of ODEs are concerned,
  since the FC BVP solution procedure does not seek to minimize error
  in function values at the collocation points. In spite of this fact,
  and for consistency, the solution errors presented in
  Figures~\ref{fig:ODE-error} and~\ref{fig:accuracy-grid} were
  evaluated on a grid resulting from ten-fold refinement.
\end{remark}
\begin{remark} The right-hand image in
  Figure~\ref{fig:bc-comparison-scaled} demonstrates that the order of
  convergence of the scaled FC(Gram) method is consistent with the
  order obtained by the original un-scaled procedure, in this case
  $1/\Nodes^6$, and that both algorithms reach round-off machine error
  levels for approximately the same values of $\Nodes$. However, it is
  clear from this figure that the convergence of the scaled FC(Gram)
  algorithm is somewhat more irregular than that of the unscaled
  procedure. The fluctuating behavior observed in
  Figure~\ref{fig:bc-comparison-scaled} and accompanying accuracy loss
  (that amounts to as much as one digit for some values of $\Nodes$ in
  this example, but does not otherwise detract from the overall
  convergence rate), can be traced to the fact that the scaled
  algorithm incorporates the
  composition~\eqref{scaled_f_match}--\eqref{xi_def} and, thus,
  associated frequency content, as well as the discontinuous
  selection~\eqref{floor} of the number $\Nd$ of continuation points
  as a function of $\Nodes$.
\end{remark}

\subsection{Accuracy of FC BVP solver: scaling and asymptotics}
\label{sec:accuracy-solver}
\subsubsection{Scaling and convergence of the exterior-source FC BVP
  solver \label{scale_conv}}
This section demonstrates the effectiveness of the scaling procedure
described in Section~\ref{scaled_fc_gram} with parameters as indicated
in Sections~\ref{gmres} and~\ref{param_selc}; for the purposes of this
demonstration the two-point boundary value problem introduced in
Section~\ref{gmres} is considered.  The left portion of
Figure~\ref{fig:ODE-error} displays in a solid blue line the maximum
ODE solution error (relative to the maximum value of the ODE solution)
based on use of the un-scaled FC(Gram) algorithm with a continuation
region containing a fixed number $\Nd=26$ of discretization
points. For comparison purposes, the left-hand figure also displays
the corresponding maximum relative error arising in the Fourier
continuation of the ODE right-hand side using the same continuation
region. Clearly, while use of a continuation region that does not grow
with $\Nodes$ does not affect the accuracy of the the FC(Gram)
procedure, it does affects greatly the accuracy of the ODE solver:
convergence beyond rather small values of $N$ is not observed for the
un-scaled ODE solution.  However, as demonstrated in the right portion
of Figure~\ref{fig:ODE-error}, use of the scaled FC(Gram) method
(with, e.g., $\Nd$ given by equation~\eqref{floor}), restores
convergence to machine precision.  Notice that the error in the ODE
solver is a quantity of order $(1/\Nodes)^{\Ndeg + 1}$, where $\Ndeg$
is the maximum polynomial degree in the Gram polynomial basis used for
the scaled FC(Gram) procedure.

\begin{figure}[!ht] 
\begin{center} 
$\begin{array}{cc}
\includegraphics[width=0.475\textwidth,angle=0]
{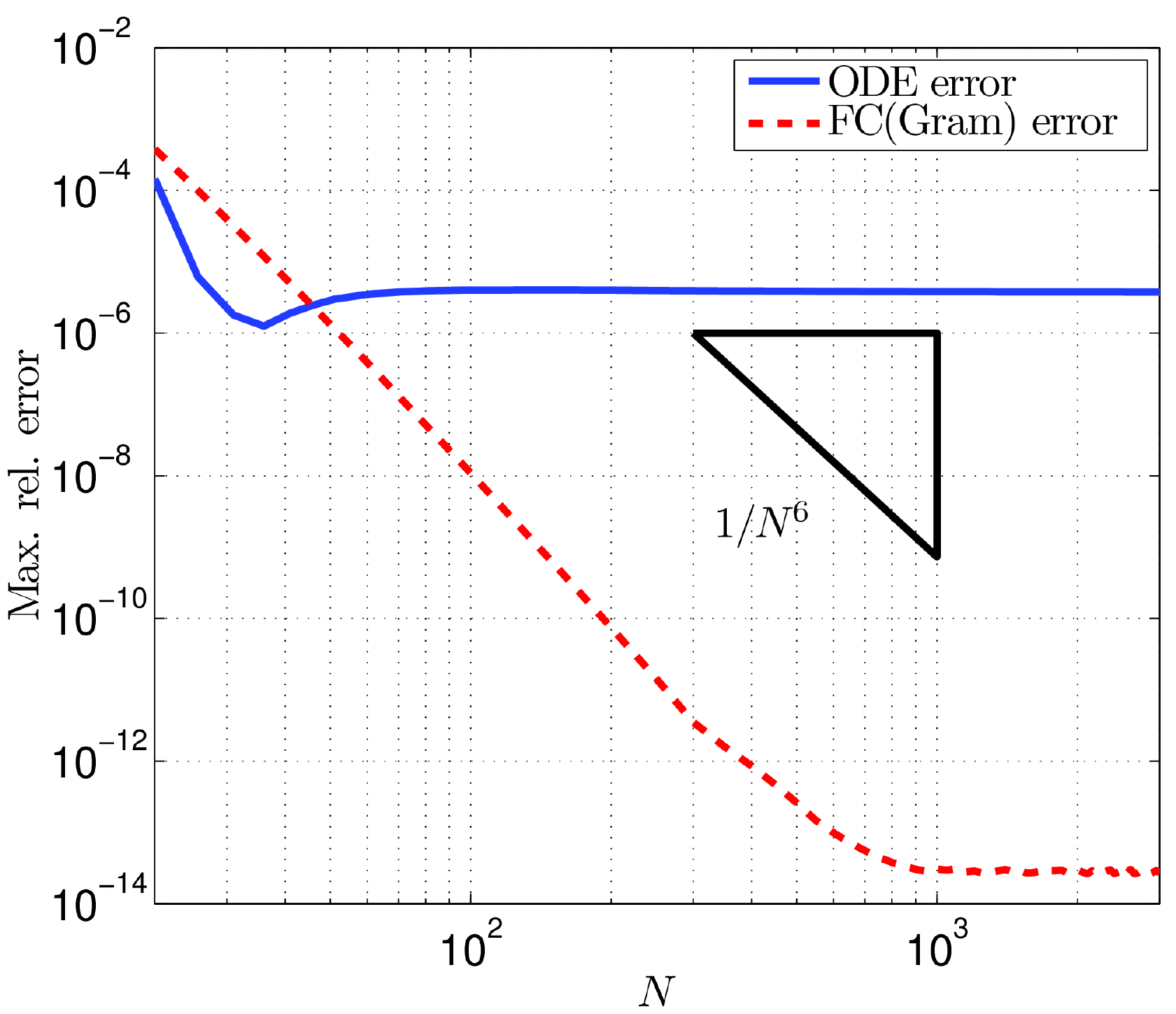} & 
\includegraphics[width=0.475\textwidth,angle=0]
{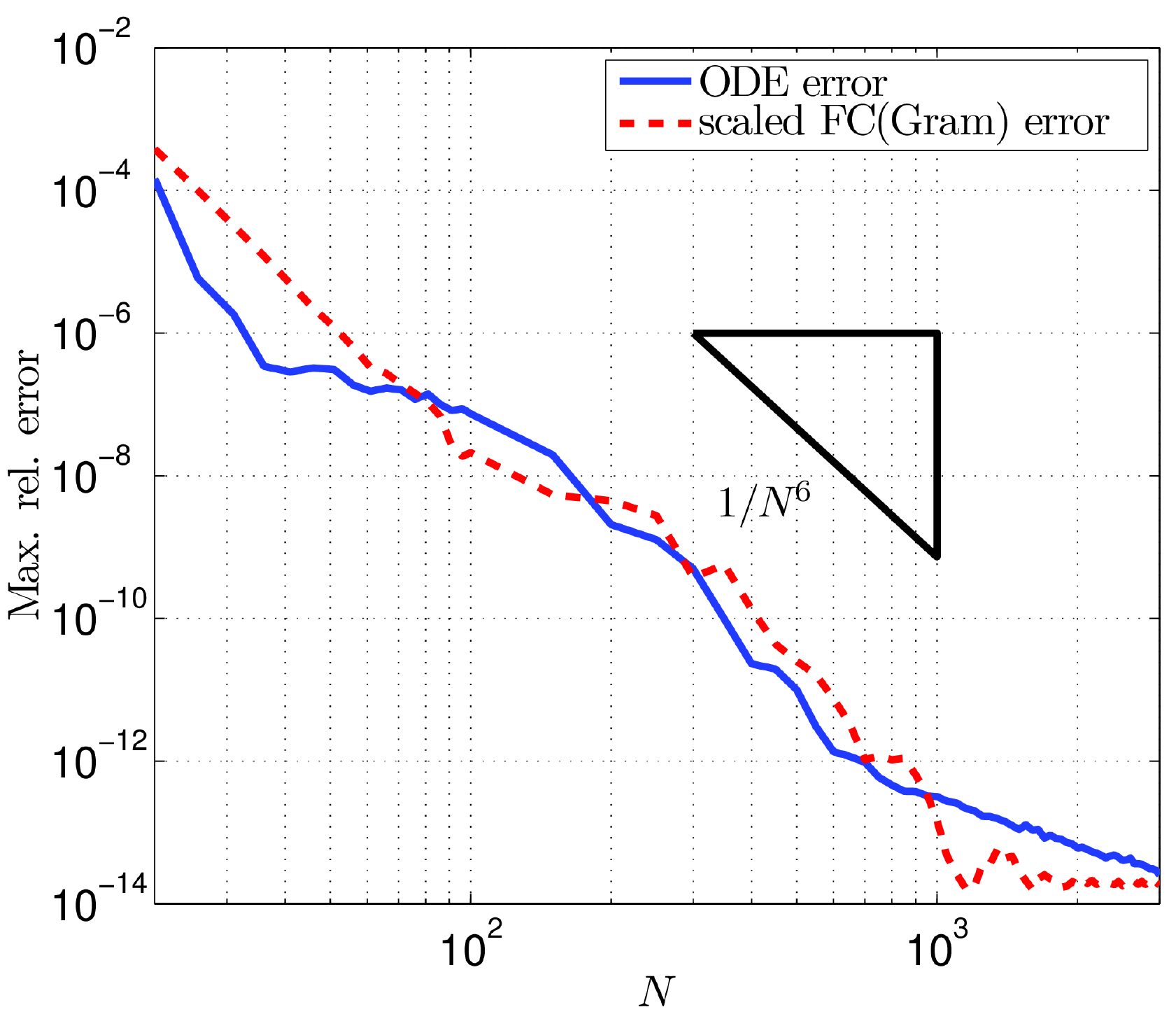} 
\end{array}$ 
\caption{Maximum error (relative to the solution maximum) in the
  FC(Gram) approximation of the ODE right-hand side (denoted by
  ``FC(Gram) error'' in the figure) and corresponding FC ODE solution
  error for the two-point boundary value problem introduced in
  Section~\ref{gmres}.  Left: un-scaled procedure. Right: scaled
  version.}
\label{fig:ODE-error}
\end{center} 
\end{figure}

\begin{figure}[!ht] 
\begin{center}
$\begin{array}{cc}
\includegraphics[width=0.475\textwidth,angle=0]
{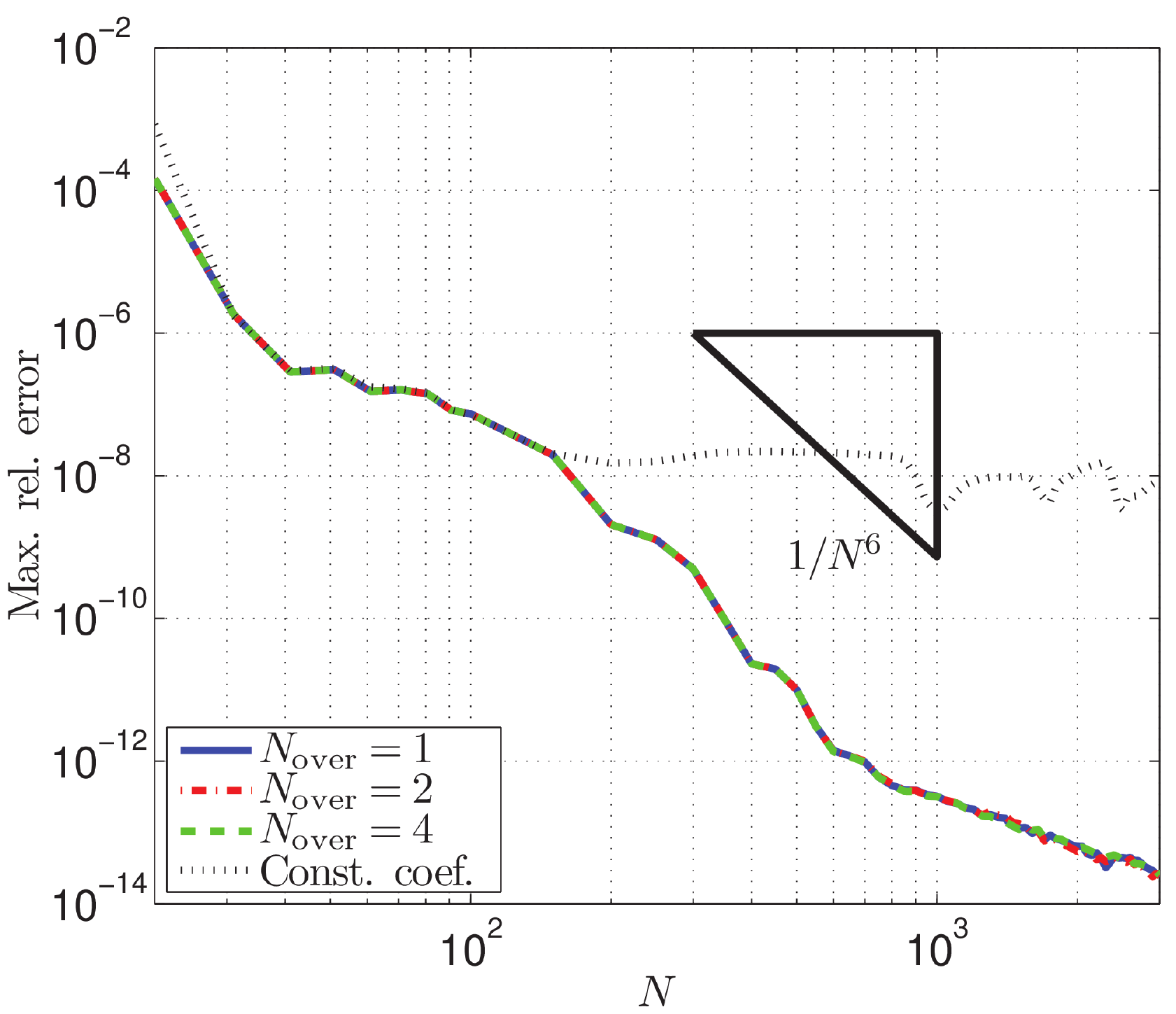} &
\includegraphics[width=0.475\textwidth,angle=0]
{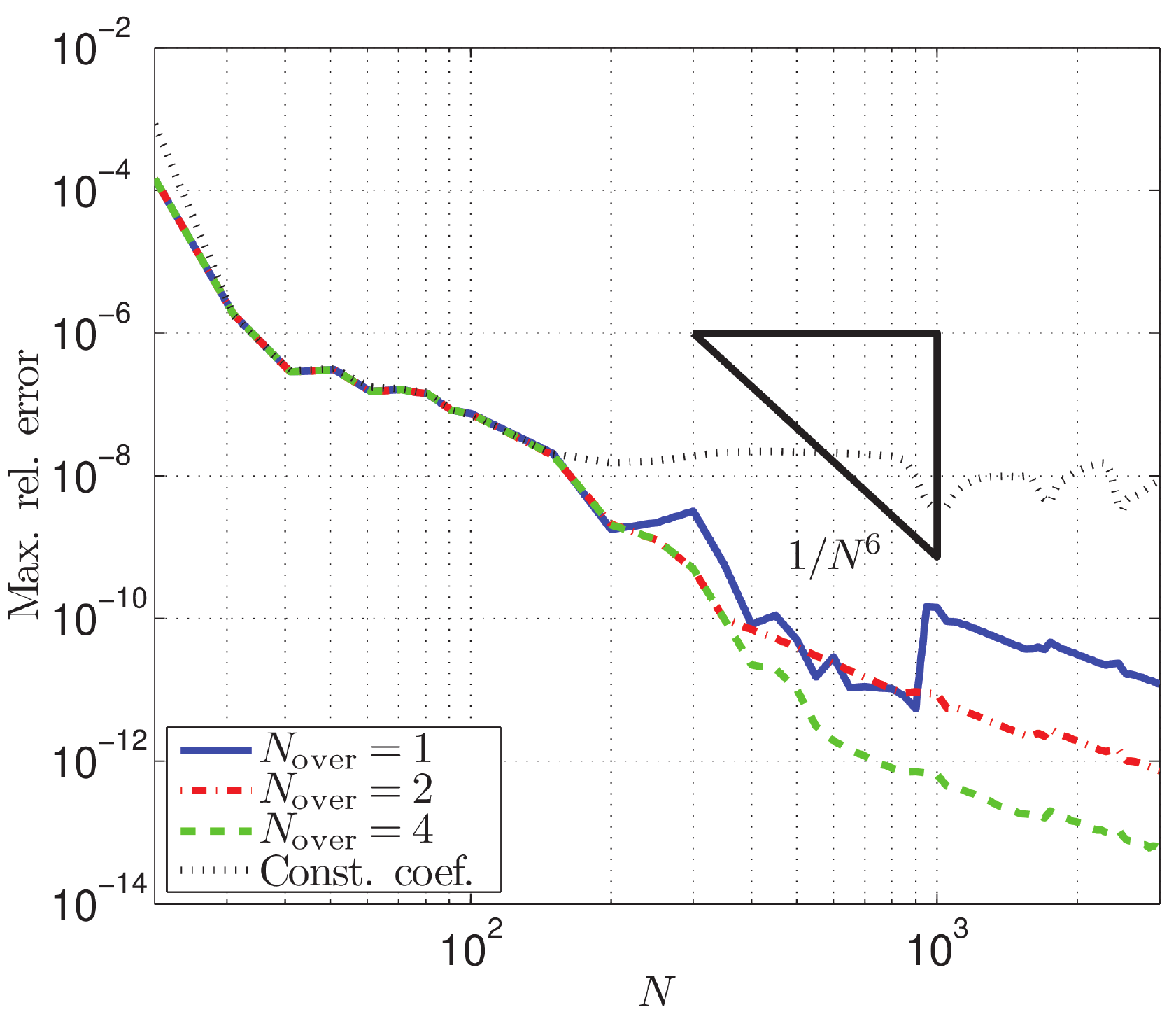} 
\end{array}$ 
\caption{Maximum error (relative to the solution maximum) in the FC
  solution of the two-point boundary value problem considered in
  Section~\ref{scale_conv}, as a function of the number of
  discretization points used, for two values of the GMRES residual
  tolerance: $tol_{\mathrm{GMRES}}=10^{-15}$ (left) and $10^{-10}$
  (right).}
\label{fig:accuracy-grid}
\end{center}  
\end{figure}

Figure~\ref{fig:accuracy-grid} demonstrates the dependence of the
numerical accuracy on the preconditioner used. For $\Nodes\leq 700$
all the preconditioners considered give rise to similar performance,
but larger value of the oversampling parameter $\Nover$ do give rise
to improved accuracies for the $10^{-10}$ GMRES tolerance and for 
$\Nodes \geq 1000$.  

\subsubsection{Convergence for stiff problems: exterior-sources and
  asymptotic-matching\label{stiff}}

This section demonstrates that an adequate combination of the scaling
and asymptotic methods described in Sections~\ref{sec:equiv-for}
and~\ref{asympt_exp} gives rise to a robust and accurate solver for
{\em arbitrarily stiff} two-point boundary value problems under
consideration (see Remark~\ref{eps_delta_t}).  For this demonstration
we consider the variable-coefficient two-point 
BVP~\eqref{eq:pb-orig}--\eqref{eq:pb-orig_2} in the interval
$(0,1)$ with $f=0$ and with
$$
p(x)=\varepsilon^2 \frac{2(1+2x)}{0.5\log(1+2x)+1}\quad\mbox{and}
\qquad q(x)=\varepsilon^2 \frac{(1+2x)^2}{0.5\log(1+2x)+1}.
$$
The exact solution of this problems is given by
$$
u(x)=C_{1}(\varepsilon)\mathrm{Ai}\left(
\varepsilon^{-\frac{2}{3}}\left(\frac{1}{2}\log(2x+1)+1\right)
\right)
+C_{2}(\varepsilon)\mathrm{Bi}\left(
\varepsilon^{-\frac{2}{3}}\left(\frac{1}{2}\log(2x+1)+1\right)
\right),
$$
where $\mathrm{Ai}$ and $\mathrm{Bi}$ are the Airy functions of first
and second kind, respectively, and where $C_{1}(\varepsilon)$ and
$C_{2}(\varepsilon)$ are chosen in such a way that $u(0)=1$ and
$u(1)=0$. The solution $u$ has a boundary layer at the endpoint $x=0$
for small values of $\varepsilon$.

The left portion of Figure~\ref{fig:ODE-boundary-layer} displays the
boundary layer solution $u(x)$ for three values of $\varepsilon$.  The
right image in the same figure presents the maximum relative errors
that result from use of the exterior-source procedure
(Section~\ref{sec:equiv-for}) and asymptotic-matching expansions
(Section~\ref{asympt_exp}) of orders $k=1$, $2$ and $3$.  As expected,
the exterior-source and asymptotic-matching procedures are both
accurate within their intended realms of applicability ($\varepsilon$
bounded away from zero and $\varepsilon\to 0$, respectively), but they
both break down otherwise. Hence, our algorithmic strategy, which
relies on one procedure or the other, depending on the value of
$\varepsilon$, results in accurate approximations for all values of
$\varepsilon>0$.

\begin{figure}[!htb]
\begin{center}
$\begin{array}{cc}
\includegraphics[width=0.480\textwidth,angle=0]
{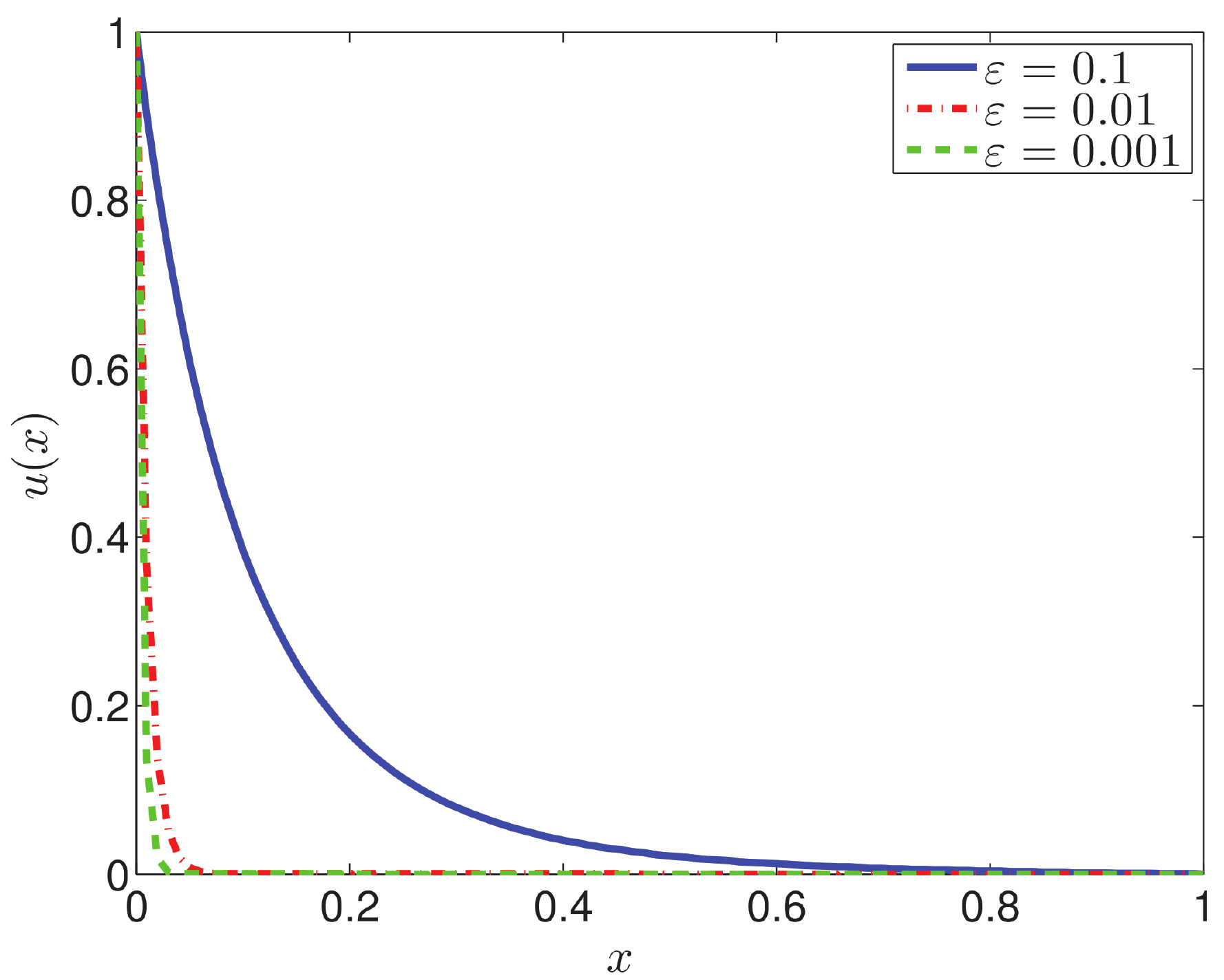}& 
\includegraphics[width=0.465\textwidth,angle=0]
{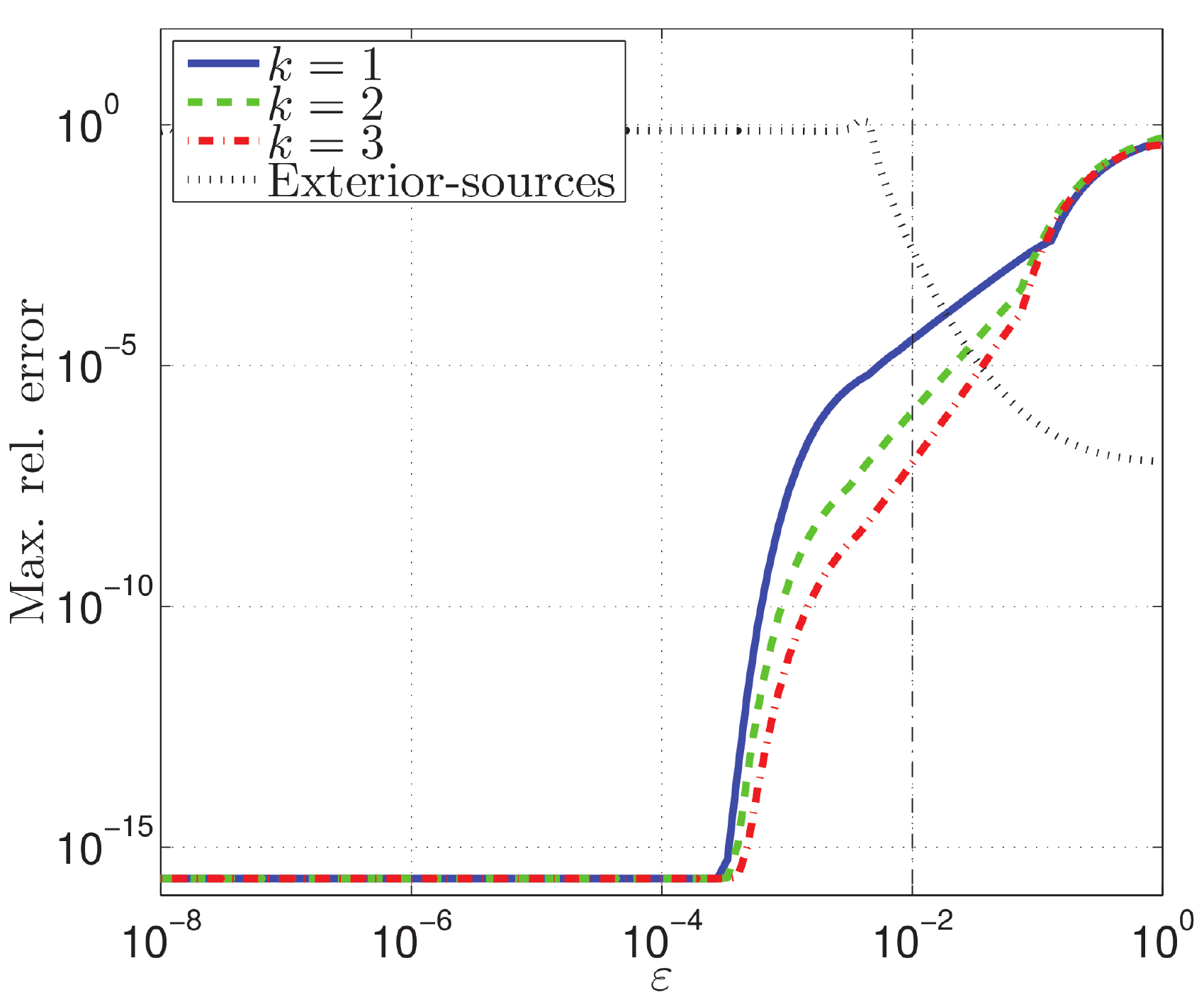} 
\end{array}$ 
\caption{Boundary layer solution with small parameters
  $\varepsilon=0.1,\,0.01,\,0.001$ (left), and maximum relative
  errors, as a function of $\varepsilon$, resulting from use of the
  asymptotic-matching expansions of order $k=1,\,2,\,3$ as well as
  the exterior-source method with $h = 10^{-2}$ (right). The dash-dot
  vertical line on the right plot is located at
  $\varepsilon=h=10^{-2}$.}
\label{fig:ODE-boundary-layer}
\end{center}
\end{figure}

\begin{remark}\label{rem:detect_boundary_layer}
  The adequate detection of the boundary layer in terms of the small
  parameter $\varepsilon$, and, thus, the evaluation of the threshold
  value that defines the limit between the $(h,\varepsilon)$ regions
  for which the exterior-source and asymptotic-matching procedures are 
  applied depends on each particular BVP. As a rule of thumb the
  threshold limit can be fixed to $\varepsilon=h$ (see
  Figure~\ref{fig:ODE-boundary-layer}).
  This selection (with $\varepsilon = \sqrt{\Delta t/2}$ and
  $\varepsilon = \Delta t$ for the heat and wave equations,
  respectively) was used in all of the numerical experiments presented
  in this paper. Section~\ref{diff_numer} and, in particular,
  Table~\ref{tab:cputime}, provide an indication of the computing
  costs required by the dual exterior-source/asymptotic-matching
  boundary-condition strategy we use.
\end{remark}

\subsection{Computational cost of the FC BVP solver\label{ODE_cost}}
The computational cost of the proposed FC-AD method for the solution
of time-dependent problems depends linearly on the cost of the
evaluations of the ODE solution and evaluation mappings $A_{\Nodes}$
and $B_{\Nodes}$ defined in (\ref{eq:def_A_N}) and (\ref{eq:def_B_N}),
respectively. Thus, an efficient numerical implementation of these
operators translate into corresponding efficiencies for the resulting
overall FC-AD time-marching scheme.

The evaluation of the mapping $A_{\Nodes}$ comprises two main
components, namely 1)~Fourier continuation (as described in
Section~\ref{scaled_fc_gram}) of the variable coefficients and the
right-hand side, and 2)~Numerical solution of the resulting linear
system (\ref{eq:ode-lin-sys}). In view of Remark~\ref{FCGram_cost},
point~1) requires $\mathcal{O}(\Nodes\log\Nodes)$ operations. With
regards to point~2), on the other hand, we note that every iteration
of the GMRES method involves one evaluation of the discrete ODE
operator $B_{\Nodes}$ and one evaluation of the
preconditioner. Inspection of Algorithm~\ref{algor:ode_eval} and the
Finite Difference preconditioning algorithm presented in
Section~\ref{gmres} therefore shows that the overall cost of the of
the ODE solver is $2\mbox{FFT}(\Nodes+\Nd-1)
+\Niter\mbox{FFT}(\Nover(\Nodes+\Nd-1))
+4\Niter\mbox{FFT}(2(\Nodes+\Nd-1)) +O(\Nover(\Nodes+\Nd-1))$, where
$\mbox{FFT}(M)$ denotes the number of operations required to evaluate
an FFT of size $M$. In brief, the ODE solver runs in
$\Niter\mathcal{O}(\Nodes\log\Nodes)$ operations.

\section{Full FC-based PDE solver}
\label{sec:full-disc}
Based on algorithms introduced above in this text, this section
introduces FC-based alternating-direction solvers for diffusion and
wave propagation PDEs with variable coefficients.

\subsection{Overall spatial discretization\label{over_spat_disc}}
Let $\Omega$ be a two-dimensional spatial domain with a piecewise
smooth boundary which, without lost of generality, we assume is
contained in the rectangle $R=[0,L^{H}]\times [0,L^{V}]$.  In order to
produce a spatial discretization of $\Omega$, the rectangle $R$ itself
is discretized by means of a uniform Cartesian grid $\Omega_{h}$ given
by
$$
\Omega_{h}=\Omega\cap\{(x_{i},y_{j})=((i-1)h^{H},(j-1)h^{V}),
\ i=1,\ldots,M^{H},\ j=1,\ldots,M^{V}\},
$$
where, for some integers $M^{H},M^{V}>1$ we have set
$h^{H}=L^{H}/(M^{H}-1)$ and $h^{V}=L^{V}/(M^{V}-1)$.  For the sake of
simplicity we assume that each grid line only crosses the boundary
$\partial\Omega$ twice, and we denote by
\begin{align*}
  \{a_{j}^{H},b_{j}^{H}\}=\{x:(x,y_{j})\in\partial\Omega\}& \qquad
  (j=1,\ldots,M^{V}),\quad\mbox {and} \\
  \{a_{i}^{V},b_{i}^{V}\}=\{y:(x_{i},y)\in\partial\Omega\}& \qquad 
  (i=1,\ldots,M^{H}),
\end{align*}
the points of intersection of $\partial \Omega$ with horizontal and
vertical Cartesian lines (shown as green circles and yellow squares
in Figure~\ref{fig:sweep}), respectively. (Generalization to cases for
which more than two intersections occur for some Cartesian lines is
straightforward.) The horizontal and vertical discretization lines
within $\Omega$ and the corresponding sets of indexes, in turn, are
given by
\begin{align*}
  & P^{H}_{j}=\{(x,y_{j})\in\Omega_{h}\} \quad ; \quad 
  I_{j}^{H}=\{i\in\mathbb{N}:(x_i,y_j)\in\Omega_{h}\}
  \qquad (j=1,\ldots,M^{V}),\qquad \mbox{and}\\
  & P^{V}_{i} = \{(x_{i},y)\in\Omega_{h}\}\quad ; \quad 
  I_{i}^{V}=\{j\in\mathbb{N}:(x_i,y_j)\in\Omega_{h}\}
  \qquad (i=1,\ldots,M^{H}).
\end{align*}

Each alternating-direction half-step in the time-marching algorithm
presented in Section~\ref{sec:time-disc} consists of a horizontal
sweep (steps (D2) and (W2)) and a vertical sweep (steps (D4) and
(W3)). For each horizontal (resp. vertical) sweep, the solver requires
solution of a Dirichlet two-point BVP along each vertical
(resp. horizontal) line $P^{V}_{i}$ (resp. $P^{H}_{j}$). The solution of
each Dirichlet two-point BVP, in turn, involves application of the
operators~\eqref{eq:def_B_N} (evaluation of the ODE right-hand side)
and~\eqref{eq:def_A_N} (BVP solution) (see also Section~\ref{bound_proj} 
and Remarks~\ref{end_points} and~\ref{AN_BN_dep_on_ab}). 
Sections~\ref{diffusion} and~\ref{wave_prop} provide a detailed 
description of the discrete versions of the horizontal and vertical 
sweeps, for the problems of diffusion and wave propagation, 
in terms of the operators $A_{\Nodes}$ and $B_{\Nodes}$.

\begin{figure}[!ht] 
\begin{center} 
$\begin{array}{c@{\hspace{1.5cm}}c}
\includegraphics[width=0.4\textwidth,angle=0]
{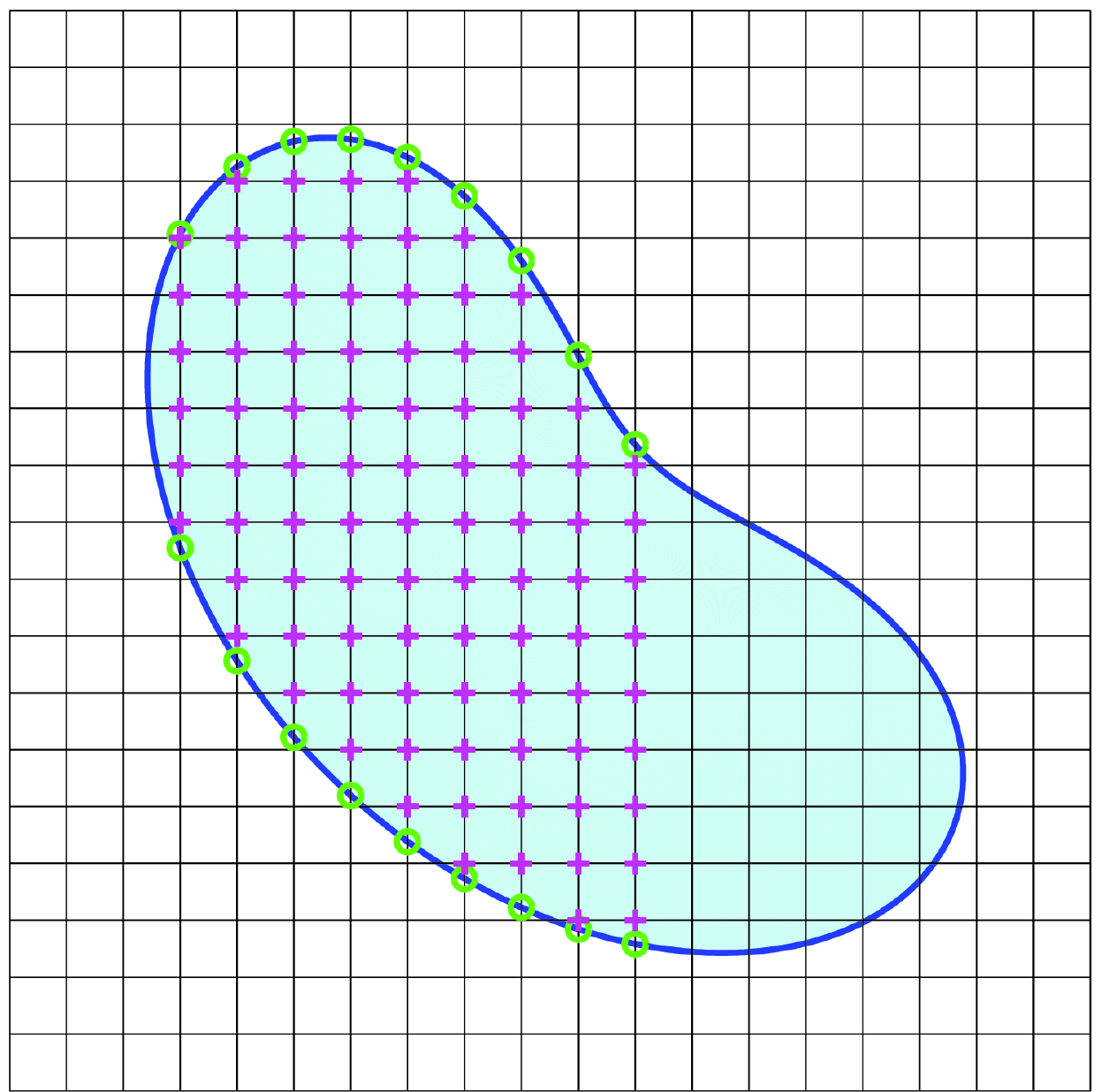} &
\includegraphics[width=0.4\textwidth,angle=0]
{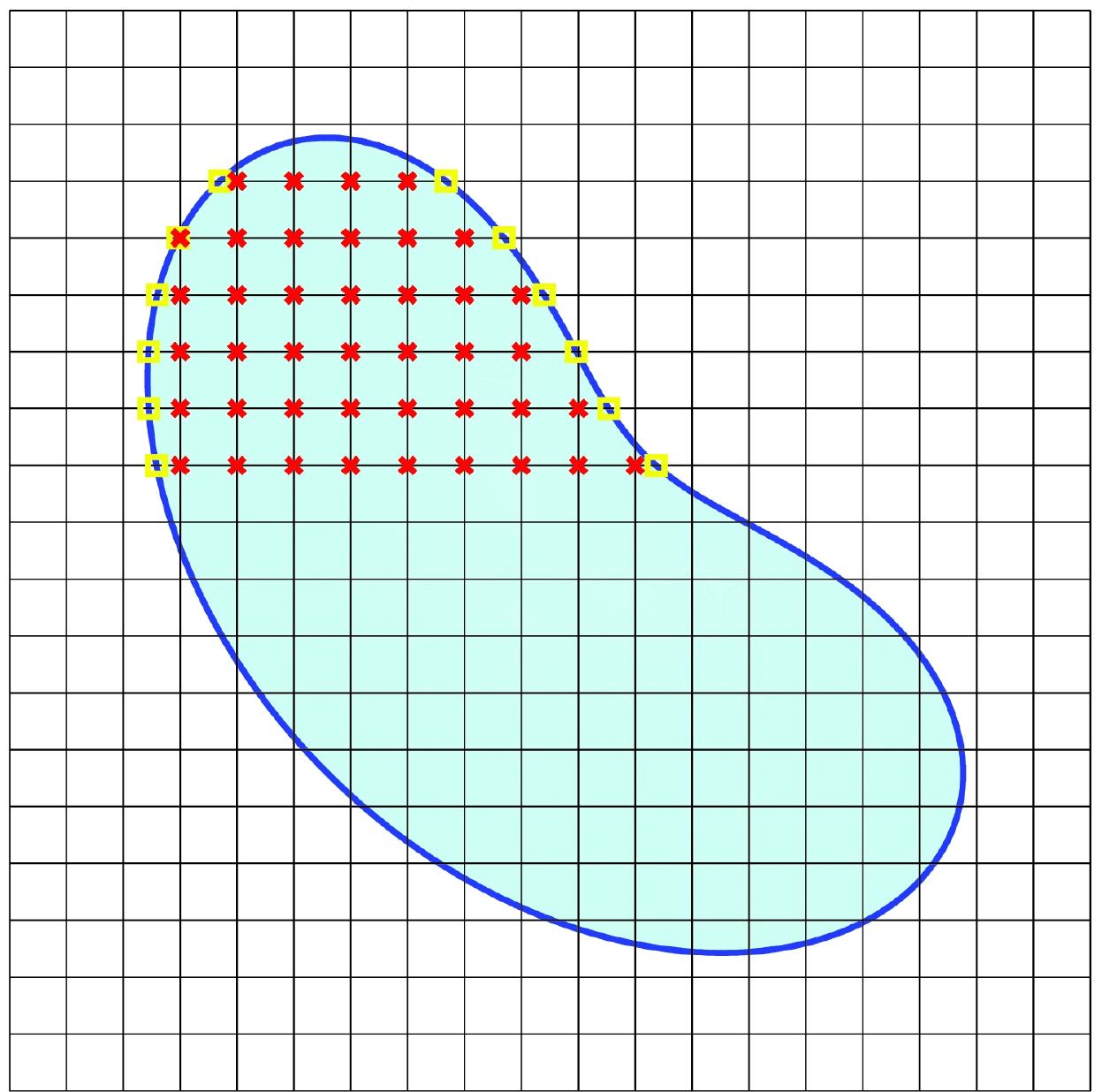} 
\end{array}$ 
\caption{Sweeping procedure used for evaluation of the grid values
  $\boldsymbol{u}^{n+\frac12}_{ij}$ on vertical lines (magenta
  ``$+$''~crosses on the left) and the grid values
  $\boldsymbol{u}^{n+1}_{ij}$ on horizontal lines (red
  ``$\times$''~crosses on the right).}
\label{fig:sweep}
\end{center} 
\end{figure}

\begin{remark}\label{pt_values}
  Throughout Section~\ref{sec:full-disc}, a bold-face symbol such as
  $\boldsymbol{\phi}_{i,j}^{n}$ denotes the value of a grid function
  at a point $(x_i,y_j)\in \Omega_h$ at time $t_{n}$, and for each fixed
  $i$ (resp. for each fixed $j$), we define the vector
  $\boldsymbol{\phi}_{i,\cdot}^{n}=
  (\boldsymbol{\phi}_{i,j}^{n})_{j\in I^{V}_i}$ (resp.
  $\boldsymbol{\phi}_{\cdot,j}^{n}=(\boldsymbol{\phi}_{i,j}^{n})_{i\in
    I^{H}_j}$). In addition, in the case of the diffusion and wave
  problems, with reference to~\eqref{eq:coefODE_heat_def}
  and~\eqref{eq:coefODE_wave_def} respectively we set
  $$
  \boldsymbol{p}^{H}_{i,j}=\mathcal{P}^{H}(x_{i},y_{j}),\quad
  \boldsymbol{p}^{V}_{i,j}=\mathcal{P}^{V}(x_{i},y_{j}),\quad
  \boldsymbol{q}_{i,j}=\mathcal{Q}(x_{i},y_{j}),\quad\mbox{and}\quad
  \boldsymbol{f}_{i,j}^{n}=\mathcal{F}(x_{i},y_{j},t_{n}).   
  $$
  Since originally only the grid values of $\alpha$ and $\beta$ are
  known, the grid values of $\partial_{x}\beta$ and $\partial_{y}\beta$
  needed to evaluate equation $\mathcal{P}^{H}$ and $\mathcal{P}^{V}$ 
  in~\eqref{eq:coefODE_heat_def} and~\eqref{eq:coefODE_wave_def} are 
  approximated at each line by the derivative of its scaled FC(Gram) 
  continuation.
\end{remark}

\subsection{FC-AD diffusion solver\label{diffusion}} 
Taking into account Sections~\ref{sec:diffusion},~\ref{bound_proj}
and~\ref{over_spat_disc} (and, in particular, Remark~\ref{pt_values}
concerning bold-face symbols denoting grid functions and vectors of
grid function values), our full FC-based ADI discrete procedure for
the diffusion problem consists of the following four steps:
\begin{itemize}
\item[(D1)$_\Nodes$] Initialize $\boldsymbol{u}_{i,j}^{0}$ and 
$\boldsymbol{w}_{i,j}^{0}$ on the vertical lines $P_{i}^{V}$ 
($j\in I_{i}^{V},\ i=1,\ldots,M^{H}$),
\begin{equation*}
\left\{
\begin{array}{ll}
  \boldsymbol{u}_{i,j}^{0}=u_{0}(x_{i},y_{j}),\\[0.2cm]
  \boldsymbol{w}_{i,\cdot}^{0}=B_{\Nodes}
  (-\boldsymbol{p}^{V}_{i,\cdot},-\boldsymbol{q}_{i,\cdot})
  \boldsymbol{u}_{i,\cdot}^{0},
\end{array}\right.
\end{equation*}
\end{itemize}
and, for $n=0,1,\ldots,n_{\mathrm{max}}$,
\begin{itemize} 
\item[(D2)$_\Nodes$] compute $\boldsymbol{u}^{n+\frac12}_{i,j}$ on 
the vertical lines $P_{i}^{V}$ ($j\in I_{i}^{V},\ i=1,\ldots,M^{H}$),
\begin{equation*}
\boldsymbol{u}^{n+\frac12}_{i,\cdot}=
A_{\Nodes}(\boldsymbol{p}^{V}_{i,\cdot},\boldsymbol{q}_{i,\cdot})
\begin{pmatrix}
\boldsymbol{w}_{i,\cdot}^{n}+\boldsymbol{f}_{i,\cdot}^{n+\frac14}\\[0.2cm]
g^{n+\frac12}(a_{i}^{V})\\[0.2cm]
g^{n+\frac12}(b_{i}^{V})
\end{pmatrix},
\end{equation*}
\item[(D3)$_\Nodes$] compute $\boldsymbol{w}^{n+\frac12}_{i,j}$ on 
the horizontal lines $P_{j}^{H}$ ($i\in I_{j}^{H},\ j=1,\ldots,M^{V}$),
\begin{equation*}
\boldsymbol{w}^{n+\frac12}_{i,j}=2\boldsymbol{u}^{n+\frac12}_{i,j}
-\boldsymbol{w}^{n}_{i,j}-\boldsymbol{f}^{n+\frac14}_{i,j},
\end{equation*}
\item[(D4)$_\Nodes$] compute $\boldsymbol{u}^{n+1}_{i,j}$ on the 
horizontal lines $P_{j}^{H}$ ($i\in I_{j}^{H},\ j=1,\ldots,M^{V}$),
\begin{equation*}
\boldsymbol{u}^{n+1}_{\cdot,j}=
A_{\Nodes}(\boldsymbol{p}^{H}_{\cdot,j},\boldsymbol{q}_{\cdot,j})
\begin{pmatrix}
\boldsymbol{w}_{\cdot,j}^{n+\frac12}+\boldsymbol{f}_{\cdot,j}^{n+\frac34}\\[0.2cm]
g^{n+1}(a_{j}^{H})\\[0.2cm]
g^{n+1}(b_{j}^{H})
\end{pmatrix}.
\end{equation*}
\end{itemize}

Note that the initialization step (D1)$_\Nodes$ only requires no more 
than $M^{H}$ evaluations of the mapping $B_{\Nodes}$, and every time 
step (D2$_{\Nodes}$)-(D4$_{\Nodes}$) involves no more than $M^{H}+M^{V}$
evaluations of the mapping $A_{\Nodes}$---in other words, each
time step requires no more than $M^{H}+M^{V}$ solutions of
one-dimensional two-point BVP with Dirichlet boundary conditions. The
computational cost of these ODE solvers is analysed in
Section~\ref{ODE_cost}. The cost of the overall diffusion PDE solver
(which, in brief, runs at FFT speeds), is illustrated in
Tables~\ref{tab:cputime} and~\ref{tab:heating-circuit}.

\subsection{FC-AD wave propagation solver\label{wave_prop}}
Taking into account
Sections~\ref{wave:time-discrete},~\ref{bound_proj}
and~\ref{over_spat_disc} as well as Remark~\ref{pt_values}, our full
FC-based ADI discrete procedure for the wave propagation problem 
consists of the following three steps:
\begin{itemize}
\item[(W1)$_\Nodes$] Initialize $\boldsymbol{u}_{i,j}^{0}$ and 
$\boldsymbol{u}_{i,j}^{1}$ on the vertical lines $P_{i}^{V}$ 
($j\in I_{i}^{V},\ i=1,\ldots,M^{H}$),
\begin{equation*}
\left\{
\begin{array}{ll}
\boldsymbol{u}_{i,j}^{0}=u_{0}(x_{i},y_{j}),\\[0.2cm]
\boldsymbol{u}_{i,j}^{1}=u_{0}(x_{i},y_{j})+\Delta t u_{1}(x_{i},y_{j}),
\end{array}\right.
\end{equation*}
\end{itemize}
and, for $n=0,1,\ldots,n_{\mathrm{max}}$,
\begin{itemize} 
\item[(W2)$_\Nodes$] compute $\boldsymbol{w}^{n+\frac12}_{i,j}$ on the 
vertical lines $P_{i}^{V}$ ($j\in I_{i}^{V},\ i=1,\ldots,M^{H}$),
\begin{equation*}
\boldsymbol{w}^{n+\frac12}_{i,\cdot}=A_{\Nodes}
(\boldsymbol{p}_{i,\cdot}^{V},\boldsymbol{q}_{i,\cdot})
\begin{pmatrix}
2\boldsymbol{u}_{i,\cdot}^{n}-\boldsymbol{u}_{i,\cdot}^{n-1}
+\boldsymbol{f}_{i,\cdot}^{n+\frac12}\\[0.2cm]
g^{n+1}(a_{i}^{V})\\[0.2cm]
g^{n+1}(b_{i}^{V})
\end{pmatrix},
\end{equation*}
\item[(W3)$_\Nodes$] compute $\boldsymbol{u}^{n+1}_{i,j}$ on the 
horizontal lines $P_{j}^{H}$ ($i\in I_{j}^{H},\ j=1,\ldots,M^{V}$),
\begin{equation*}
\boldsymbol{u}^{n+1}_{\cdot,j}=A_{\Nodes}
(\boldsymbol{p}_{\cdot,j}^{H},\boldsymbol{q}_{\cdot,j})
\begin{pmatrix}
\boldsymbol{w}^{n+\frac12}_{\cdot,j}\\[0.2cm]
g^{n+1}(a_{j}^{H})\\[0.2cm]
g^{n+1}(b_{j}^{H})
\end{pmatrix}.
\end{equation*}
\end{itemize}

Every time step (W2$_{\Nodes}$)-(W3$_{\Nodes}$) in the present wave
equation algorithm requires at most $M^{H}+M^{V}$ evaluations of the
mapping $A_{\Nodes}$ and thus, in view of Section~\ref{ODE_cost}, the
overall FC-based wave propagation solver runs at FFT speeds.

\section{Numerical results}
\label{sec:numerical}
This section presents a variety of numerical results demonstrating the
accuracy, unconditional stability, reduced computational cost and
spatial dispersionlesness of the variable-coefficient FC-AD algorithms
introduced in this paper. Implementation details and hardware setup
used include the following:
\begin{enumerate}
\item All numerical simulations presented in this section have
  resulted from Fortran implementations of the FC-AD algorithms
  introduced in previous sections, running on a single processor
  Intel{\textregistered} Xeon{\textregistered} (model X5570) at
  2.93~GHz with 8MB cache size.
\item \label{pt_two}In accordance with Sections~\ref{gmres}
  and~\ref{param_selc}, the numerical simulations presented in this
  section use the FC(Gram) parameters $\Ndelta=10$ and $\Nd=\lfloor
  26(1+(\Nodes-21)/100)\rfloor$, and, following~\cite{bruno10,lyon10},
  the values $\Ndeg=5$ and $4$ for PDE solver of the diffusion and the
  wave model, respectively. (In agreement with those references we
  have found that these values of $\Ndeg$ ensure unconditional
  stability of the FC-AD algorithm.) In all cases the oversampling
  parameter used for the GMRES preconditioner (Section~\ref{gmres}) is
  set to $\Nover = 4$, and boundary conditions are enforced by means
  of the hybrid exterior-source/asymptotic-matching algorithm (see
  Remark~\ref{rem:detect_boundary_layer}).
\item \label{pt_three} At each time step $t_n$ the solution is stored
  as a matrix $\left(\boldsymbol{u}^{n+1}_{i,j}\right)$ of size 
  $M^{H}\times M^{V}$ containing the approximate solution values at 
  $(x_i,y_j )\in\Omega_h$ and zeroes for $(x_i,y_j )\not\in\Omega_h$).
\item \label{pt_five} Since the forward operator $B_{\Nodes}$ and the
  BVP solver $A_{\Nodes}$ use periodic extensions of the original
  problem, the Fortran implementation of both procedures for each
  horizontal and vertical grid lines use work vectors with more
  entries than the number of points supported on lines $P^{H}_{i}$ and
  $P^{V}_{j}$ of $\Omega_{h}$. For each half-time step ((D2)$_N$,
  (W2)$_N$, (D4)$_N$ and (W3)$_N$) and each associated horizontal and
  vertical line, only the output quantities for indices in the
  corresponding sets $I^{H}_{i}$ and $I^{V}_{j}$ (that is, for the
  corresponding points in the computational domain $\Omega_h$) are
  stored in the matrices $\left(\boldsymbol{u}^k_{i,j}\right)$.
\item All needed FFTs are performed using the FFTW
  library~\cite{frigo05}. In addition, FFTW have been used to transpose
  {\it in-place} the matrices $\left(\boldsymbol{u}^k_{i,j}\right)$ to
  preserve the contiguous memory access (column-wise in the Fortran
  implementation) prior to the needed transfers of horizontal or
  vertical lines to the BVP solver work vector mentioned in
  point~\ref{pt_five}.
\end{enumerate}

\subsection{Boundary conditions: exterior-source and 
asymptotic-matching procedures\label{sec:bnd_numer}}
As indicated in Remark~\ref{rem:detect_boundary_layer}, our
variable-coefficient FC-AD algorithm automatically selects the
mechanism---either exterior-sources or asymptotic-matching---for
enforcement of boundary conditions. In view of the discussion in
Section~\ref{stiff}, it is clear that the hybrid
exterior-source/asymptotic-matching procedure can be used to produce
accurate solutions for arbitrary time-steps in computing times per
time-step that remain bounded as $\Delta t\to 0$. The present section,
in turn, presents numerical results that demonstrate quantitatively
the impact of the hybrid approach, in terms of computing time and
accuracy, on the solution of full PDE problems.

Table~\ref{tab:cputime}, which presents computing times required by
the FC-AD diffusion solver for $1000\times 1000$ two-dimensional grid
and for two different values of the time step $\Delta t$ (in the
particular case of the first diffusion problem considered in
Section~\ref{diff_numer}, see also Figure~\ref{fig:heat-time}),
provides some insight into the computing costs required by the hybrid
approach in each of the two possible $(h,\Delta t)$ regimes. The
``Setup'' columns in this and subsequent tables display the overall
time used in precomputations---including each one of the following
operations {\em for each horizontal and vertical line $P^{H}_{j}$ and
  $P^{V}_{i}$}: 1)~Precomputation of FFTW plans~\cite{frigo05} in
real-valued arithmetic; 2)~Evaluation of the scaled FC(Gram)
continuations for the variable coefficients and the right-hand side;
3)~LU factorization of the preconditioning finite-difference matrix
(see Section~\ref{gmres}); 4)~Evaluation of the auxiliary solutions
for treatment of boundary conditions (see Section~\ref{sec:equiv-for} 
and~\ref{asympt_exp}); and 5)~Initialization of the initial
time step. As can be gleaned from Table~\ref{tab:cputime}, for the
diffusion problem under consideration, the FC-AD asymptotic-matching
procedure leads to somewhat smaller overall setup times but comparable
times per time-step as the corresponding exterior-source method;
similar remarks apply to our FC-AD implementation of the
variable-coefficient wave propagation problem. Thus, the
asymptotic-matching method resolves the boundary layers that arise for
small values of $\Delta t$ (which cannot be accurately discretized by
the exterior-source method, unless unduly fine grids are used) at a
cost comparable to that which would be required by the exterior-source
method in absence of boundary layers. The exterior-source method, in
turn, can adequately treat cases in which no significant boundary
layers exist---for which the asymptotic-matching method would be
inaccurate; cf. Figure~\ref{fig:ODE-boundary-layer}
and~\ref{fig:k100_f1}.

\begin{table}[htb!]
\begin{center}
\begin{tabular}{llccccccccc}
\hline
Boundary condition & & 
\multicolumn{2}{c}{$tol_{\mathrm{GMRES}}=10^{-15}$} && 

\multicolumn{2}{c}{$tol_{\mathrm{GMRES}}=10^{-10}$} && 
\multicolumn{2}{c}{$tol_{\mathrm{GMRES}}=10^{-6}$}\\
\cline{3-4} \cline{6-7} \cline{9-10}
enforcement & $\Nover$ & Setup & Time step 
&& Setup & Time step 
&& Setup & Time step\\
\hline
$h=10^{-3}$, 
& 1 & 13.493  & 6.169 && 5.617 & 2.265 && 4.264 & 1.387 \\
$\Delta t=5\times 10^{-3}$,
& 2 & ~8.119 & 3.576 && 5.693 & 2.129 && 4.943 & 1.480 \\
(exterior-sources
& 4 & ~8.628 & 3.788 && 6.343 & 2.332 && 5.408 & 1.816  \\
regime)
& 8 & 11.700 & 5.189 && 8.816 & 3.273 && 7.532 & 2.548 \\
\hline 
$h=10^{-3}$, 
& 1 & ~3.039 & 6.301 && 3.005 & 3.339 && 3.041 & 1.478 \\
 $\Delta t=10^{-6}$,
& 2 & ~3.182 & 5.976 && 3.142 & 2.474 && 3.180 & 1.626 \\
(asymptotic-matching
& 4 & ~4.175 & 6.326 && 4.147 & 2.785 && 4.183 & 1.848 \\
regime)
& 8 & ~5.499 & 8.275 && 5.456 & 3.630 && 5.496 & 2.585 \\
\hline
\end{tabular}
\caption{CPU time (in seconds) required by our implementation of the
  FC-AD method, for the variable-coefficient diffusion problem
  mentioned in Section~\ref{sec:bnd_numer}, in two different regimes
  of the boundary condition enforcement algorithm.}
\label{tab:cputime}
\end{center}
\end{table}

To demonstrate, in a simple context, the hybrid
exterior-source/asymptotic-matching procedure for enforcement of
boundary conditions
(Sections~\ref{wave:time-discrete},~\ref{sec:equiv-for}
and~\ref{asympt_exp}), here we apply the hybrid algorithm in
conjunction with a one-dimensional FC-based time marching scheme to a
one-dimensional wave-propagation problem: the one-dimensional wave
equation~\eqref{eq:wave} with variable coefficients $\alpha(x)=1+4x^2$
and $\beta(x)=2-x+8x^2$ in the unit interval, and with Dirichlet
boundary conditions such that the exact solution is given by the
oscillatory function $u(x,t)=\sin(100x-2\pi t)$.
Figure~\ref{fig:k100_f1} displays the resulting maximum relative
solution error as a function of the time step $\Delta t$ and the grid
size $h$.  The left error map in this figure presents the relative
errors that result as the exterior-source procedure is used for all
values of $(h,\Delta t)$. As expected, in presence of boundary layers
(that arise for small values of $\Delta t$), the overall numerical
approximation provided by the exterior-source algorithm is completely
inaccurate. The right-hand error map in Figure~\ref{fig:k100_f1}, on
the other hand, displays the errors that result from the hybrid
boundary-conditions algorithm. The dashed line in the right-hand map
separates regimes in the hybrid approach: above this line the
exterior-source algorithm was utilized, below this line the
asymptotic-matching method was used.  Close consideration of the
right-hand error map shows that the hybrid boundary-conditions
algorithm leads to an overall convergent PDE solver. In the right-most
region of the right-hand error map (larger values of $h$), the error,
which results mostly from the coarseness of the spatial FC
discretization, is essentially independent of $\Delta t$. In the
left-most portion of the map, where finer spatial grids are used, the
spatial discrete errors are negligible in comparison with the time
discretization errors, and hence the resulting errors depends only on
the time-step value $\Delta t$.

\begin{figure}[!ht]
\begin{center}
$\begin{array}{cc}
\includegraphics[width=0.475\textwidth,angle=0]
{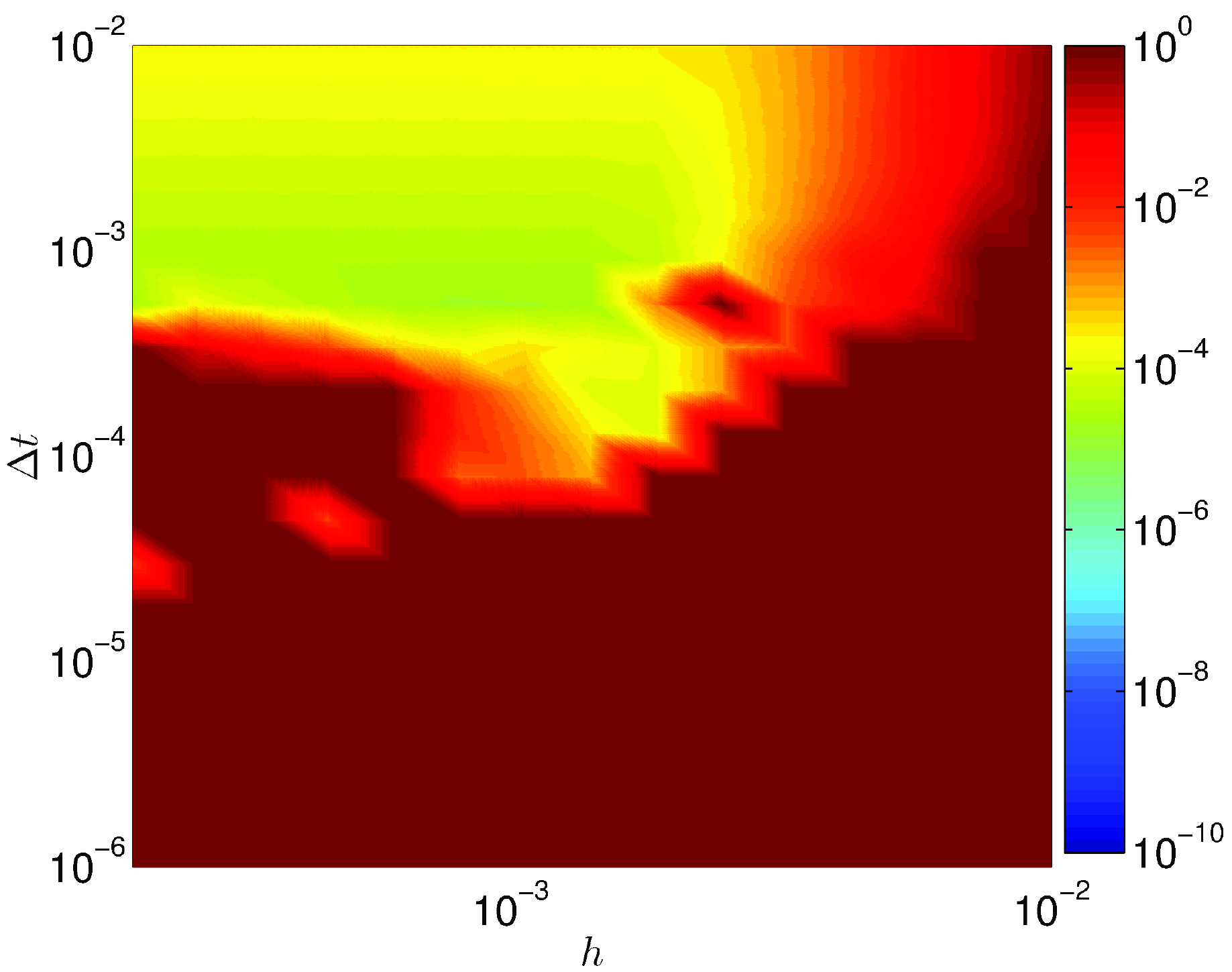} &
\includegraphics[width=0.475\textwidth,angle=0]
{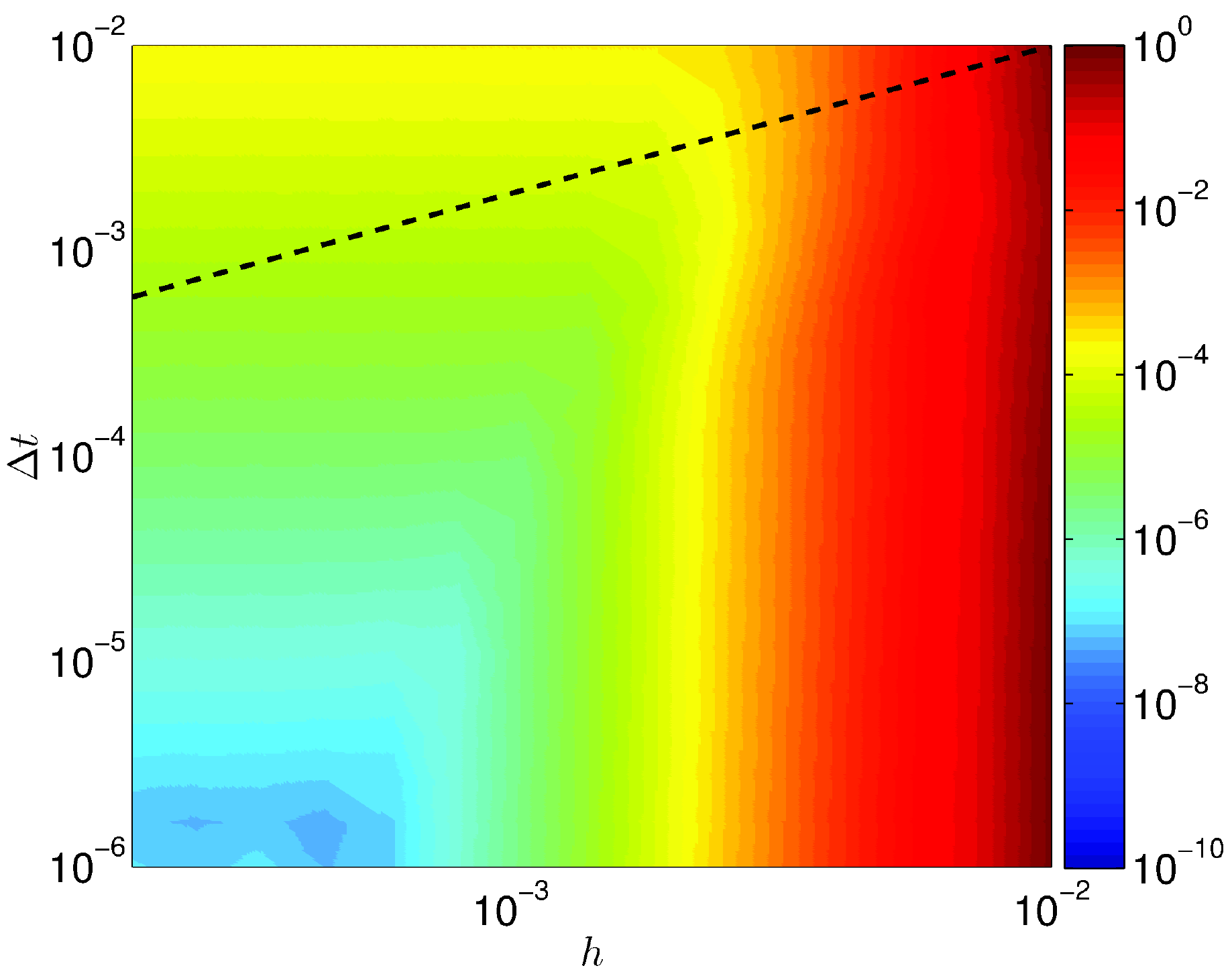}
\end{array}$
\caption{Maximum relative error in the FC-AD approximate solution of a
  one dimensional wave equation of the form~\eqref{eq:wave} using
  solely the exterior-source procedure (left) and using the hybrid
  boundary-conditions algorithm (right).
\label{fig:k100_f1}}
\end{center}
\end{figure}

\subsection{Diffusion problem: 
performance, convergence and stability\label{diff_numer}}
A variety of numerical examples presented in this section demonstrate
the character of the FC-AD scheme for diffusion problems with variable
coefficients.  For the first set of tests of this section we consider
a problem of the form~\eqref{eq:heat} with variable coefficients given
by $\alpha(x,y)=x+y+1$ and $\beta(x,y)=2x+0.5y+1$ in the domain
bounded by the curve $(x/9)^6+(y/5)^6=(1/20)^6$. The right-hand side and the boundary
conditions have been selected in such a way that the function
$u(x,y,t)=\sin(\pi(3x^2+2y^2+2t))$ is the exact solution of the
problem. For our $\Delta t$ convergence studies the FC-AD numerical
solution is produced up to the final time $T=0.1$.
\begin{figure}[!ht] 
\begin{center} 
$\begin{array}{cc}
\includegraphics[height=0.28\textheight,angle=0]
{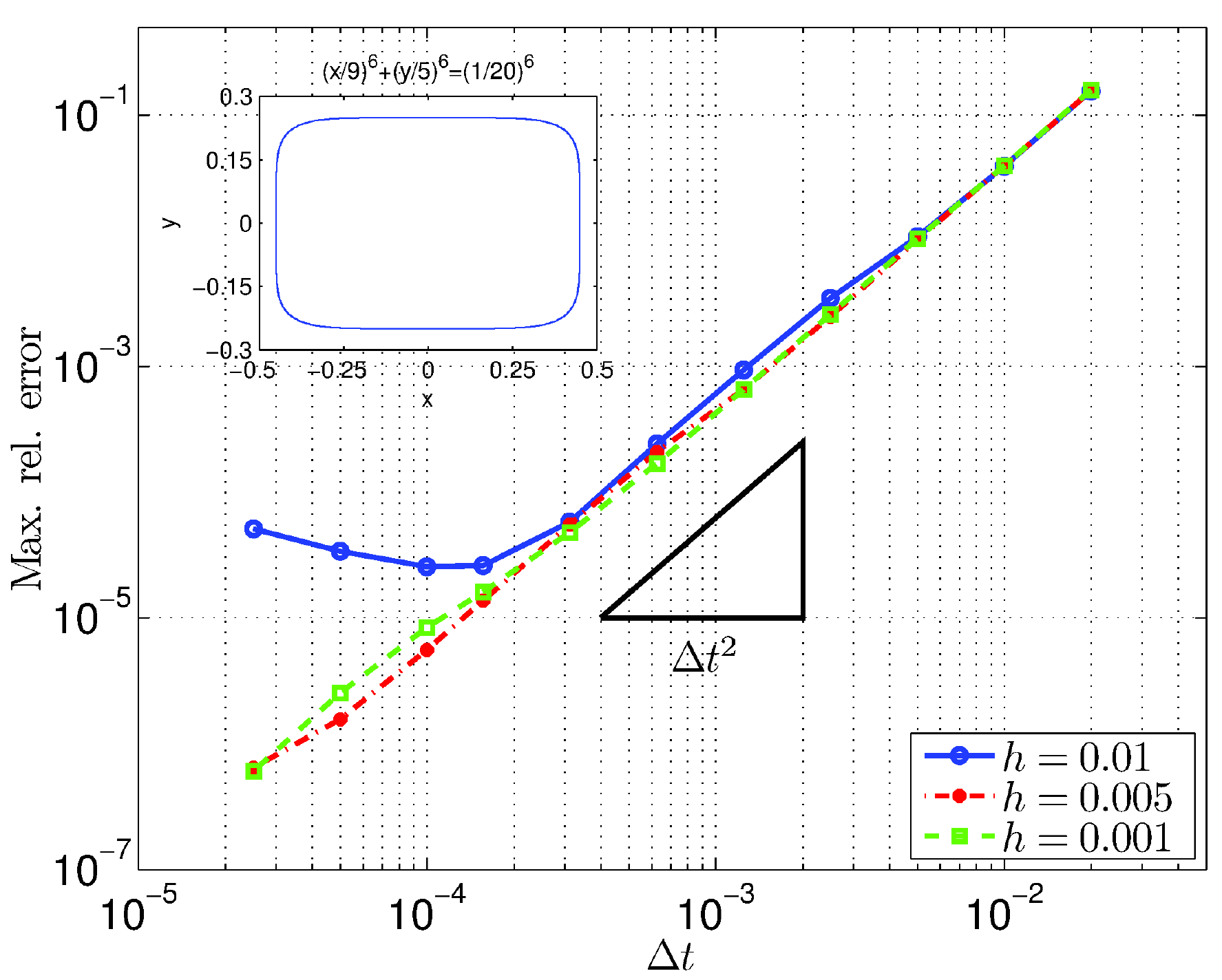} & 
\includegraphics[height=0.28\textheight,angle=0]
{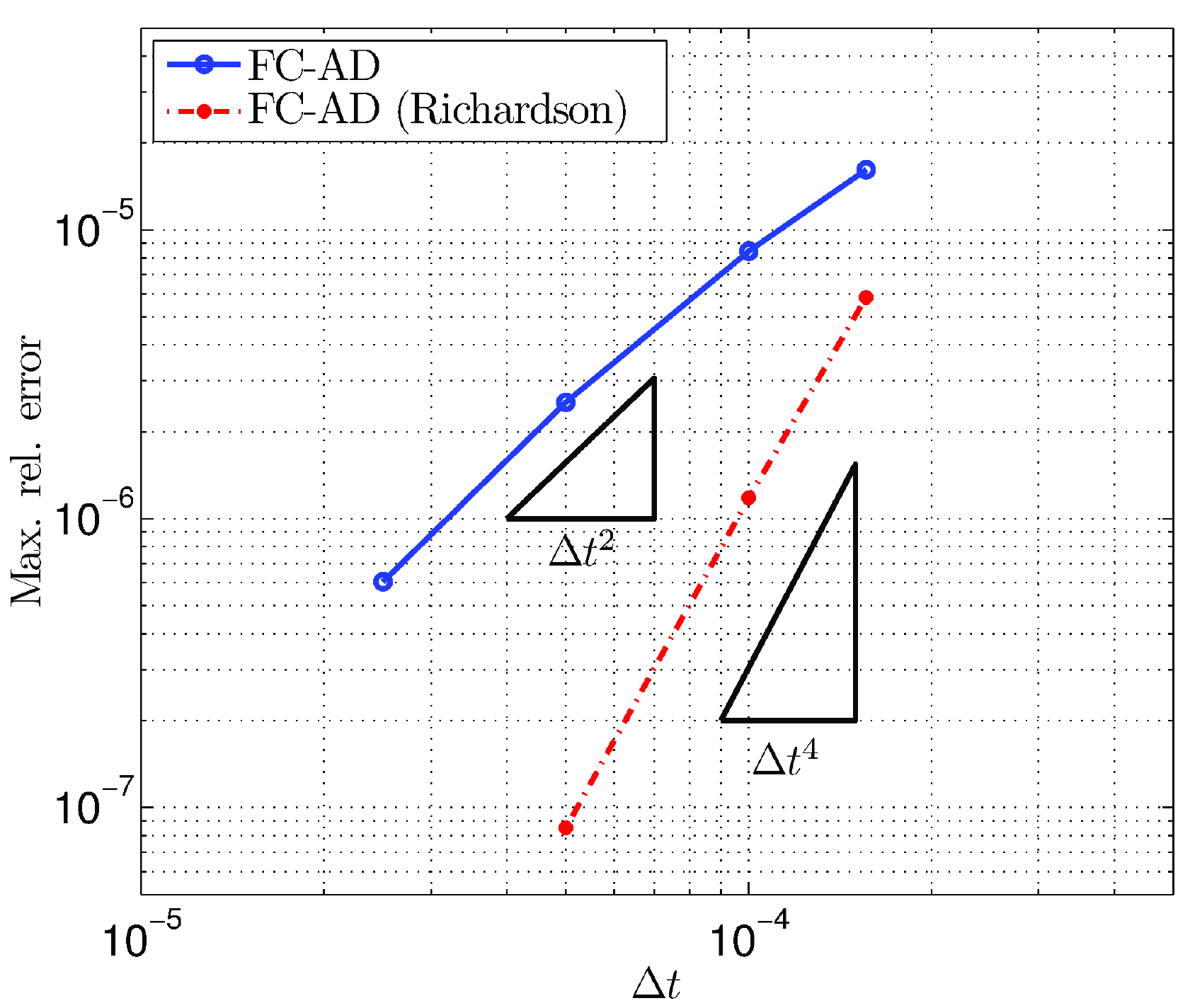} 
\end{array}$ 
\caption{Left: Maximum relative errors in the FC-AD approximate
  solution of the first diffusion problem mentioned in
  Section~\ref{diff_numer}, as a function of the time step $\Delta t$,
  for three different spatial grids. Right: Comparison of these
  solution errors with those obtained, for the same problem, via an
  application of the Richardson extrapolation method leading to two
  orders of improvement in the temporal convergence rate.}
\label{fig:heat-time}
\end{center} 
\end{figure}
The spatial discretizations are given by uniform grids of size $h$
within the square $[-0.5,0.5]\times[-0.3,0.3]$, whereas the 
preconditioner uses  the oversampling ratio $\Nover=4$ and the GMRES 
tolerance $tol_{\mathrm{GMRES}}=10^{-10}$.

Figure~\ref{fig:heat-time} displays maximum relative errors throughout
$\Omega_{h}$ (evaluated through comparison with the exact solution)
produced by the second-order in time FC-AD algorithm described in
Section~\ref{diffusion} for $h=0.01$, $0.005$, and $0.001$ (left plot
in Figure~\ref{fig:heat-time}), as well as corresponding results
obtained by means an additional application of the Richardson
extrapolation procedure (right plot in Figure~\ref{fig:heat-time},
see~\cite{bruno10} and references therein for details on the the
application of the Richardson extrapolation method in the time
domain). Since the spatial discretization errors for different grids
are negligible with respect to the time discretization (at least for
the the coarser time steps), the relative error exhibits the expected
second-order convergence that results from the ADI time-marching
scheme described in Section~\ref{sec:time-disc}. Additionally, the
convergence displayed in this figure demonstrates that the underlying
solver is not subject to the ordinary CFL constraint $\Delta t\sim
h^2$ required by explicit solvers: for the case $h=0.001$ the largest
$\Delta t$ values used here are in fact four orders of magnitude
larger than would be allowed by the quadratic CFL constraint. (In
fact, our experiments suggest that $\Delta t $ can be increased
arbitrarily without leading to instability: values as large as $\Delta
t = 100$ and $\Delta t = 1000$, etc, lead to stable, albeit inaccurate
solutions.)

To demonstrate the performance of the FC-AD method for general
geometries we consider the curved, non-convex {\it heating circuit}
structure depicted in Figure~\ref{fig:heating-circuit}.  The variable
heat-equation coefficients used correspond to (variable) thermal
constants of silicon; the geometry, in turn, represents a
$80\times 120$ rectangular plate containing a
curved heating circuit.
\begin{figure}[!ht] 
\begin{center} 
$\begin{array}{c@{\hspace{0.16\textwidth}}c}
\includegraphics[width=0.35\textwidth,angle=0]
{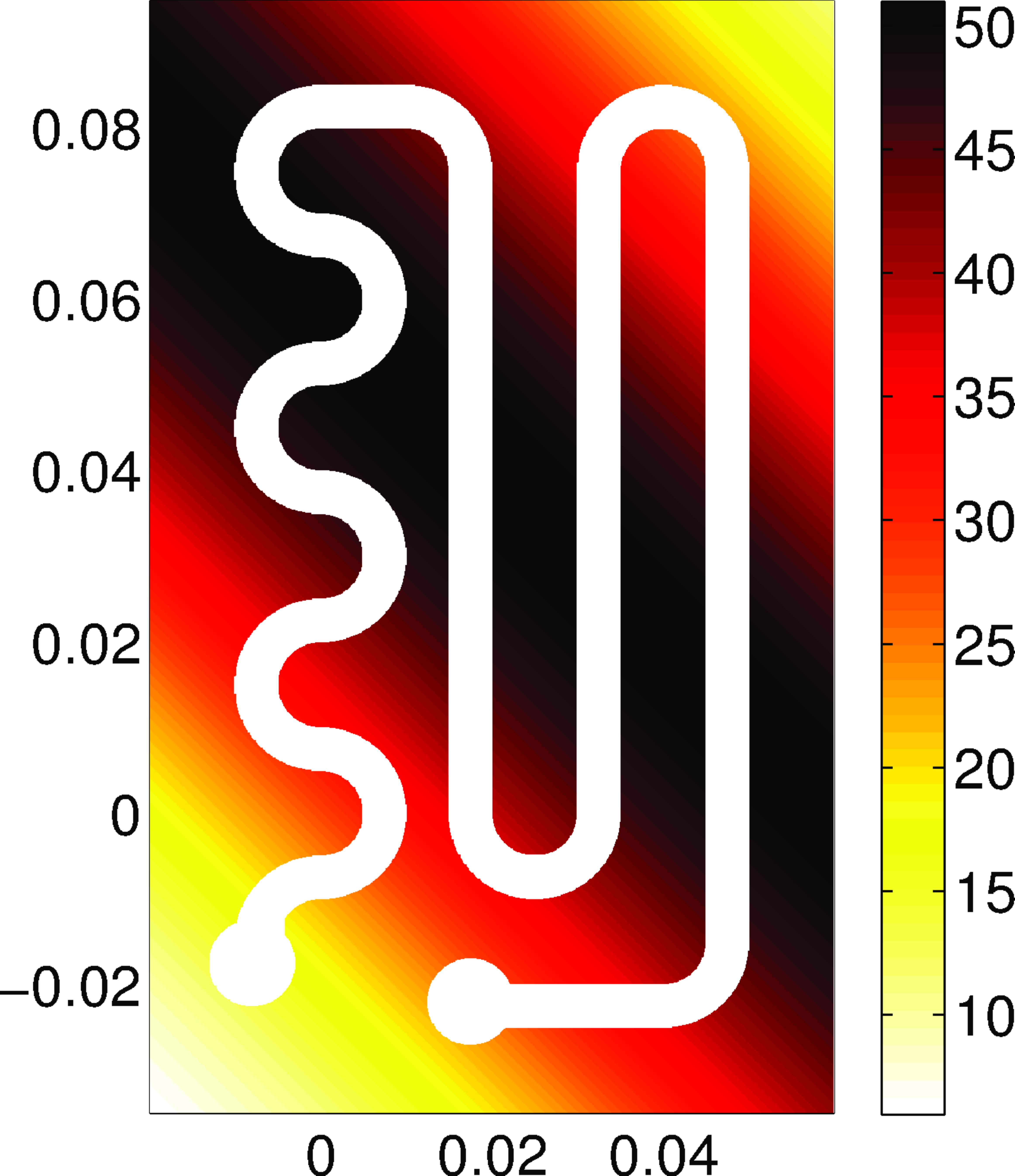} &
\includegraphics[width=0.35\textwidth,angle=0]
{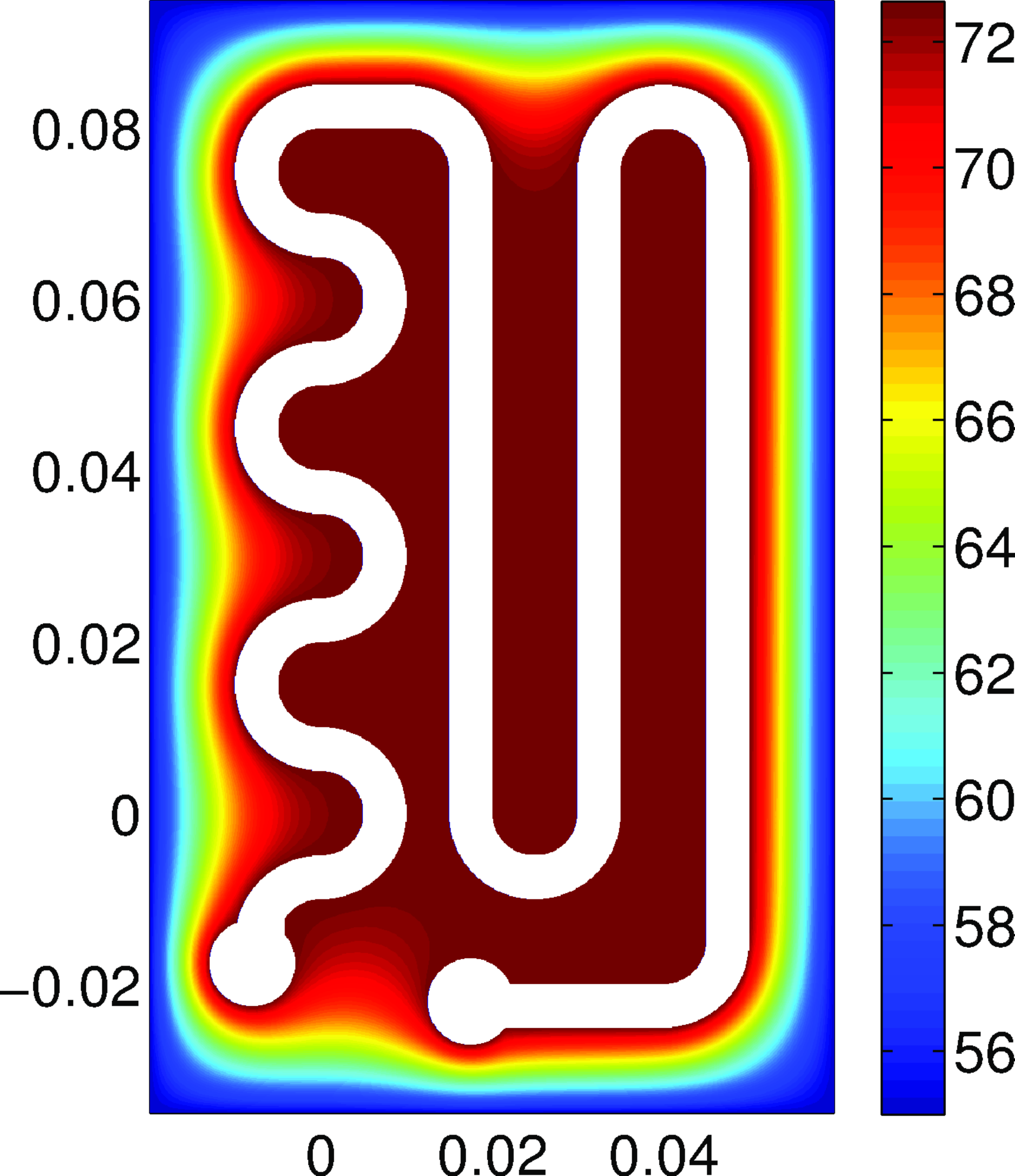} 
\end{array}$ 
\caption{Variable thermal conductivity (left) and temperature field
  produced by the FC-AD method at time $T=0.1$ (right) for the second
  diffusion problem mentioned in Section~\ref{diff_numer}.}
\label{fig:heating-circuit}
\end{center} 
\end{figure}
The thermal conductivity varies from $5$ to $50$ (the typical range for 
silicon) as shown in the left plot of
Figure~\ref{fig:heating-circuit}. The initial temperature has been
fixed to $u_{0}=55$. The temperature profile on the exterior
boundaries of the plate is fixed also to $55$, whereas that the
inner boundaries of the heating circuit are driven by the function
$g(x,y,t)=55+30\sin(20\pi t)$.

The right plot in Figure~\ref{fig:heating-circuit} presents the
temperature field at time $T=0.1$---that is, one period of the
boundary data function $g$---evaluated by means of the FC-AD algorithm
described in Section~\ref{diffusion} (without Richardson
extrapolation, for simplicity) on a spatial grid containing $800\times
1300$ discretization points.  The corresponding maximum errors were
evaluated through comparison with a reference numerical solution
produced using $\Delta t=2.5\times 10^{-5}$ on a $1600\times 2600$
spatial grid: it was found that the $800\times 1300$ solutions with
$\Delta t=10^{-4}$ and $5\times 10^{-5}$ and
$tol_{\mathrm{GMRES}}=10^{-10}$ contain maximum errors of $0.054$\%
and $0.013$\%, respectively---demonstrating, in particular, second
order convergence in time. Analogous relative errors at somewhat
faster computing times result from use of the tolerance value
$tol_{\mathrm{GMRES}}=10^{-6}$.  
\begin{table}[htb!]
\begin{center}
\begin{tabular}{ccccccccc}
\hline
 & & 
\multicolumn{3}{c}{$tol_{\mathrm{GMRES}}=10^{-10}$} && 
\multicolumn{3}{c}{$tol_{\mathrm{GMRES}}=10^{-6}$} \\
\cline{3-5} \cline{7-9} 
$\Delta t$  & $\Nover$ & Setup & Time step & Total
                      && Setup & Time step & Total\\
\hline
10$^{-4}$ & 4 & 33.989 & 2.343 & 268.289 && 
                25.798 & 2.107 & 236.498 \\
5$\times$10$^{-5}$ & 4 & 33.696 & 2.406 & 274.296 && 
                         23.944 & 1.883 & 212.244\\
2.5$\times$10$^{-5}$ & 4 & 34.095 & 2.371 & 271.195 && 
                           23.271 & 1.693 & 192.571\\
\hline
\end{tabular}
\caption{CPU times (in seconds) required to evolve the heating-circuit
  FC-AD solver for a total 100 time steps on a $800\times 1300$ grid
  for various values of the time step $\Delta t$.}
\label{tab:heating-circuit}
\end{center}
\end{table}

The CPU times required to obtain the various heating-circuit solutions
mentioned above are presented in Table~\ref{tab:heating-circuit} for
two different values of the GMRES residual tolerance
$tol_{\mathrm{GMRES}}$ (both of which yield similar errors, as
indicated previously in this section, for the values the parameters
under consideration presently).

\subsection{Wave propagation problem: performance, convergence and
stability\label{wave_numer_perf}}
The following two subsections demonstrate the properties of the FC-AD
solver for the wave equation with variable
coefficients. Section~\ref{sec:dispersionless} highlights the highly
significant performance gains that arise from the low spatial dispersion
inherent in the FC methodology, including comparisons with
finite-difference methods. Section~\ref{wave_numer}, in turn, presents
an application of the FC solver to a non-trivial geometry, and it
demonstrates the convergence and stability of the approach.
\subsubsection{Spatial dispersionlessness}
\label{sec:dispersionless}
We demonstrate the spatial dispersionlessness of our FC wave-equation
solvers by means of the one-dimensional problem defined by the
equation
$\alpha(x)\partial_{tt}u-\partial_{x}(\beta(x)\partial_{x}u)=0$ with
variable coefficients $\alpha(x)=1+x/2$ and $\beta(x)=2/(2+x)$ in the
spatial interval $(a,b)=(0,1)$, with final time $T=1$, and with
Dirichlet boundary conditions such that the exact solution is given by
$u(x)=\sin(2\pi f(x^2/4+x+t))$.  The FC-based time marching scheme
for the present one-dimensional case is analogous to that presented in
Section~\ref{wave:time-discrete} for two dimensions but, of course, in
here use of alternating directions is neither possible nor necessary.
In order to focus attention on the spatial dispersion properties of
the algorithm, our first example in this section uses a fixed time
discretization which gives rise to errors smaller than all the
corresponding spatial errors: $\Delta t=8\times10^{-7}$ used in conjunction
with second order Richardson extrapolation.  The spatial grids used,
in turn, are frequency dependent: they are taken to hold a prescribed
number of points per wavelength (PPW).  Since the coefficients and
thus the wavelength vary within the physical domain, the PPW quantity
is defined as the number of discretization points used within the {\em
  shortest} wavelength---which, in our case, can be defined as the
minimum value of the ``wavelength function'' 
$\min\{1/(f(x/2+1)):x\in[0,1]\}=2/(3f)$.

\begin{figure}[!ht] 
\begin{center} 
$\begin{array}{cc}
\includegraphics[width=0.475\textwidth,angle=0]
{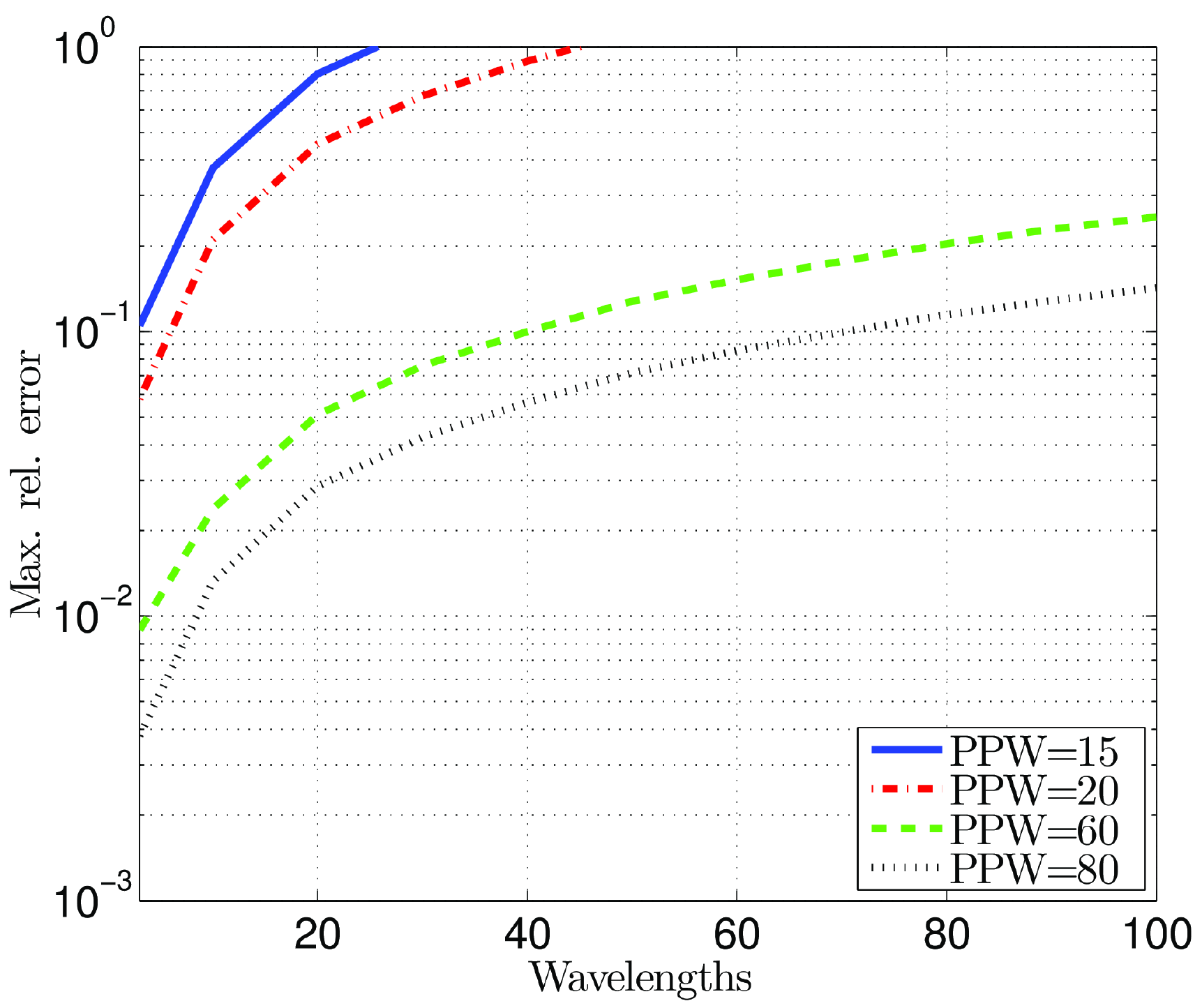} &
\includegraphics[width=0.475\textwidth,angle=0]
{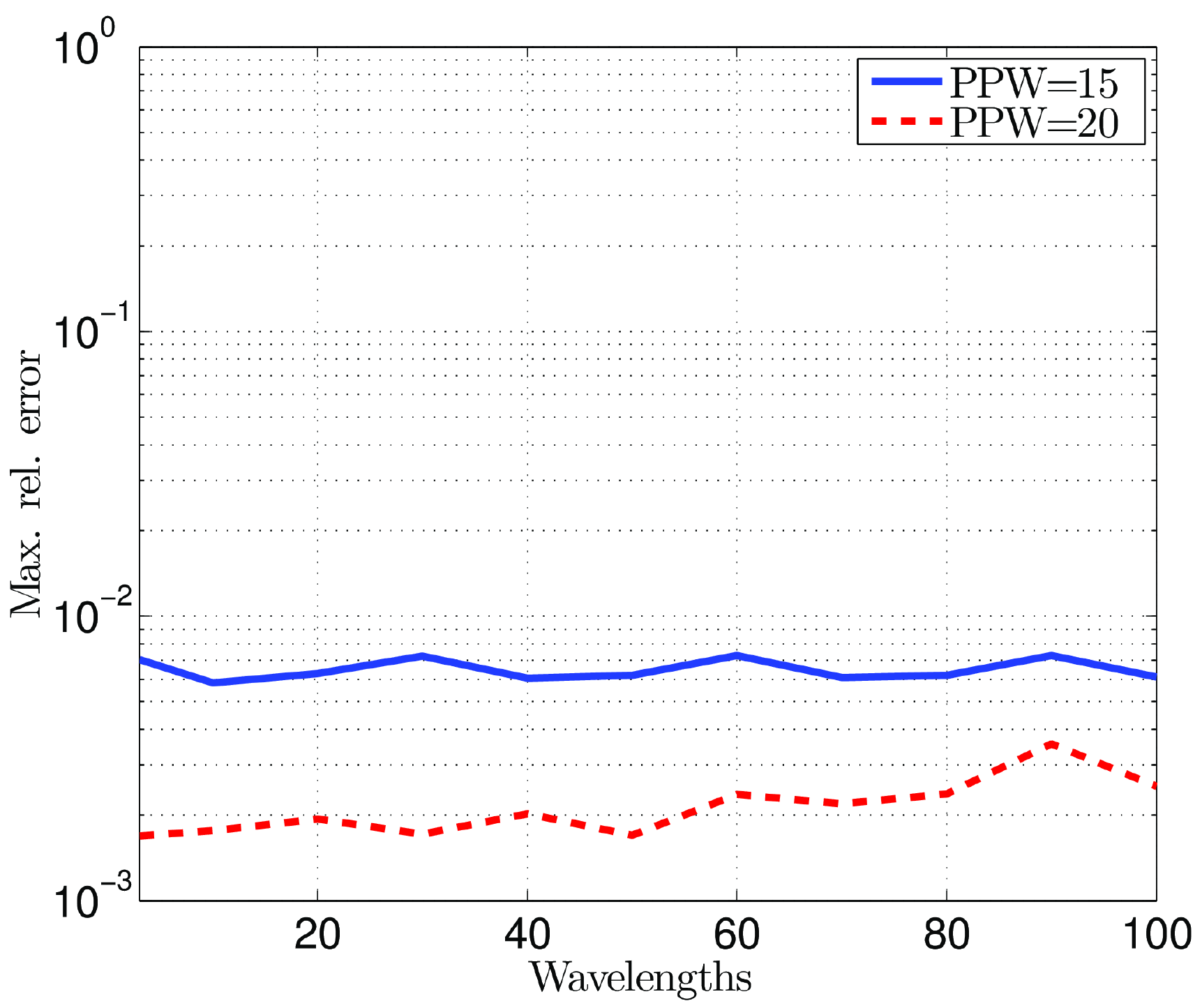}
\end{array}$ 
\caption{Maximum relative errors in numerical solutions of the
  one-dimensional wave equation in the time interval $[0,1]$ using a
  fixed number of Points-Per-Wavelength (PPW). Left: one-dimensional
  finite difference solver. Right: one-dimensional FC-AD solver.}
\label{fig:ODE-dispersion}
\end{center} 
\end{figure}

Figure~\ref{fig:ODE-dispersion} displays maximum errors relative to
the maximum solution values for numerical solutions produced by two
spatial differentiation methods, namely second-order finite
differences (left portion of Figure~\ref{fig:ODE-dispersion}) and
Fourier continuation (right portion of
Figure~\ref{fig:ODE-dispersion}). The time discretization in both
cases is the one described above: second-order Richardson
extrapolation with $\Delta t=8\times10^{-7}$. It can be seen from these
figures that, even at the lowest frequencies (lowest number of
wavelengths in the domain $(0,1)$), the FC solutions for 15 and 20 PPW
are significantly more accurate than the corresponding
finite-difference solutions, and, for such low frequencies, only the
finite-difference solutions using 60 and 80 PPW approximately match
the error in the 15 and 20 PPW FC solutions.  As the frequency
increases the numerical error in the FC solutions remains essentially
constant, that is, the FC solver is essentially spatially
dispersionless. The finite difference scheme, on the other hand, is
not: numerical errors increase significantly with frequency, and an 80
PPW mesh can only produce 1\% accurate solutions for problems
containing no more than five PPW.

\subsubsection{Wave propagation problem: complex structures \label{wave_numer}} 
This section demonstrates the properties of the overall FC-AD scheme
for problems of wave motion with variable coefficients. In our first
example we consider the wave propagation problem~\eqref{eq:wave} with
variable coefficients given by $\alpha(x,y)=1+x+y$,
$\beta(x,y)=2.0x+0.5y+1$ in the domain bounded by the curve
$(x/9)^6+(y/5)^6=(1/20)^6$. The PDE right-hand side and boundary 
conditions are selected in such a way that the function 
$u(x,y,t)=\sin(\pi(x+2y-t))$ is the exact solution of the problem.
\begin{figure}[!ht] 
\begin{center} 
$\begin{array}{cc}
\includegraphics[height=0.28\textheight,angle=0]
{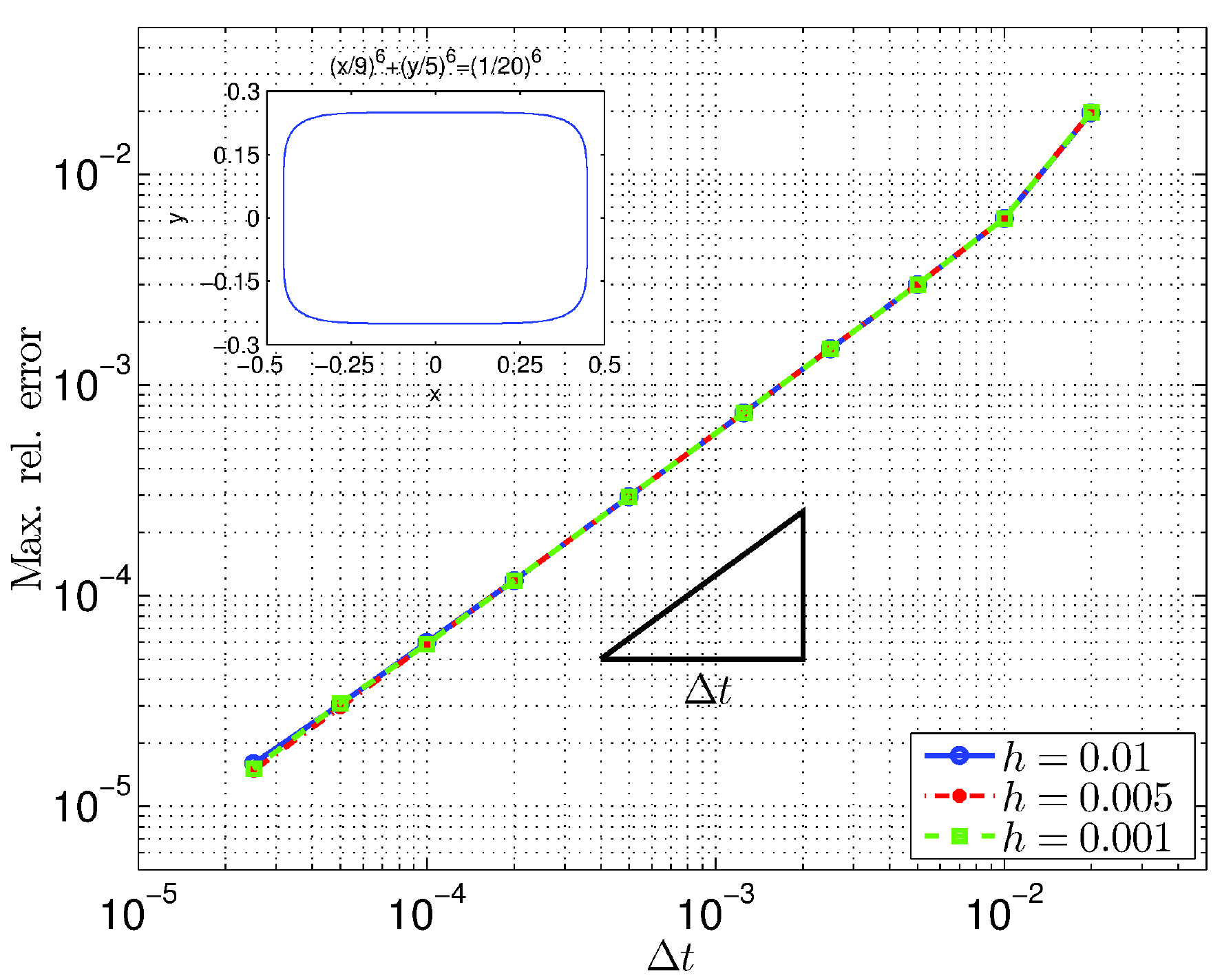} & 
\includegraphics[height=0.28\textheight,angle=0]
{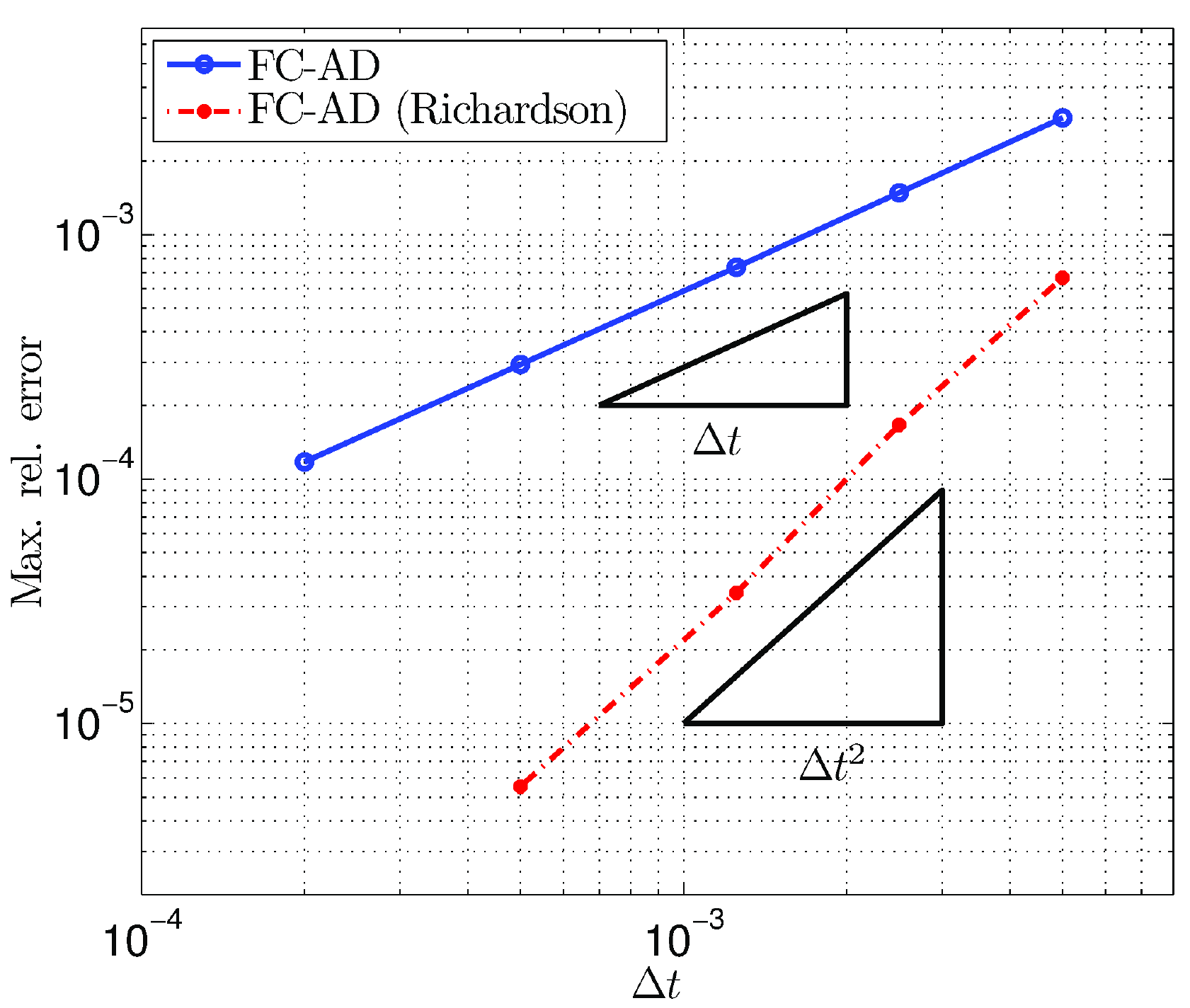} 
\end{array}$ 
\caption{Left: Maximum relative errors in the FC-AD approximate
  solution of the first wave propagation problem mentioned in
  Section~\ref{wave_numer}, as a function of the time step $\Delta t$,
  for three different spatial grids. Right: Comparison of these
  solution errors with those obtained, for the same problem, via an
  application of the Richardson extrapolation method leading to one
  order of improvement in the temporal convergence rate.}
\label{fig:wave-time}
\end{center} 
\end{figure}
Uniform meshes of grid-size $h$ are used to discretize the square
$[-0.5,0.5]\times[-0.3,0.3]$ which contains the PDE domain. For our $\Delta t$
convergence studies, errors in the FC-AD numerical solution are
evaluated at the final time $T=0.1$. The GMRES tolerance is set to
$tol_{\mathrm{GMRES}}=10^{-10}$, and the preconditioner uses the
oversampling ratio $\Nover=4$.

Figure~\ref{fig:wave-time} presents maximum values of the solution
error throughout $\Omega_{h}$ (evaluated through comparison with the
exact solution) relative to the maximum value of the solution, that
results from use of the FC-AD algorithm described in
Section~\ref{wave:time-discrete} for $h=0.01$, $0.005$, and $0.001$
(left plot), as well as corresponding results obtained by means an
additional application of the Richardson extrapolation procedure
(right plot).  Since the spatial discretization errors for different
grids are negligible with respect to the time discretization, the
relative error exhibits the expected first and second-order
convergence that results from the FC-AD time-marching scheme
(described in Section~\ref{sec:time-disc}) and the application of the
Richardson extrapolation procedure, respectively. Once again, the
unconditional stability of the FC-AD time-marching scheme allows us to
consider a wide range of time steps, without CFL-type
restrictions. Notice that, as a result of the hybrid
exterior-source/asymptotic-matching procedure for enforcement of
boundary conditions, small values of the time-step $\Delta t$ (and the
associated boundary layers) do not give rise to accuracy losses even
when the coarse mesh-size $h=0.01$ is used.

To illustrate the applicability of the FC-AD wave solver for PDEs with
general spatially variable coefficients, we consider the {\it
  waveguide} depicted in the left portion of
Figure~\ref{fig:wave-device}, which consists of a pair of curved
channels within a rectangular plate of dimensions $1\times 2$ with
chamfered corners.
\begin{figure}[!ht] 
\begin{center} 
$\begin{array}{c@{\hspace{1.5cm}}c}
\includegraphics[height=0.3\textheight,angle=0]
{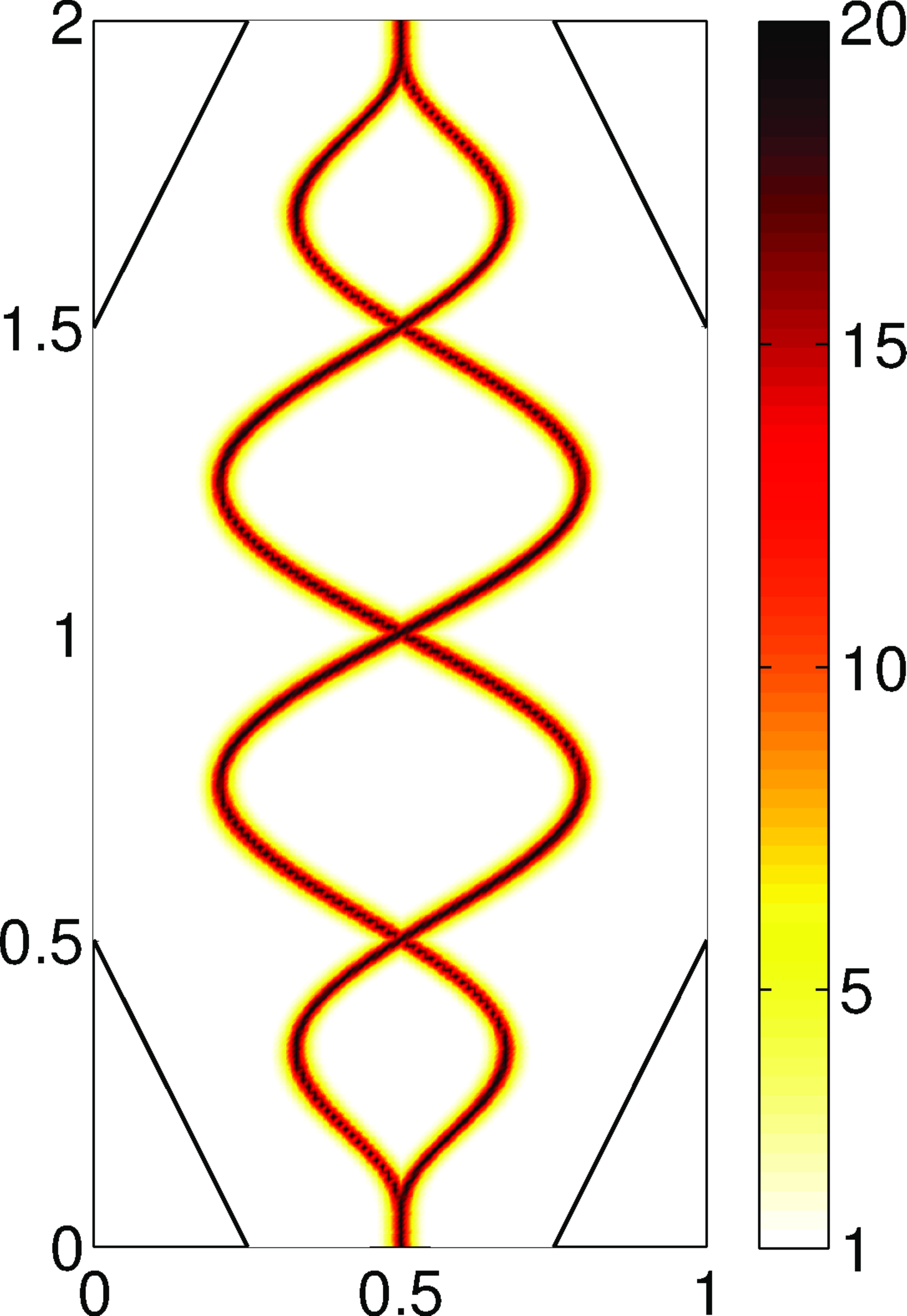} &
\includegraphics[height=0.3\textheight,angle=0]
{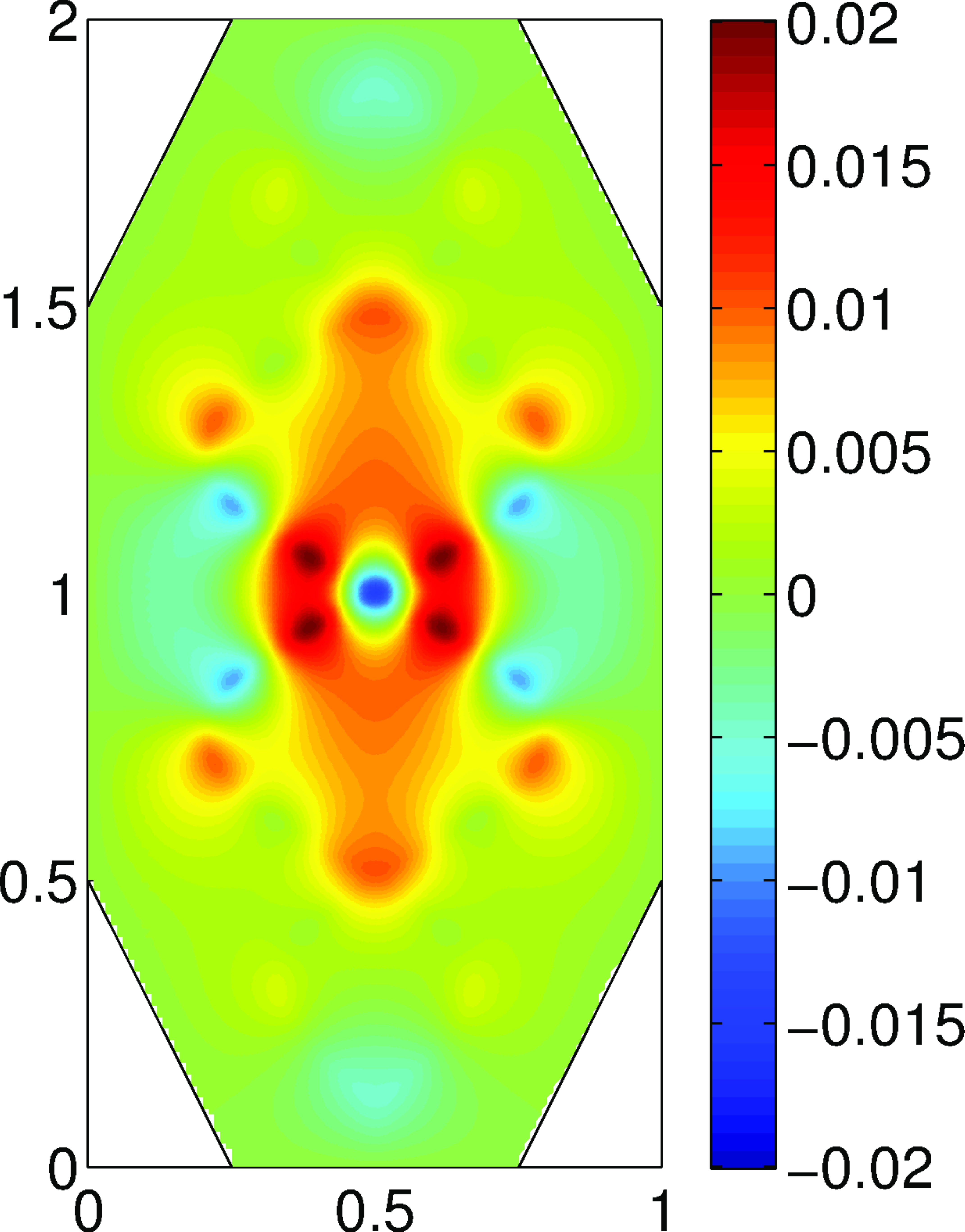} 
\end{array}$ 
\caption{Left: Values of the variable coefficient
  $\alpha=\alpha(x,y)$. Right: FC-AD approximate solution at time
  $T=4$.}
\label{fig:wave-device}
\end{center} 
\end{figure}
With reference to equation~\eqref{eq:wave}, the material is
characterized by variable coefficients $\alpha$ and $\beta$ where
$\alpha$ varies between $1$ and $20$ as depicted in the left portion
of Figure~\ref{fig:wave-device}, and where $\beta=1$.  The structure
is excited by a spatially Gaussian harmonic pulsed source of frequency
$f=3\pi$ supported in a disk of radius $0.05$ centered at point
$(0.5,1)$.  The solution at the final time $T=4$, which is displayed
on the right portion of Figure~\ref{fig:wave-device}, was obtained
using $\Delta t=10^{-2}$ on a $200\times 400$ spatial grid and
$tol_{\mathrm{GMRES}}=10^{-6}$ in a total computational time of
$194.147$ seconds (requiring $18.147$ for the setup and $0.440$ seconds 
per time-step). Using, for reference, a solution computed via a 
$400\times 800$ spatial grid and $\Delta t=2.5\times 10^{-3}$, it was 
found that the solution above contains an error of $0.002$\%. 

\section*{Conclusions}

We have introduced Fourier-based alternating direction time-marching
schemes for the numerical solution of linear PDEs {\em with variable
  coefficients}. Following~\cite{bruno10,lyon10}, our use of the
Fourier continuation method in combination with the ADI time-marching
scheme and the Fast Fourier Transform, gives rise to a highly
desirable combination of properties, namely, unconditional stability
and high-order accuracy and spatial dispersionlessness at FFT speeds in the
general context of non-periodic functions and for general domains.  A
variety of numerical results demonstrate the properties of the
resulting solvers for problems of diffusion as well as wave
propagation and scattering in media with spatially varying
characteristics.

\appendix
\section{An auxiliary lemma}
\label{Appen}
\setcounter{equation}{0}
\numberwithin{equation}{section}
\begin{lemma}\label{lem:ode}
  Let $\tilde{q}^{\ell}$, $g_{a}$ and $g_{b}$ be smooth functions
  defined in the interval $[b,c]$, and let $\tilde{q}^{\ell}$ be
  strictly positive in that interval.  If $g_{a}$ and $g_{b}$ satisfy
  the conditions~\eqref{eq:assump-lemma-ga-gb}, then the
  overdetermined ODE problem
\begin{align} 
&v-\tilde{p}\frac{dv}{dx}
-\tilde{q}^{\ell}\frac{d^{2}v}{dx^2}=g_{a}+\mu g_{b} 
\quad \mbox{ in }(b,c),\label{eq:aux-appen-1}\\
&v(b)=\displaystyle\frac{dv}{dx}(b)=0,
\qquad v(c)=\displaystyle\frac{dv}{dx}(c)=0,\label{eq:aux-appen-3}
\end{align}
is not solvable:
equations~\eqref{eq:aux-appen-1}--\eqref{eq:aux-appen-3} do not admit
solutions $v$ for any real value of the constant $\mu$.
\end{lemma}

\begin{proof}
  Assume a solution $v$ of the problem
  \eqref{eq:aux-appen-1}--\eqref{eq:aux-appen-3} exists. Denoting by
  $G(x,\xi)$ the Green function of the problem,
\begin{align*} 
&G(x,\xi)-\tilde{p}\frac{\partial}{\partial x} G(x,\xi)
-\tilde{q}^{\ell}\frac{\partial^{2}}{\partial x^2}G(x,\xi)=\delta_{(x=\xi)}
\quad \mbox{ in }(b,c),\\
&G(b,\xi)=G(c,\xi)=0,
\end{align*}
and letting
\begin{equation}\label{h_label}
  h=g_{a}+\mu g_{b},
\end{equation}
the solution $v$ can be expressed in the form
$$
v(x)=\int_{b}^{c}G(x,\xi)h(\xi)\,\mathrm{d}\xi.
$$
Taking into account the Neumann boundary
conditions~\eqref{eq:aux-appen-3} we then obtain
\begin{align}
&\frac{dv}{dx}(b)=
\int_{b}^{c}\frac{\partial G}{\partial x}(b,\xi)h(\xi)\,\mathrm{d}\xi=0,
\label{eq:neu-1}\\
&\frac{dv}{dx}(c)=
\int_{b}^{c}\frac{\partial G}{\partial x}(c,\xi)h(\xi)\,\mathrm{d}\xi=0.
\label{eq:neu-2}
\end{align}

Now, as is known (see e.g.  in \cite[Ch. V.28]{Weinberger95}), the
function $\partial G/\partial \xi$ satisfies the ODE problems
\begin{align*} 
  &\frac{\partial G}{\partial\xi}(x,b)-
  \tilde{p}\frac{\partial}{\partial x}\left(\frac{\partial
      G}{\partial\xi}(x,b)\right) -\tilde{q}^{\ell}\frac{\partial^{2}}{\partial x^2}
  \left(\frac{\partial G}{\partial\xi}(x,b)\right)=0
  \quad \mbox{ in }(b,c),\\
  &\frac{\partial
    G}{\partial\xi}(b,b)=\frac{1}{\tilde{q}^{\ell}(b)},\qquad
  \frac{\partial G}{\partial\xi}(c,b)=0
\end{align*}
and
\begin{align*} 
&\frac{\partial G}{\partial\xi}(x,c)-
\tilde{p}\frac{\partial}{\partial x}\left(\frac{\partial G}{\partial\xi}(x,c)\right)
-\tilde{q}^{\ell}\frac{\partial^{2}}{\partial x^2}
\left(\frac{\partial G}{\partial\xi}(x,c)\right)=0
\quad \mbox{ in }(b,c),\\
&\frac{\partial G}{\partial\xi}(b,c)=0,\qquad
\frac{\partial G}{\partial\xi}(c,c)=-\frac{1}{\tilde{q}^{\ell}(c)}.
\end{align*}
In view of the identity $\frac{\partial G}{\partial
  x}(x,\xi)=\frac{\partial G}{\partial\xi}(\xi,x)$ (which follows from
the symmetry $G(x,\xi)=G(\xi,x)$ of the Green function) it follows
that the function
\[
H (x) = \frac{\partial G}{\partial x}(b,x)- \frac{\partial G}{\partial
  x}(c,x)
\]
satisfies the two-point boundary-value problem
\begin{align*} 
  &H- \tilde{p}\frac{dH}{dx}
    -\tilde{q}^{\ell}\frac{d^{2}H}{d x^2}=0
    \quad \mbox{ in }(b,c),\\
    &H(b)=\frac{1}{\tilde{q}^{\ell}(b)},\qquad
      H(c)=\frac{1}{\tilde{q}^{\ell}(c)}.
\end{align*}
Applying the strong maximum principle~\cite{evans98} to this elliptic
equation we obtain the estimate
\[
H (x) = \frac{\partial G}{\partial x}(b,x)-\frac{\partial G}{\partial x}(c,x)
\ge C > 0\qquad\mbox{for }x\in[b,c],
\]
where $C$ is the strictly positive constant
$C=\min\{1/\tilde{q}^{\ell}(b),1/\tilde{q}^{\ell}(c)\}$.
From~\eqref{eq:assump-lemma-ga-gb},~\eqref{h_label},~\eqref{eq:neu-1}
and~\eqref{eq:neu-2} we thus obtain
\begin{equation*}
0=
\int_{b}^{c}\left(\frac{\partial G}{\partial x}(b,x)-
\frac{\partial G}{\partial x}(c,x)\right)h(x)\,\mathrm{d}x
\ge C\int_{b}^{c} (g_{a}(x)+\mu g_{b}(x))\,\mathrm{d}x
 =  C\int_{b}^{c}g_{a}(x)\,\mathrm{d}x>0,
\end{equation*}
which is a contradiction, and the lemma follows.
\end{proof}

\bibliography{biblio} 
\end{document}